\documentclass[11pt]{article}

\usepackage{amsfonts,amsmath,latexsym,color,epsfig}
\newtheorem{theorem}{Theorem}[section]
\newtheorem{lemma}{Lemma}[section]
\newtheorem{proposition}{Proposition}[section]

\newtheorem{corollary}{Corollary}[section]

\newtheorem{claim}{Claim}[section]

\newcommand{\pr}[1]{\mathbb{P}\left[#1\right]}
\newcommand{\E}[1]{\mathbb{E}\left[#1\right]}

\newcommand{\qed}{\hfill \ensuremath{\Box}}

\newenvironment{proof}
      {\medskip\noindent{\bf Proof:}\hspace{1mm}}
      {\hfill$\Box$\medskip}

\input{epsf}

\makeatletter
\def\Ddots{\mathinner{\mkern1mu\raise\p@
\vbox{\kern7\p@\hbox{.}}\mkern2mu
\raise4\p@\hbox{.}\mkern2mu\raise7\p@\hbox{.}\mkern1mu}}
\makeatother

\title{\vspace{-0.7cm} Bounds for graph regularity and removal lemmas}
\author{David Conlon\thanks{St John's College, Cambridge CB2 1TP, United
Kingdom.
E-mail: {\tt
d.conlon@dpmms.cam.ac.uk}. Research supported by a Royal Society University
Research Fellowship.} \and Jacob Fox\thanks{Department of Mathematics, MIT,
Cambridge, MA 02139-4307. E-mail: {\tt fox@math.mit.edu}. Research
supported by a Simons Fellowship and NSF grant DMS-1069197.}}
\date{}
\begin{document}
\maketitle

\begin{abstract}
We show, for any positive integer $k$, that there exists a graph in which any
equitable partition of its vertices into $k$ parts has at least $ck^2/\log^* k$
pairs of parts which are not $\epsilon$-regular, where $c,\epsilon>0$ are
absolute constants. This bound is tight up to the constant $c$ and addresses a
question of Gowers on the number of irregular pairs in Szemer\'edi's regularity
lemma.

In order to gain some control over irregular pairs, another regularity lemma,
known as the strong regularity lemma, was developed by Alon, Fischer,
Krivelevich, and Szegedy. For this lemma, we prove a lower bound of
wowzer-type, which is one level higher in the Ackermann hierarchy than the
tower function, on the number of parts in the strong regularity lemma,
essentially matching the upper bound. On the other hand, for the induced graph
removal lemma, the standard application of the strong regularity lemma, we find
a different proof which yields a tower-type bound.

We also discuss bounds on several related regularity lemmas, including the weak
regularity lemma of Frieze and Kannan and the recently established regular
approximation theorem. In particular, we show that a weak partition with
approximation parameter $\epsilon$ may require as many as
$2^{\Omega(\epsilon^{-2})}$ parts. This is tight up to the implied constant and
solves a problem studied by Lov\'asz and Szegedy.
\end{abstract}

\section{Introduction}

Originally developed by Szemer\'edi as part of his proof of the celebrated
Erd\H{o}s-Tur\'an conjecture on long arithmetic progressions in dense subsets
of the integers \cite{Sz1}, Szemer\'edi's regularity lemma \cite{Sz} has become a central tool in extremal combinatorics. Roughly speaking, the lemma says that the vertex
set of any graph may be partitioned into a small
number of parts such that the bipartite subgraph between almost every pair of
parts behaves in a random-like fashion.

Given two subsets $X$ and $Y$ of a graph $G$, we write $d(X,Y)$ for the density
of edges between $X$ and $Y$. The pair $(X, Y)$ is said to be {\it
$(\epsilon,\delta)$-regular} if for some $\alpha$ and all $X' \subset X$ and
$Y' \subset Y$ with $|X'| \geq \delta |X|$ and $|Y'| \geq \delta |Y|$, we have
$\alpha<d(X',Y')<\alpha+\epsilon$. In the case where $\delta = \epsilon$, we
say that the pair $(X, Y)$ is {\it $\epsilon$-regular}. By saying that a pair
of parts is random-like, we mean that they are $(\epsilon, \delta)$-regular
with $\epsilon$ and $\delta$ small, a property which is easily seen to be
satisfied with high probability by a random bipartite graph. We will also ask that the
different parts be of comparable size, that is, that the partition $V(G) = V_1
\cup \ldots \cup V_k$ be {\it equitable}, that is, $||V_i|-|V_j|| \leq 1$ for
all $i$ and $j$.

The {\it regularity lemma} now states that for each $\epsilon,\delta,\eta > 0$,
there is a positive integer $M=M(\epsilon,\delta,\eta)$ such that the vertices
of any graph $G$ can be equitably partitioned $V(G) = V_1 \cup \ldots \cup V_M$
into $M$ parts where all but at most an $\eta$ fraction of the pairs
$(V_i,V_j)$ are $(\epsilon,\delta)$-regular. We shall say that such a partition
is $(\epsilon, \delta, \eta)$-regular and simply $\epsilon$-regular in the case
$\epsilon=\delta=\eta$. For more background on the regularity lemma, see the
excellent surveys by Koml\'os and Simonovits \cite{KoSi} and R\"odl and Schacht
\cite{RoSc}.

Use of the regularity lemma is now widespread throughout graph theory. However,
one of the earliest applications, the triangle removal lemma of Ruzsa and
Szemer\'edi \cite{RuSz}, remains the standard example. It states that for any
$\epsilon > 0$ there exists $\delta > 0$ such that any graph on $n$ vertices
with at most $\delta n^3$ triangles can be made triangle-free by removing
$\epsilon n^2$ edges. It easily implies Roth's theorem \cite{Ro} on 3-term
arithmetic progressions in dense sets of integers, and Solymosi \cite{So}
showed that it further implies the stronger corners theorem of Ajtai and
Szemer\'edi \cite{AjSz}, which states that any dense subset of the integer grid
contains the vertices of an axis-aligned isosceles triangle. This result was extended to all graphs in \cite{EFR,ADLRY}. The extension, known as the graph removal lemma, says that given a graph $H$ on $h$ vertices and $\epsilon > 0$ there exists $\delta > 0$ such that any graph on $n$ vertices with at most $\delta n^h$ copies of $H$ can be made $H$-free by removing $\epsilon n^2$ edges.

One disadvantage of applying the regularity lemma to prove this theorem is the bounds that it gives for the size of $\delta$ in terms of $\epsilon$. The proof of the regularity lemma yields a bound of tower-type for the number of pieces in the partition. When this is applied to graph removal, it gives a bound for $\delta^{-1}$ which is a tower
of twos of height polynomial in $\epsilon^{-1}$. Surprisingly, any hope that a
better bound for the regularity lemma might be found was put to rest by Gowers
\cite{Go}, who showed that there are graphs for which a tower-type number of
parts are required in order to obtain a regular partition.

To be more precise, the proof of the regularity lemma shows that
$M(\epsilon,\delta,\eta)$ can be taken to be a tower of twos of height
proportional to $\epsilon^{-2}\delta^{-2}\eta^{-1}$. Gowers' result, described
in \cite{Bo} as a tour de force, is a lower bound, with $c = 1/16$, on
$M(1-\delta^c,\delta,1-\delta^c)$ which is a tower of twos of height
proportional to $\delta^{-c}$. As Gowers notes, it is an easy exercise to
translate lower bounds for small $\delta$ and large $\epsilon$ into lower
bounds for large $\delta$ and small $\epsilon$ which are also of tower-type.
However, the natural question, discussed by Szemer\'edi \cite{Sz}, Koml\'os and Simonovits \cite{KoSi}, and Gowers \cite{Go}, of determining the dependency of $M(\epsilon,
\delta, \eta)$ on $\eta$, which measures the fraction of allowed irregular pairs, has remained open. This is the first problem we will address here, showing that the dependence is again of tower-type.

This does not mean that better bounds cannot be proved for the graph removal
lemma. Recently, an alternative proof was found by the second author \cite{Fo},
allowing one to show that $\delta^{-1}$ may be taken to be a tower of twos of
height $O(\log \epsilon^{-1})$, better than one could possibly do using
regularity. Though this remains quite far from the lower bound of $\epsilon^{-
O(\log \epsilon^{-1})}$, it clears a significant hurdle.

The second major theme of this paper is a proof of the induced graph removal
lemma which similarly bypasses a natural obstacle. Let $H$ be a graph on $h$
vertices. The induced graph removal lemma, proved by Alon, Fischer,
Krivelevich, and Szegedy \cite{AFKS}, states that for any $\epsilon > 0$ there
exists $\delta > 0$ such that any graph on $n$ vertices with at most $\delta
n^h$ induced copies of $H$ may be made induced $H$-free by adding or deleting
at most $\epsilon n^2$ edges.

This result, which easily implies the graph removal lemma, does not readily
follow from the same technique used to prove the graph removal lemma, mainly
because of the possibility of irregular pairs in the regularity partition. To
overcome this issue, Alon, Fischer, Krivelevich, and Szegedy \cite{AFKS}
developed a strenghthening of Szemer\'edi's regularity lemma. Roughly, it gives
an equitable partition $\mathcal{A}$ and an equitable refinement $\mathcal{B}$
of $\mathcal{A}$ such that $\mathcal{A}$ and $\mathcal{B}$ are both regular, with the guaranteed regularity of $\mathcal{B}$ allowed to depend on the size of $\mathcal{A}$, and the edge density between almost all pairs of parts in $\mathcal{B}$ close to the edge
density between the pair of parts in $\mathcal{A}$ that they lie in.

The proof of the strong regularity lemma involves iterative applications of
Szemer\'edi's regularity lemma. This causes the upper bound on the number of
parts in $\mathcal{B}$ to grow as a wowzer function, which is one level higher
in the Ackermann hierarchy than the tower function. In order to get an improved
bound for its various applications, one may hope that an improved bound of
tower-type could be established. We show that no such bound exists. In fact, we
will show that a seemingly weaker statement requires wowzer-type bounds. On the
other hand, we give an alternative proof of the induced graph removal lemma,
allowing one to show that $\delta^{-1}$ may be taken to be a tower of twos of
height polynomial in $\epsilon^{-1}$, better than one could possibly achieve
using the strong regularity lemma.

We also make progress on determining bounds for various related regularity
lemmas, including the Frieze-Kannan weak regularity lemma \cite{FrKa, FrKa1}
and the regular approximation theorem, due independently to Lov\'asz and
Szegedy \cite{LS} and to R\"odl and Schacht \cite{RS07}. We discuss all these
contributions in more detail in the sections below.

\subsection{The number of irregular pairs}\label{irregularpairssection}

The role of $\eta$ in the regularity lemma is to measure how many pairs of
subsets in the partition are regular. If a partition into $k$ pieces is
$(\epsilon, \delta, \eta)$-regular, then there will be at most $\eta
\binom{k}{2}$ irregular pairs in the partition. Szemer\'edi \cite{Sz} wrote
that it would be interesting to determine if the assertion of the regularity
lemma holds when we do not allow any irregular pairs. This question remained
unanswered for a long time until it was observed by Lov\'asz, Seymour, Trotter,
and Alon, Duke, Lefmann, R\"odl, and Yuster \cite{ADLRY} that irregular pairs
can be necessary. The simple example of the half-graph shows that this is
indeed the case. The half-graph is a bipartite graph with vertex sets
$A=\{a_1,\ldots,a_n\}$ and $B=\{b_1,\ldots,b_n\}$ in which $(a_i,b_j)$ is an
edge if and only if $i \leq j$. Any partition of this graph into $M$ parts will
have $\Omega(M)$ irregular pairs. In other words, $M(\epsilon,\delta,\eta)$
must grow at least linearly in $\eta^{-1}$.

However, the number of irregular pairs, or, in other words, the dependence of
$M(\epsilon,\delta,\eta)$ on $\eta^{-1}$ with $\epsilon$ and $\delta$ fixed,
has not been well understood despite being asked several times, including by
Koml\'os and Simonovits \cite{KoSi} and, more explicitly, by Gowers \cite{Go}.
This problem and related problems have continued to attract interest (see,
e.g., \cite{KoRo}, \cite{MSh}). The linear bound obtained from the half-graph
appears to be the only bound in the literature for this problem.

For fixed constants $\epsilon$ and $\delta$ and each $M$, we give a
construction in which any partition into $M$ parts has at least $cM^2/\log^* M$
irregular parts, where $c>0$ is an absolute constant, and this is tight apart
from the constant $c$. The iterated logarithm $\log^* n$ is the number of times
the logarithm function needs to be applied to get a number which is at most
$1$. That is, $\log^* x = 0$ if $x \leq 1$ and otherwise $\log^* x =1+\log^*
(\log x)$ denotes the iterated logarithm. In other words, the dependence on
$\eta$ in $M(\epsilon,\delta,\eta)$ is indeed a tower of twos of height
proportional to $\Theta(\eta^{-1})$.

\begin{theorem}\label{exceptionalpairs}
There are absolute constants $c,\epsilon,\delta>0$ such that for every $k$
there is a graph in which every equitable partition of the graph into $k$ parts
has at least $ck^2/\log^* k$ pairs of parts which are not
$(\epsilon,\delta)$-regular. In other words, $M(\epsilon,\delta,\eta)$ is at
least a tower of twos of height $c\eta^{-1}$.
\end{theorem}
We prove Theorem \ref{exceptionalpairs} with $\epsilon=\frac{1}{2}$,
$\delta=2^{-500}$, and $c=2^{-700}$, and we make no attempt to optimize
constants. The proof of Theorem \ref{exceptionalpairs} can be easily modified
to obtain the same result with $\epsilon$ tending to $1$ at the expense of
having $\delta$ and $c$ tending to $0$.

In the important special case where $\epsilon = \delta = \eta$, we let
$M(\epsilon)= M(\epsilon, \delta, \eta)$. Gowers \cite{Go} gave two different
constructions giving lower bounds on $M(\epsilon)$. The first construction is
simpler, but the lower bound it gives is a tower of twos of height only
logarithmic in $\epsilon^{-1}$. The second construction gives a lower bound
which is a tower of twos of height $\epsilon^{-1/16}$, but is more complicated.
Theorem \ref{exceptionalpairs} also gives a lower bound on $M(\epsilon)$ which
is a tower of twos of height polynomial in $\epsilon^{-1}$, in fact linear in
$\epsilon^{-1}$, and the construction is a bit simpler. Unfortunately, the
proof that it works, which builds upon Gowers' simpler first proof, is still
rather complicated and delicate.

We give a rough idea of how the graph $G$ used to prove Theorem
\ref{exceptionalpairs} is constructed. The graph $G$ has a sequence of vertex
partitions $P_1,\ldots,P_s$, with $P_{i+1}$ a refinement of $P_i$ for $1 \leq i
\leq s-1$, and the number of parts of $P_{i+1}$ is roughly exponential in the
number of parts of $P_i$. For each $i$, $1 \leq i \leq s-1$, we pick a random
graph $G_i$ with vertex set $P_i$, where each edge is picked independently with probability $p_i$. For every two vertex subsets $X,Y \in P_i$ of
$G$ which are adjacent in $G_i$, we take random vertex partitions $X=X_Y^1 \cup
X_Y^2$ and $Y=Y_X^1 \cup Y_X^2$ into parts of equal size, with each of these
parts the union of parts of $P_{i+1}$. Then, for $d=1,2$, we add the edges to
$G$ between $X_Y^d$ and $Y_X^d$. We will show that with positive probability
the graph $G$ constructed above has the desired properties for Theorem
\ref{exceptionalpairs}. In fact, in Theorem \ref{maingen}, we will show that it
has the stronger property that any $(\epsilon,\delta,\eta)$-regular vertex
partition of $G$ is close to being a refinement of $P_s$.

A novelty of our construction, not present in the constructions of Gowers, is
the use of the random graphs $G_i$, which allow us to control the number of
irregular pairs. Instead, for every pair of parts $X,Y$ in $P_i$, Gowers
\cite{Go} introduces or deletes some edges between them so as to make the pair
of parts far from regular. To prove the desired result, we first establish
several lemmas on the edge distribution in $G$. The construction is general
enough and Theorem \ref{maingen} strong enough that we also use it to establish
a wowzer-type lower bound for the strong regularity lemma, as described in the
next subsection.

\subsection{The strong regularity lemma}\label{strongsubsection}

Before stating the strong regularity lemma, we next define a notion of
closeness between an equitable partition and an equitable refinement of this
partition.  For an equitable partition $\mathcal{A} = \{V_i|1 \leq i \leq k\}$
of $V(G)$ and an equitable refinement $\mathcal{B} = \{V_{i,j}|1 \leq i \leq
k,1 \leq j \leq \ell\}$ of $\mathcal{A}$, we say that $\mathcal{B}$ is {\it
$\epsilon$-close} to $\mathcal{A}$ if the following is satisfied. All $1\leq i
\leq i' \leq k$ but at most $\epsilon k^2$ of them are such that for all $1\leq
j,j' \leq \ell$ but at most $\epsilon \ell^2$ of them
$|d(V_i,V_{i'})-d(V_{i,j},V_{i',j'})|<\epsilon$ holds. This notion roughly says
that $\mathcal{B}$ is an approximation of $\mathcal{A}$. We are now ready to state the strong regularity lemma of  Alon, Fischer, Krivelevich, and Szegedy \cite{AFKS}.

\begin{lemma}\label{strongreg} {\bf (Strong regularity lemma)} For every
function $f:\mathbb{N} \rightarrow (0,1)$ there exists a number $S = S(f)$ with
the following property. For every graph $G=(V,E)$, there is an equitable
partition $\mathcal{A}$ of the vertex set $V$ and an equitable refinement
$\mathcal{B}$ of $\mathcal{A}$ with $|\mathcal{B}| \leq S$ such that the
partition $\mathcal{A}$ is $f(1)$-regular, the partition $\mathcal{B}$ is
$f(|\mathcal{A}|)$-regular, and $\mathcal{B}$ is $f(1)$-close to $\mathcal{A}$.
\end{lemma}

The upper bound on $S$, the number of parts of $\mathcal{B}$, that the proof
gives is of wowzer-type, which is one level  higher in the Ackermann hierarchy
than the tower function. The {\it tower function} is defined by $T(1)=2$, and
$T(n)=2^{T(n-1)}$ for $n \geq 2$. The {\it wowzer function} $W(n)$ is defined
by $W(1)=2$ and $W(n)=T(W(n-1))$. For reasonable choices of the function $f$,
which is the case for all known applications, such as those for which $1/f$ is
an increasing function which is at least a constant number of iterations of the
logarithm function, the upper bound on $S(f)$ is at least wowzer in a power of
$\epsilon=f(1)$. Recall that $M(\epsilon)$, the number of parts required for
Szemer\'edi's regularity lemma, grows as a tower of height a power of
$\epsilon^{-1}$. The precise upper bound on the number of parts in the strong
regularity lemma is defined as follows. Let $W_1=M(\epsilon)$ and
$W_{i+1}=M(2f(W_i)/W_i^2)$.
The proof of the strong regularity lemma \cite{AFKS} shows that
$S(f)=512\epsilon^{-4}W_j$ with $j= 64\epsilon^{-4}$ satisfies the required
property.

For a partition $\mathcal{P}:V(G)=V_1 \cup \ldots \cup V_k$ of the vertex set
of a graph $G$, the {\it mean square density} of $\mathcal{P}$ is defined by
$$q(\mathcal{P})=\sum_{i,j}d^2(V_i,V_j)p_ip_j,$$ where $p_i=|V_i|/|V(G)|$. This
function plays an important role in the proof of Szemer\'edi's regularity lemma
and its variants.

The strong regularity lemma gives a regular partition $\mathcal{A}$, and a
refinement $\mathcal{B}$ which is much more regular and is close to
$\mathcal{A}$. For equitable partitions $\mathcal{A}$ and $\mathcal{B}$ with
$\mathcal{B}$ a refinement of $\mathcal{A}$, the condition $\mathcal{B}$ is
$\epsilon$-close to $\mathcal{A}$ is equivalent,  up to a polynomial change in
$\epsilon$, to $q(\mathcal{B}) \leq q(\mathcal{A})+\epsilon$. Indeed, if $\mathcal{B}$ is $\epsilon$-close to $\mathcal{A}$, then $q(\mathcal{B}) \leq q(\mathcal{A}) + O(\epsilon)$, while if $q(\mathcal{B}) \leq q(\mathcal{A})+\epsilon$, then $\mathcal{B}$ is $O(\epsilon^{1/4})$-close to $\mathcal{A}$. A version of this statement is present in Lemma 3.7 of \cite{AFKS}. As it is suffiicent and more convenient to work with mean square density instead of $\epsilon$-closeness, we do so from now on. 

Note that in the strong regularity lemma, without loss of generality we may
assume $f$ is a (monotonically) decreasing function. Indeed, this can be shown
by considering the decreasing function $f'(k):=\min_{1 \leq i \leq k} f(k)$.
From the above discussion, it is easy to see that the strong regularity lemma
has the following simple corollary, with a similar upper bound.

\begin{corollary} \label{strongcor}
Let $\epsilon>0$ and $f:\mathbb{N} \rightarrow (0,1)$ be a decreasing function.
Then there exists a number $S=S(f,\epsilon)$ such that for every graph $G$
there are equitable partitions $\mathcal{A},\mathcal{B}$ of the vertex set of
$G$ with $|\mathcal{B}| \leq S$, $q(\mathcal{B}) \leq q(\mathcal{A})+\epsilon$,
and $\mathcal{B}$ is $f(|\mathcal{A}|)$-regular.
\end{corollary}

We prove a lower bound for the strong regularity lemma of wowzer-type, which
essentially matches the upper bound. Maybe surprisingly, our construction
further shows that much less than what is required from the strong regularity
lemma already gives wowzer-type bounds. In particular, even for Corollary
\ref{strongcor}, which appears considerably weaker than the strong regularity
lemma, we get a wowzer-type lower bound. Note that in Corollary
\ref{strongcor}, $\mathcal{B}$ is not required to be a refinement of
$\mathcal{A}$. In this case we could have $q(\mathcal{B})$ being close to
$q(\mathcal{A})$ but the edge densities between the parts in these partitions
are quite different from each other, i.e., these partitions are not close to
each other.

\begin{theorem} \label{stronglow}
Let $0<\epsilon<2^{-100}$ and $f:\mathbb{N} \rightarrow (0,1)$ be a decreasing function with $f(1) \leq 2^{-100}\epsilon^{6}$. Define
$W_{\ell}$ recursively by $W_1=1$, $W_{\ell+1} =
T\left(2^{-70}\epsilon^5/f(W_{\ell})\right)$, where $T$ is the tower function.
Let $W=W_{t-1}$ with $t=2^{-20}\epsilon^{-1}$. Then there is a graph $G$ such
that if equitable partitions $\mathcal{A},\mathcal{B}$ of the vertex set of $G$
satisfy $q(\mathcal{B}) \leq q(\mathcal{A})+\epsilon$ and $\mathcal{B}$ is
$f(|\mathcal{A}|)$-regular, then $|\mathcal{A}|,|\mathcal{B}| \geq W$.
\end{theorem}

We have the following corollary (by replacing $\epsilon$ by $\epsilon^{1/7}$),
which is a simple to state lower bound of wowzer-type.

\begin{corollary} \label{stroncor}
For $0<\epsilon<2^{-700}$, there is a graph $G$ such that if equitable
partitions $\mathcal{A},\mathcal{B}$
of the vertex set of $G$ satisfy $|\mathcal{B}| \geq |\mathcal{A}|$,
$q(\mathcal{B}) \leq q(\mathcal{A})+\epsilon$ and $\mathcal{B}$ is $\epsilon/|A|$-regular,
then $|\mathcal{B}|, |\mathcal{A}|$ are bounded below by a function which is
wowzer in $\Omega(\epsilon^{-1/7})$.
\end{corollary}

\subsection{Induced graph removal} \label{removalsubsection}

Let $H$ be a fixed graph on $h$ vertices and let $G$ be a graph with $o(n^h)$ copies of $H$. To prove the graph removal lemma, we need to prove that all copies of $H$ can be removed from $G$ by deleting $o(n^2)$ edges. The standard approach is to apply the regularity lemma to the graph $G$ to obtain an $\epsilon$-regular vertex partition (with an appropriate $\epsilon$) into a constant number of parts $M(\epsilon)$. Then delete edges between pairs of parts
$(V_i,V_j)$, including $i=j$, if the pair is not $\epsilon$-regular or the
density between the pair is small. It is easy to see that there are few deleted
edges. Furthermore, if there is a copy of $H$ in the remaining subgraph, then
the edges go between pairs of parts which are $\epsilon$-regular and not of
small density. A counting lemma then shows that in such a case the number of
copies of $H$ is $\Omega(n^h)$ in the remaining subgraph, and hence in $G$ as
well. But this would contradict the assumption that $G$ has $o(n^h)$ copies of $H$, so all copies of $H$ must already have been removed.

Recall that the induced graph removal lemma \cite{AFKS} is the analogous
statement for induced subgraphs, and it is stronger than the graph removal
lemma. It states that for any graph $H$ on $h$ vertices and $\epsilon>0$ there
is $\delta=\delta(\epsilon,H)>0$ such that if a graph $G$ on $n$ vertices has
at most $\delta n^h$ induced copies of $H$, then we can add or delete $\epsilon
n^2$ edges of $G$ to obtain an induced $H$-free graph.

One well-known application of the induced graph removal lemma is in property
testing. This is an active area of computer science where one wishes to quickly
distinguish between objects that satisfy a property from objects that are far
from satisfying that property. The study of this notion was initiated by
Rubinfield and Sudan \cite{RuSu}, and subsequently Goldreich, Goldwasser, and
Ron \cite{GGR} started the investigation of property testers for combinatorial
objects. One simple consequence of the induced graph removal lemma is a
constant time algorithm for induced subgraph testing with one-sided error (see
[2] and its references). A graph on $n$ vertices is $\epsilon$-far from being
induced $H$-free if at least $\epsilon n^2$ edges need to be added or removed
to make it induced $H$-free. The induced graph removal lemma implies that there
is an algorithm which runs in time $O_{\epsilon}(1)$ which accepts all induced
$H$-free graphs, and rejects any graph which is $\epsilon$-far from being
induced $H$-free with probability at least 2/3. The algorithm samples $t =
2\delta^{-1}$ $h$-tuples of vertices uniformly at random, where $\delta$ is
picked according to the induced graph removal lemma, and accepts if none of
them form an induced copy of $H$, and otherwise rejects. Any induced $H$-free
graph is clearly accepted. If a graph is $\epsilon$-far from being induced
$H$-free, then it contains at least $\delta n^h$ copies of $H$, and the
probability that none of the sampled $h$-tuples forms an induced copy of $H$ is
at most $(1 -\delta)^t < 1/3$. Notice that the running time as a function of
$\epsilon$ depends on the bound in the induced graph removal lemma, and the
proof using the strong regularity lemma gives a wowzer-type dependence.

It is tempting to try the same approach using Szemer\'edi's regularity lemma to
obtain the induced graph removal lemma.  However, there is a significant
problem with this approach, which is handling the pairs between irregular
pairs. To get around this issue,  Alon, Fischer, Krivelevich, and Szegedy
\cite{AFKS} developed the strong regularity lemma.

Because of its applications, including those in graph property testing, it has
remained an intriguing problem to improve the bound in the induced graph
removal lemma. This problem has been discussed in several papers by Alon and
his collaborators \cite{A02}, \cite{AFN}, \cite{AlSh06}. The main result
discussed in this subsection addresses this problem, improving the bound on the
number of parts in the induced graph removal lemma from wowzer-type to
tower-type. The {\it tower function} $t_i(x)$ is defined by $t_0(x)=x$ and
$t_{i+1}(x)=2^{t_i(x)}$. We say that $t_i(x)$ is a tower in $x$ of height $i$. 

\begin{theorem}\label{inducedtower}
For any graph $H$ on $h$ vertices and $0 < \epsilon< 1/2$ there is $\delta>0$ with $\delta^{-1}$ a tower in $h$ of height polynomial in $\epsilon^{-1}$ such that if a graph $G$ on $n$ vertices has
at most $\delta n^h$ induced copies of $H$, then we can add or delete $\epsilon
n^2$ edges of $G$ to obtain an induced $H$-free graph.
\end{theorem}

The following lemma is an easy corollary of the strong regularity lemma which
was used in \cite{AFKS} to establish the induced graph removal lemma.

\begin{lemma}\label{strongeasycor}
For each $0 < \epsilon < 1/3$ and decreasing function $f:\mathbb{N}\rightarrow
(0,1/3)$ there is $\delta'=\delta'(\epsilon,f)$ such that every graph $G=(V,E)$
with $|V| \geq \delta'^{-1}$ has an equitable partition $V=V_1 \cup \ldots \cup V_k$ and vertex subsets $W_i
\subset V_i$ such that $|W_i| \geq \delta' |V|$, each pair $(W_i,W_j)$ with
$1 \leq i \leq j \leq k$ is $f(k)$-regular, and all but at most $\epsilon k^2$
pairs $1 \leq i \leq j \leq k$ satisfy $|d(V_i,V_j)-d(W_i,W_j)| \leq \epsilon$.
\end{lemma}

In fact, Lemma \ref{strongeasycor} is a little bit stronger than the original
version in \cite{AFKS} in that each set $W_i$ is $f(k)$-regular with itself.
The original version follows from the strong regularity lemma by taking the
partition $V=V_1 \cup \ldots \cup V_k$ to be the partiton $\mathcal{A}$ in the
strong regularity lemma, and the subset $W_i$ to be a
random part $V_{ij} \subset V_i$ of the refinement $\mathcal{B}$ of
$\mathcal{A}$ in the strong regularity lemma.

From this slightly stronger version, the proof of the induced graph removal
lemma is a bit simpler and shorter.  Indeed, with $f(k)=\frac{1}{4h}\epsilon^h$, which
does not depend on $k$, if there is a mapping
$\phi:V(H) \rightarrow \{1,\ldots,k\}$ such that for all adjacent vertices
$v,w$ of $H$, the edge density between $W_{\phi(v)}$ and $W_{\phi(w)}$ is at
least $\epsilon$, and for all distinct nonadjacent vertices $v,w$ of $H$, the
edge density between $W_{\phi(v)}$ and $W_{\phi(w)}$ is at most $1-\epsilon$,
then a standard counting lemma shows that $G$ contains at least $\delta n^h$
induced copies of $H$, where $\delta=(\epsilon/4)^{{h \choose 2}}\delta'^h$. Hence, we may
assume that there is no such mapping $\phi$.
We then delete edges between $V_i$ and $V_j$ if the edge density between $W_i$
and $W_j$ is less than $\epsilon$, and one adds the edges between $V_i$ and
$V_j$ if the density between $W_i$ and $W_j$ is more than $1-\epsilon$. The
total number of edges added or removed is at most $5\epsilon n^2$, and no
induced copy of $H$ remains. Replacing $\epsilon$ by $\epsilon/8$ in the above argument gives the induced graph removal lemma. 

We find another proof of Lemma \ref{strongeasycor} with a better tower-type
bound. This in turn implies, by the argument sketched above, the tower-type bound for the induced graph removal lemma stated in Theorem \ref{inducedtower}. 

The starting point for our approach to Lemma \ref{strongeasycor} is a weak regularity lemma due to Duke, Lefmann and
R\"odl \cite{DLR}. This lemma says that for a $k$-partite graph, between sets
$V_1, V_2, \dots, V_k$, there is an $\epsilon$-regular partition of the
cylinder $V_1 \times \dots \times V_k$ into a relatively small number of cylinders $K = W_1 \times \dots
\times W_k$, with $W_i \subset V_i$ for $1 \leq i \leq k$. The definition of an
{\it $\epsilon$-regular} partition here is that all but an $\epsilon$-fraction
of the $k$-tuples $(v_1, \dots, v_k) \in V_1 \times \dots \times V_k$ are in
{\it $\epsilon$-regular} cylinders, where a cylinder $W_1 \times \dots \times
W_k$ is $\epsilon$-regular if all $\binom{k}{2}$ pairs $(W_i, W_j)$, $1 \leq i
< j \leq k$, are $\epsilon$-regular in the usual sense.

In the same way that one derives the strong regularity lemma from the ordinary
regularity lemma, we show how to derive a strong version of this lemma. We will
refer to this strengthening, of which Lemma \ref{strongeasycor} is a
straightforward consequence, as the  {\it strong cylinder regularity lemma}. It
will also be convenient if, in this lemma, we make the requirement that a
cylinder be regular slightly stronger, by asking that each $W_i$ be regular
with itself. That is, we say that a cylinder $W_1 \times \dots \times W_k$ is
{\it strongly $\epsilon$-regular} if all pairs  $(W_i,W_j)$ with $1 \leq i,j
\leq k$ are $\epsilon$-regular. A partition $\mathcal{K}$ of the cylinder $V_1
\times \dots \times V_k$ into cylinders $K = W_1 \times \dots \times W_k$, with
$W_i \subset V_i$ for $1 \leq i \leq k$, is then said to {\it strongly
$\epsilon$-regular} if all but an $\epsilon$-fraction of the $k$-tuples
$(v_1,\ldots,v_k) \in V_1 \times \cdots \times V_k$ are in strongly
$\epsilon$-regular cylinders.

Let $P : V = V_1 \cup \dots \cup V_k$ be a partition of the vertex set of a
graph and $\mathcal{K}$ be a partition of the cylinder $V_1 \times \dots \times
V_k$ into cylinders. For each $K = W_1 \times \dots \times W_k$, with $W_i
\subset V_i$ for $1 \leq i \leq k$, we let $V_i(K) = W_i$. We then define the
partition  $Q(\mathcal{K})$ of $V$ to be the refinement of $P$ which is the
common refinement of all the parts $V_i(K)$ with $i \in [k]$ and $K \in
\mathcal{K}$. The strong cylinder regularity lemma is now as
follows.

\begin{lemma}\label{scrl} For $0<\epsilon<1/3$, positive integer $s$, and decreasing function $f:\mathbb{N} \rightarrow
(0,\epsilon]$, there is $S=S(\epsilon,s,f)$ such that the following holds. For every
graph $G$, there is an integer $s \leq k \leq S$, an equitable  partition $P:V=V_1 \cup
\ldots \cup V_k$ and a strongly $f(k)$-regular partition $\mathcal{K}$ of the
cylinder $V_1 \times \cdots \times V_k$ into cylinders satisfying that the
partition $Q=Q(\mathcal{K})$ of $V$ has at most $S$ parts and $q(Q) \leq
q(P)+\epsilon$. Furthermore, there is an absolute constant $c$ such that
letting $s_1=s$ and $s_{i+1}=t_4\left(\left(s_i/f(s_i)\right)^c\right)$, we may
take $S=s_{\ell}$ with $\ell=2\epsilon^{-1}+1$.

\end{lemma}

In order to prove this lemma, we need, in addition to the Duke-Lefmann-R\"odl
regularity lemma, a lemma showing that for each $\epsilon > 0$ there is $\delta
> 0$ such that every graph $G = (V, E)$ contains a vertex subset $U$ with $|U|
\geq \delta |V|$ which is $\epsilon$-regular with itself, where, crucially,
$\delta^{-1}$ is bounded above by a tower function of $\epsilon^{-1}$ of
absolute constant height. While seemingly standard, we do not know of such a result in the literature. 

Lemma \ref{strongeasycor} follows from Lemma \ref{scrl} by considering a random cylinder $K$ in the cylinder partition $\mathcal{K}$, with each cylinder picked with probability proportional to its size, and letting $W_i=V_i(K)$. 

\subsection{Frieze-Kannan weak regularity lemma}

Frieze and Kannan \cite{FrKa}, \cite{FrKa1} developed a weaker notion of
regularity which is sufficient for certain applications and for which the
dependence on the approximation $\epsilon$ is much better. It states the
existence of a vertex partition into a small number of parts for which the
number of edges across any two vertex subsets is within $\epsilon n^2$ of what
is expected based on the edge densities between the parts of the partition and
the intersection sizes of the vertex subsets with these parts.

\begin{lemma}{\bf (Frieze-Kannan weak regularity lemma)}\label{FKWRL}
For each $\epsilon>0$ there is a positive integer $k(\epsilon)$ such that every
graph $G=(V,E)$ has an equitable vertex partition $V=V_1 \cup \ldots \cup V_k$ with $k
\leq k(\epsilon)$ satisfying that for all subsets $A,B \in V$, we have $$\big
|e(A,B)-\sum_{1 \leq i, j \leq k}
d(V_i,V_j)|A \cap V_i||B \cap V_j|\big | \leq \epsilon |V|^2.$$
\end{lemma}

The weak regularity lemma has a number of algorithmic applications. Frieze and
Kannan \cite{FrKa1} used the weak regularity lemma to give constant-time
approximation algorithms for some general problems in dense graphs, a special
case being the Max-Cut of a graph. Recently, Bansal and Williams \cite{BaWi}
used the weak regularity lemma to obtain a faster combinatorial algorithm for
Boolean matrix multiplication. The importance of the weak regularity lemma is
further discussed in the citation of the recent Knuth Prize to Kannan.

As there are several applications of the weak regularity lemma to fundamental
algorithmic problems, we would like to know the correct bounds on the number of
parts for the weak regularity lemma. The proof of the weak regularity lemma
\cite{FrKa1} shows that we may take $k(\epsilon)=2^{O(\epsilon^{-2})}$. If this
upper bound could be improved, it would lead to faster
algorithms for several problems of interest. Lov\'asz and Szegedy
\cite{LS} studied the problem of estimating the minimum number of parts
$k(\epsilon)$ required for the weak regularity lemma, proving a lower bound on
$k(\epsilon)$ of the form $2^{\Omega(\epsilon^{-1})}$. Here we close the gap by
proving a new lower bound which matches the upper bound.

\begin{theorem}\label{lbweakmain}
For each $\epsilon>0$, there are graphs for which the minimum number of parts
in a weak regular partition with approximation $\epsilon$ is
$2^{\Omega(\epsilon^{-2})}$.
\end{theorem}

A careful analysis of the proof of Theorem \ref{lbweakmain} shows that the
number of parts required in the weak regularity lemma with approximation
$\epsilon$ is at least $2^{-2^{-60}\epsilon^{-2}}$ for $0<\epsilon \leq 2^{-50}$. In fact, the theorem yields a stronger result, since we do not here require that the partition be equitable.

While the number of parts in the weak regularity lemma is
$2^{\Theta(\epsilon^{-2})}$, the proof obtains the partition as an overlay of
only $O(\epsilon^{-2})$ sets. As discussed in \cite{LS}, in some applications,
such as in \cite{AFKK}, this can be treated as if there were only about
$O(\epsilon^{-2})$ classes, which makes the weak regularity lemma quite
efficient. It was shown in \cite{AFKK}, and is also implied by Theorem \ref{lbweakmain}, that the partition cannot be the overlay of fewer sets.

\subsection{The regular approximation lemma}\label{LSsubsection}

Another strengthening of Szemer\'edi's regularity lemma came from the study of
graph limits by Lov\'asz and Szegedy \cite{LS}, and also from work on the
hypergraph generalization of the regularity lemma by R\"odl and Schacht
\cite{RS07}. This regularity lemma, known as the regular approximation lemma
\cite{RoSc}, provides an arbitrary precision for the regularity as a function
of the number of parts of the partition if an $\epsilon$-fraction of the edges
are allowed to be added or removed.

For a function $g:\mathbb{N} \rightarrow (0,1)$, a partition of the vertex set
into $k$ parts is {\it $g$-regular} if all pairs of distinct parts in the
partition are $g(k)$-regular.

\begin{lemma}\label{LSlem} ({\bf Regular approximation lemma})
For every $\epsilon > 0$, positive integer $s$ and decreasing function $g : \mathbb{N} \rightarrow
(0,1)$, there is an integer $T = T(g,\epsilon,s)$ so that given a graph $G$ with
$n$ vertices, one can add-to/remove-from $G$ at most $\epsilon n^2$ edges and
thus get a graph $G'$ that has a $g$-regular equitable partition of order $k$ for some $s \leq k
\leq T$.
\end{lemma}

Lov\'asz and Szegedy \cite{LS} state that the regular approximation lemma is
equivalent to the strong regularity lemma, Lemma \ref{strongreg}. It is not
difficult to deduce Lemma \ref{LSlem} from the strong regularity lemma, see
\cite{AlShSt} or \cite{RoSc} for details. Unlike the original graph limit
approach, this proof of the regular approximation lemma gives explicit bounds
and yields a polynomial time algorithm for finding the partition and the
necessary edge modifications. In the other direction, by applying Lemma
\ref{LSlem} with $1/g$ a tower in the $1/f$ from Lemma \ref{strongreg}, letting
$\mathcal{A}$ be the $g$-regular partition of $G'$, and then using
Szemer\'edi's regularity lemma to get a refinement $\mathcal{B}$ of
$\mathcal{A}$ which is an $f(\mathcal{A})$-regular partition of $G$, it is easy to deduce
the strong regularity lemma.

The major caveat here is the additional use of Szemer\'edi's regularity lemma
in deducing the strong regularity lemma from the regular approximation lemma.
Due to the additional use of Szemer\'edi's regularity lemma, it does not rule
out the possibility that the wowzer-type upper bound on $T$ in the regular
approximation lemma can be improved to tower-type. Maybe surprisingly, we
indeed make such an improvement.

\begin{theorem}\label{lsnewbound}
For $\epsilon>0$, positive integer $s$ and a decreasing function $g:\mathbb{N} \rightarrow (0,1)$, let $\delta(t)=\min(\frac{g(t)^3}{32t^2},\epsilon/2)$. Let $t_1=s$ and for
$i\geq 1$ let $t_{i+1}=t_ik(\delta(t_i))$, where $k$ is as in the weak
regularity lemma, so $k(\alpha)=2^{O(\alpha^{-2})}$. Let $T_0=t_j$ with
$j=4\epsilon^{-2}$. Then the regular approximation lemma holds with
$T=16T_0/\delta(T_0)^2$.  In other words, the regular approximation lemma holds
with a tower-type bound.
\end{theorem}

It is usually the case that $1/g(t)$ in the regular approximation lemma is at
most a tower of constant height in $\epsilon^{-1}$ and $t$, and in this case
the upper bound $T$ on the number of parts is only a tower of height polynomial
in $\epsilon^{-1}$. Only in the unusual case of $1/g$ being of tower-type
growth does the number of parts needed in the regular approximation lemma grow
as wowzer-type.

Alon, Shapira, and Stav \cite{AlShSt} give a proof of the regular approximation
lemma which yields a polynomial time algorithm for finding the partition and
the necessary edge modifications. Similarly, our new proof
can be made algorithmic with  a polynomial time algorithm for finding the
partition and the necessary edge modifications. Making the proof algorithmic is
essentially the same as done in \cite{AlShSt}, so we do not include the
details.

A partition of a graph satisfying the weak regularity lemma, Lemma \ref{FKWRL},
is called a weak $\epsilon$-regular partition. Tao showed \cite{Ta1} (see also
\cite{RoSc}), by iterating the weak regularity lemma, that one obtains the
following regularity lemma which easily implies Szemer\'edi's regularity lemma
with the usual tower-type bounds.

\begin{lemma}\label{Taoregularity}
For all $\epsilon>0$, positive integers $s$ and functions $\delta: \mathbb{N} \rightarrow (0,1)$,
there is a $T_0$ such that every graph has an equitable vertex partition $P$
into $t \geq s$ parts which is weak $\epsilon$-regular, an equitable vertex refinement
$Q$ into at most $T_0$ parts which is weak $\delta(t)$-regular, and $q(Q) \leq
q(P)+\epsilon$.
\end{lemma}

Let $t_1=s$, and for $i \geq 1$, let $t_{i+1}=t_ik(\delta(t_i))$, where $k$ is
as in the weak regularity lemma. Recall $k(\epsilon)$ is exponential in
$\epsilon^{-2}$. Then $T_0$ in Lemma \ref{Taoregularity} is given by $T_0=t_j$
with $j=\epsilon^{-1}$. In particular, if $\delta^{-1}$ is bounded above by a
tower of constant height, then $T_0$ in Tao's regularity lemma grows as a tower
of height linear in $\epsilon^{-1}$.

Our proof of Theorem \ref{lsnewbound} shows that the regular approximation
lemma is equivalent to Tao's regularity lemma with similar bounds. In
fact, we show that $T$ in the regular approximation lemma can be taken to be
$T=16T_0/\delta(T_0)^2$, where $T_0=T_0(\delta,\epsilon_0,s)$ is the bound on the
number of parts in Tao's regularity lemma,
$\delta(t)=\min(\frac{g(t)^3}{32t^2},\epsilon/2)$, and
$\epsilon_0=(\epsilon/2)^2$. As Tao's regularity lemma is a simple consequence
of the regular approximation lemma and an application of the weak regularity
lemma, it suffices to show how to deduce the regular approximation lemma from
Tao's regularity lemma.

The proof starts by applying Tao's regularity lemma with $\delta$ and
$\epsilon_0$ as above. For each pair $(X,Y)$ of parts in $Q$, where $X \subset
A$ and $Y \subset B$ with $A,B$ parts of $P$, we randomly add/delete edges
between $X,Y$ with a certain probability so that the density between $X$ and
$Y$ is about the same as the density between $A$ and $B$. We show that in
doing this we have made every pair $(A,B)$ of parts of $P$ $g(t)$-regular with $t=|P|$.
Since $q(Q) \leq q(P)+\epsilon_0$, the edge density $d(X,Y)$ between most pairs
$(X,Y)$ of parts of $Q$ is close to the edge density $d(A,B)$ between $A$ and
$B$, and few edges are changed to obtain a graph $G'$ for which the partition
$P$ is $g$-regular.

We next briefly discuss lower bounds for the regular approximation lemma. In
the case $g$ is a (small) constant function, a tower-type lower bound follows
from Theorem \ref{exceptionalpairs}. If $g$ is at least a tower function, we
get a lower bound of wowzer-type from Theorem \ref{stronglow} and the fact that the
strong regularity lemma follows from the regularity approximation lemma with an
additional application of Szemer\'edi's regularity lemma as discussed earlier.
One could likely come up with a construction giving a general lower bound
essentially matching Theorem \ref{lsnewbound}, but as the already mentioned
interesting cases discussed above are handled by Theorems \ref{exceptionalpairs}
and \ref{stronglow}, we do not include such a construction.

\vspace{0.1cm}
{\bf Organization}

In the next section, we prove some useful tools for establishing lower bounds
for Szemer\'edi's regularity lemma and the strong regularity lemma. In Section
\ref{IrregularSection}, we give a general construction and use it to prove
Theorem \ref{exceptionalpairs} which addresses questions of Szemer\'edi and
Gowers on the number of irregular pairs in Szemer\'edi's regularity lemma. In
Section \ref{strongsect}, we use the general construction to prove Theorem
\ref{stronglow}, which gives a wowzer-type lower bound on the number of parts
of the two partitions in the strong regularity lemma. In Section
\ref{indremovalsection}, we prove the strong cylinder regularity lemma and use
it to prove a tower-type upper bound on the induced graph removal lemma. In
Section \ref{RALsection}, we prove a tower-type upper bound on the number of
parts in the regular approximation lemma. In Section \ref{weakregsection}, we
prove a tight lower bound on the number of parts in the weak regularity lemma.
These later sections, Sections \ref{indremovalsection}, \ref{RALsection} and \ref{weakregsection}, are largely independent of earlier sections and of each other. The interested reader may therefore skip forward without fear of losing the thread.

We finish with some concluding remarks. This includes a discussion showing that
in the regularity lemma, the condition that the parts in the partition are of
equal size does not affect the bounds by much. We also discuss an early version
of Szemer\'edi's regularity lemma, and a recent result of Malliaris and Shelah
which shows an interesting connection between irregular pairs in the regularity
lemma and the appearance of half-graphs. 

Throughout the paper, we systematically omit floor and ceiling signs whenever they are not crucial for the sake of clarity of presentation. We also do not make any serious attempt to
optimize absolute constants in our statements and proofs.

\section{Tools}

Suppose $S=S_1+\cdots+S_n$ is the sum of $n$ mutually independent random
variables, where for each $i$, $\textrm{Pr}[S_i=1]=p$ and
$\textrm{Pr}[S_i=0]=1-p$. The sum $S$ has a binomial distribution with
parameters $p$ and $n$, and has expected value $pn$. A Chernoff-type estimate
(see Theorem A.1.4 in \cite{AlSp}) implies that for $a>0$,
\begin{equation}\label{chernoffest} \textrm{Pr}[S-pn>a] < e^{-2a^2/n}
\end{equation}
By symmetry, we also have $\textrm{Pr}[S-pn<-a]<e^{-2a^2/n}$ and hence
$\textrm{Pr}[|S-pn|>a] < 2e^{-2a^2/n}$.

We start by proving a couple of lemmas on the edge distribution of random
bipartite graphs with different part sizes. Consider the random bipartite graph
$B=B(m,M)$ with parts $[m]$ and $[M]$ formed by each vertex
$i \in [m]$ having exactly $M/2$ neighbors (we assume $M$ is even) in $[M]$
picked uniformly at random and independently of the choices of the
neighborhoods for the other vertices in $[m]$.

The following lemma shows that, with high probability, certain simple estimates on the number of common neighbors or nonneighbors of any two vertices in $B(m, M)$ hold.

\begin{lemma}\label{firstlemma1}
Let $M \geq m$ be positive integers with $M \geq 2^{20}$ even, and $0<\mu<1/2$ be such
that $m \geq 2\mu^{-2}\log M$. Then, with probability at least $1-M^{-2}$, the
random bipartite graph $B=B(m,M)$ has the following properties:
\begin{itemize}
\item for any distinct $j,j'
\in [M]$, the number of $i$ for which $j$ and $j'$ are either both neighbors of
$i$ or both nonneighbors of $i$ is less than $(\frac{1}{2}+\mu)m$.
\item for any distinct $i,i' \in [m]$, the number of common neighbors of $i$
and $i'$ and the number of common nonneighbors of $i$ and $i'$ in $[M]$ are
both less than $\left(\frac{1}{4}+M^{-1/4}\right)M$.
\end{itemize}
\end{lemma}
\begin{proof}
Fix distinct $j,j' \in [M]$. For each $i \in [m]$, the probability that $i$ is
adjacent to both $j,j'$ or nonadjacent to both $j,j'$ is
$(\frac{M}{2}-1)/(M-1)<\frac{1}{2}$, and these events are independent of each
other. Therefore,
by (\ref{chernoffest}), the probability that the number of $i$ for which $j$
and $j'$ are either both neighbors of $i$ or both nonneighbors of $i$ is at
least $(\frac{1}{2}+\mu)m$ is at most $e^{-2(\mu m)^2/m}=e^{-2\mu^2 m} \leq
M^{-4}$. As there are $\binom{M}{2}$ choices for $j,j'$, and $\frac{1}{2}
M^{-2} \geq
M^{-4}\binom{M}{2}$, by the union bound we have that $B$ has the first desired
property
with probability at least $1-\frac{1}{2}M^{-2}$.

As the hypergeometric distribution is at least as concentrated as the corresponding binomial
distribution (for a proof, see Section 6 of \cite{Ho}), we can apply
(\ref{chernoffest}) to conclude that
for each fixed pair $i,i' \in [m]$ of distinct vertices the probability that
the number of common neighbors of  $i$ and $i'$  is at least
$\left(\frac{1}{4}+M^{-1/4}\right)M$ is at most
$e^{-2(M^{-1/4}M)^2/M}=e^{-2M^{1/2}}$. Similarly, for each fixed pair $i,i'
\in [m]$ of distinct vertices the probability that the number of common
nonneighbors of  $i$ and $i'$ in $[M]$ is at least
$\left(\frac{1}{4}+M^{-1/4}\right)M$ is at most $e^{-2M^{1/2}}$.

As there are ${m \choose 2}$ choices for $i,i'$ and $\frac{1}{2}M^{-2} \geq
2e^{-2M^{1/2}}{m \choose 2}$, by the union bound we have that $B$ has the
second desired property with probability at least $1-\frac{1}{2}M^{-2}$. Hence,
with probability at least $1-M^{-2}$, $B$ has both desired properties, which
completes the proof.
\end{proof}

The next lemma shows that the edges in $B(m,M)$ are almost surely uniformly
distributed between large vertex subsets.

\begin{lemma}\label{firstlemma2}
Let $M$ and $m$ be positive integers with $M$ even. With probability at least
$1-M^{-1}$, for any $U_1 \subset [m]$ and $U_2 \subset [M]$ with $|U_1|=u_1$
and $|U_2|=u_2$,
we have
\begin{equation}\label{discrep}|e_B(U_1,U_2)-\frac{1}{2}u_1u_2|\leq\sqrt{f},\end{equation}
where $$f=f(u_1,u_2)=u_1u_2\left(u_1\ln \frac{em}{u_1}+u_2\ln \frac{eM}{u_2}
\right). $$
 \end{lemma}
\begin{proof}
For fixed subsets $U_1 \subset [m]$ and $U_2 \subset [M]$, the random variable
$e_B(U_1,U_2)$, which has mean $\frac{1}{2}|U_1||U_2|$, despite not satisfying a binomial distribution, still satisfies the estimate (\ref{chernoffest}) for the corresponding binomial distribution with parameters $1/2$ and $|U_1||U_2|$. Indeed, note that $e_B(U_1,U_2)$ is the sum of the degrees of the vertices of $U_1$ in $U_2$, and these $|U_1|$ degrees are identical independent random variables, each satisfying a hypergeometric distribution. By Theorem 4 in Section 6 of \cite{Ho}, the expected value of the exponential of a random variable with a hypergeometric distribution is at most the expected value of the exponential of the random variable with the corresponding binomial distribution. Substituting this estimate into the proof of (\ref{chernoffest}) shows that the Chernoff estimate also holds for $e_B(U_1,U_2)$. Hence, the probability
(\ref{discrep}) doesn't hold for a particular pair $U_1,U_2$ is less than
$2e^{-2f/(u_1u_2)}$. By the union bound, the probability that there is a pair
of subsets $U_1 \subset [m]$ and $U_2 \subset [M]$ not satisfying
(\ref{discrep})
is at most
\begin{eqnarray*}\sum_{u_1=1}^m\sum_{u_2=1}^M {m \choose u_1}{M \choose
u_2}2e^{-2f/(u_1u_2)} & \leq &
\sum_{u_1=1}^m\sum_{u_2=1}^M
\left(\frac{em}{u_1}\right)^{u_1}\left(\frac{eM}{u_2}\right)^{u_2}2e^{-2f/(u_1u_2)}
\\ & = & \sum_{u_1=1}^m\sum_{u_2=1}^M 2\left(
\frac{em}{u_1}\right)^{-u_1}\left( \frac{eM}{u_2}\right)^{-u_2} \leq M^{-1}.
\end{eqnarray*}
\end{proof}

From the bipartite graph $B$, we construct equitable partitions
$(A_i,B_i)_{i=1}^m$
of $[M]$, by letting $A_i$ denote the set of neighbors of vertex $i \in [m]$ in
graph $B$. From Lemmas \ref{firstlemma1} and \ref{firstlemma2}, we have the
following corollary.

\begin{corollary}\label{firstcor} Suppose $M \geq m$ are positive integers
with $M \geq 2^{20}$ even, and $0<\mu<1/2$ is such that $m \geq 2\mu^{-2}\log M$. There is
a bipartite graph $B$ with parts $[m]$ and $[M]$, with each vertex in $[m]$ of
degree $M/2$ with the following properties. The estimate  (\ref{discrep}) holds
for all $U_1 \subset [m]$ and $U_2 \subset [M]$ with $|U_1|=u_1$ and
$|U_2|=u_2$, and $B$ satisfies the two properties in the conclusion of Lemma
\ref{firstlemma1}.
\end{corollary}

The next lemma is a useful consequence of the equitable partitions
$(A_i,B_i)_{i=1}^{m}$ behaving randomly. Given a vector $\lambda \in \mathbb{R}^M$ and $1 \leq q < \infty$, write $||\lambda||_q$ for $\left(\sum_{i=1}^M |\lambda_i|^q\right)^{1/q}$ and $||\lambda||_{\infty}$ for $\max_{1 \leq i \leq M} |\lambda_i|$.

\begin{lemma}\label{quiteuseful}
Let $M$ be a positive even integer, $0<\mu<1/2$, and $(A_i,B_i)_{i=1}^m$ be a
sequence of partitions satisfying the conclusion of Corollary \ref{firstcor}.
Suppose
that $0<\sigma,\tau,\alpha$ are such that $\sigma,\tau<1$, $\alpha<1/2$, and
$$(\frac{1}{2}-\mu)(1-\sigma^2)>\frac{\tau}{2}+2(1-\tau)\alpha(1-\alpha).$$ Then
for every sequence $\lambda=(\lambda_1,\ldots,\lambda_M)$ of nonnegative real
numbers which are not all zero with $||\lambda||_2=\sigma||\lambda||_1$, there
are at least
$\tau m$ values of $i$ for which
$\min(a_i,b_i)>\alpha||\lambda||_1$, where $a_i=\sum_{j \in A_i}\lambda_j$ and
$b_i=\sum_{j
\in B_i}\lambda_j$.
\end{lemma}
\begin{proof}
Note that by multiplying each $\lambda_j$ by $1/||\lambda||_1$, we may assume
without loss of generality that $||\lambda||_1=1$.
For distinct $j,j' \in [M]$, let $(j,j')_i$ denote that $j$ and $j'$ lie in
different sets in the partition $(A_i,B_i)$.
Since for any distinct $j,j' \in [M]$, the number of $i$ for which $(j,j')_i$
holds is at least $(\frac{1}{2} - \mu) m$, we have
\begin{equation}\label{estim}\sum_{(j,j')_i}\lambda_j\lambda_{j'} \geq
(\frac{1}{2}-\mu) m\sum_j \lambda_j(1-\lambda_j) =(\frac{1}{2}-\mu)
m(||\lambda||_1-||\lambda||_2^2)=(\frac{1}{2}-\mu) m (1-\sigma^2),\end{equation}
where the sum is over all {\it ordered} triples $(j,j',i)$ with $j,j'$ distinct
and $j$ and $j'$ lie in different sets in the partition $(A_i,B_i)$.
We have the identity
$$\sum_{(j,j')_i}\lambda_j\lambda_{j'}=2\sum_i a_ib_i.$$
Since $a_i+b_i=1$, we have $a_ib_i \leq 1/4$ and if $\min(a_i,b_i) \leq
\alpha$, then $a_ib_i \leq \alpha(1-\alpha)$.
So if $\min(a_i,b_i) \leq \alpha$ for all but less than $\tau m$ values of $i$,
then
$$\sum_{(j,j')_i}\lambda_j\lambda_{j'} < \frac{\tau}{2} m
+2(1-\tau)m\alpha(1-\alpha).$$
Comparing with (\ref{estim}) and dividing by $m$, this contradicts the
supposition, and completes the proof.
\end{proof}

As usual, $G(n,p)$ denotes the random graph on $n$ vertices chosen by picking
each pair of vertices as an edge randomly and independently with probability
$p$. We finish this section with a few standard lemmas on the edge distribution
in $G(n,p)$.

\begin{lemma}\label{fdis}
In $G(n,p)$, with probability at least $1 - n^{-2}$, every pair of disjoint
vertex subsets $U_1$ and $U_2$ satisfy \begin{equation}\label{discrep123}
|e(U_1,U_2)-pu_1u_2| \leq
\sqrt{g},\end{equation}
where $u_1 = |U_1|$, $u_2 = |U_2|$ and, for $u_1 \leq u_2$,
$g=g(u_1,u_2)=2u_1u_2^2\ln \frac{ne}{u_2}$.
\end{lemma}
\begin{proof}
For fixed sets $U_1$ and $U_2$, the quantity $e(U_1,U_2)$ is a binomial
distributed random variable with parameters $u_1u_2$ and $p$. By
(\ref{chernoffest}), we have that the probability (\ref{discrep123}) does not
hold is
less than $2e^{-2g/(u_1u_2)}$.
By the union bound, the probability that there are disjoint sets $U_1$ and
$U_2$ for which (\ref{discrep123}) does not hold is at most \begin{eqnarray*}
\sum_{u_2=1}^{n}\sum_{u_1=1}^{u_2} {n \choose u_2}{n-u_2 \choose
u_1}2e^{-2g/(u_1u_2)}
& \leq & \sum_{u_2=1}^{n}\sum_{u_1=1}^{u_2}
\left(\frac{ne}{u_2}\right)^{u_2}\left(\frac{ne}{u_1}\right)^{u_1}2e^{-2g/(u_1u_2)}
\\ & \leq & \sum_{u_2=1}^{n}\sum_{u_1=1}^{u_2}
2\left(\frac{ne}{u_2}\right)^{-2u_2}
\leq n^{-2}.\end{eqnarray*}
The result follows.
\end{proof}

\begin{lemma}\label{sdis}
In $G(n,p)$, with probability at least $1 -n^{-2}$, every vertex subset $U$
satisfies \begin{equation}\label{discrep1234}|e(U)-p{u \choose 2}|
\leq \sqrt{g},\end{equation}
where $u = |U|$ and $g=g(u)=\frac{1}{2}u^{3}\ln \frac{ne}{u}$.
\end{lemma}
\begin{proof}
For fixed $U$, the quantity $e(U)$ is a binomially distributed random variable
with parameters ${u \choose 2}$ and $p$. By (\ref{chernoffest}), we have that
the
probability (\ref{discrep1234}) does not hold is less than $2e^{-2g/{u \choose
2}}$.
By the union bound, the probability that there is a vertex subset $U$ for which
(\ref{discrep1234}) does not hold is at most $$\sum_{u=2}^{n}{n \choose
u}2e^{-2g/{u
\choose 2}}
\leq \sum_{u=2}^{n}\left(\frac{ne}{u}\right)^{u}2e^{-2g/{u \choose 2}}
\leq 2\sum_{u=2}^{n}\left(\frac{ne}{u}\right)^{-(u+1)} \leq n^{-2}.$$
\end{proof}

Combining the estimates from the previous two lemmas, we can bound the
probability in $G(n,p)$ that there are two not necessarily disjoint subsets
with large edge discrepancy between them.

\begin{lemma}\label{gnpedge}
In $G(n,p)$, the probability that there are integers $u_1$ and $u_2$ with $u_1
\leq u_2$ and  not necessarily disjoint vertex subsets $U_1$ and $U_2$ with
$|U_1|=u_1$ and $|U_2|=u_2$ such that
\begin{equation}\label{discrep12} |e(U_1,U_2)-pu_1u_2|
> 5\sqrt{h},\end{equation}
where  $h=h(u_1,u_2)=u_1u_2^2\ln \frac{ne}{u_2}$, is at most $2 n^{-2}$.
\end{lemma}
\begin{proof}
For sets $U_1$ and $U_2$, letting $U_1'=U_1 \setminus U_2$, $U_2'=U_2 \setminus
U_1$, and
$U=U_1 \cap U_2$, we have $$e(U_1,U_2)=e(U_1',U_2)+2e(U)+e(U,U_2').$$
We have that the bounds in Lemmas \ref{fdis} and \ref{sdis} hold with
probability at
least $1 - 2n^{-2}$.
Hence, using the triangle inequality, and $|U_1'| \leq u_1$, $|U| \leq u_1$,
 $|U_2'| \leq u_2$, we have
$$|e(U_1,U_2)-pu_1u_2| \leq 2\sqrt{g(u_1,u_2)}+2\sqrt{g(u_1)} +pu_1\leq
5\sqrt{h}$$
with probability at least $1 - 2n^{-2}$. Here the extra $p u_1$ factor comes
from the fact that degenerate edges are not counted in $e(U)$.
\end{proof}

\section{A general graph construction} \label{IrregularSection}

In this section, we will define a nonuniform random graph $G=(V,E)$ which,
assuming certain estimates, has the property that any sufficiently regular
partition of its vertex set is close to being a refinement of a particular
partition of $G$ into many parts. As this particular partition has many parts,
this will imply that any sufficiently regular partition will have many parts.
After defining $G$, we will prove that certain useful estimates on the edge
distribution of $G$ hold with positive probability. We will use these estimates
to show that $G$ has the desired property.

\subsection{Defining graph $G$}

Following Gowers \cite{Go}, we attempt to reverse engineer the proof of Szemer\'edi's regularity lemma to show that the upper bound is essentially best possible. The proof of
the regularity lemma follows a sequence of refinements of the vertex set of the
graph until we arrive at a regular partition, with the number of parts in each
partition exponentially larger than in the previous partition. We build a
sequence of partitions of the vertex set, and then describe how the edges of
$G$ are distributed between the various parts of the partition.  To show that
{\it any} (sufficiently) regular partition $\mathcal{Z}$ of $V(G)$ requires
many parts, we show that $\mathcal{Z}$ is roughly a refinement of the
partitions we constructed in defining $G$.

Let $m_1 \geq 2^{200}$ be a positive integer and $\rho=2^{-20}$. For $2 \leq i
\leq s$, let $m_i=m_{i-1}a_{i-1}$, where $a_{i-1}=2^{\lfloor \rho
m_{i-1}^{9/10}\rfloor}$. Suppose $p_i \geq m_i^{-1/10}$ for $1 \leq i \leq
s-1$.

The vertex set $V$ has a sequence of equitable partitions $P_1,\ldots,P_s$,
where $P_j$ is a refinement of $P_i$ for $j>i$ defined as follows. The number
of parts of $P_i$ is $m_i$. For each set $X$ in partition $P_i$, we pick an
equitable partition of $X$ into $a_i$ parts, and let $P_{i+1}$ be the partition
of $V$ with $m_{i+1}=m_{i}a_i$ parts consisting of the union of these
partitions of parts of $P_i$.

For $1 \leq i \leq s-1$, let $G_i$ be a uniform random graph on $P_i$ with edge
probability $p_i$. That is, the vertices of the graph are the $m_i$ pieces of
the partition and we place edges independently with probability $p_i$. In
practice, we will make certain specific assumptions about the edge distribution
of $G_i$ but these will hold with high probability in a random graph. For
example, we shall assume that every vertex in $G_i$ has degree at least $p_i
m_i/2$.

For each $X,Y \in P_i$ with $(X,Y)$ an edge of $G_i$, we have an equitable
partition
$Q_{XY}:X=X^1_Y \cup X^2_Y$ into two parts, where $X^j_Y$ is a union of some of
the parts in
$P_{i+1}$ for $j=1,2$. For each $X \in P_i$, we shall choose the partitions
$Q_{XY}$ with $Y$
adjacent to $X$ in $G_i$ to satisfy the properties of Corollary \ref{firstcor}
with $\mu = 2\rho^{1/2}=2^{-9}$. Note that this is possible since we are taking
$M = a_{i}$ and $m \geq p_i m_i/2 \geq m_i^{9/10}/2$, so $m \geq 2 \mu^{-2}
\log M$, as required.

We finish the construction of $G$ by defining which pairs of vertices are
adjacent. Vertices $u,v \in V$ are adjacent in $G$ if there is $i$, $1 \leq i
\leq s-1$, an edge $(X,Y)$ of $G_i$, and $j \in \{1,2\}$ with $u \in X^j_Y, v
\in Y^j_X$.

An equivalent way of defining the graph $G$ is as follows. For $1 \leq j < i$,
let $G_{j,i}$ denote the graph with vertex set $P_i$, where $X,Y \in P_i$ is an
edge of $G_{j,i}$
if there are $X',Y' \in P_j$ that are adjacent in $G_j$, and $d \in \{1,2\}$
with $X \subset X'^d_{Y'}$ and $Y \subset Y'^d_{X'}$. For $1<i \leq s$, let
$G^i$ denote the graph on $P_{i}$ whose edge set is the union of the edge sets
of $G_{1,i},\ldots,G_{i-1,i}$. Finally, two vertices $u,v \in V$ are adjacent
in $G$ if there is an edge $(X,Y)$ of $G^s$ with
$u \in X$ and $v \in Y$. Note that $G^1$ is simply the empty graph on $P_1$. 

We say that a subset $Z$ {\it $\beta$-overlaps} another set $X$ if $|X \cap Z|
\geq \beta|Z|$, that is, if a $\beta$-fraction of $Z$ is in $X$. A set $Z$
is {\it $\beta$-contained} in a partition $P$ of $V$ if there is a set $X \in
P$ such that $Z$ $\beta$-overlaps $X$.

An equitable partition $\mathcal{Z}$ of $V$ is a {\it
$(\beta,\upsilon)$-refinement} of a partition $P$ of $V$ if, for at least
$(1-\upsilon)|\mathcal{Z}|$ sets $Z \in \mathcal{Z}$, the set $Z$ is
$(1-\beta)$-contained in $P$. In particular, when $\beta=\upsilon=0$, this
notion agrees with the standard notion of refinement. That is, $\mathcal{Z}$ is
a refinement of $P$ is equivalent to $\mathcal{Z}$ being a $(0,0)$-refinement
of $P$.

Our main result, from which Theorem \ref{exceptionalpairs} easily follows, now says that for an appropriate choice of $p_i$, every regular partition of $G$ must be close to a refinement of $P_{s-1}$. In the proof of Theorem \ref{exceptionalpairs}, Theorem \ref{maingen} will be used only in the case $a = s-1$. However, for the lower bound on the strong regularity lemma in Theorem \ref{stronglow}, we will need to apply Theorem \ref{maingen} for various values of $a$. This is why the parameter $a$ is introduced.

\begin{theorem}\label{maingen} Let $\nu=3\sum_{i=1}^{s-1}p_i$, and suppose
$p_i>2^{10}\eta m_1^2$ for $1 \leq i \leq a$, $1-2^7\nu>\epsilon$,
 $\beta=20m_1^{-3/2}$, $\delta<\beta/4$, and $\upsilon=5m_1^{-1/2}$. With
positive probability, the random graph $G$ has the following property. Every
$(\epsilon,\delta,\eta)$-regular equitable partition of $G$ is a
$(\beta,\upsilon)$-refinement of $P_a$.
\end{theorem}

\subsection{Edge distribution in $G$}

Having defined the (random) graph $G$, we now show that with positive
probability $G$ satisfies certain properties (see Lemma \ref{importantstep})
concerning its edge distribution which we will use to prove Theorem
\ref{exceptionalpairs}. Note that $G$ is determined by the $G_i$. For some of
the desired properties, it will be enough to show that the edges in each $G_i$
are sufficiently uniform. For other properties, we will need to consider how
the edge distribution between the various $G_i$ interact with each other. In
bounding the probabilities of certain events, we will often consider the
probability of the event given $G_i$ is picked at random conditioned on the event that
$G_j$ with $j<i$ are already chosen.

In the random graph $G(n,p)$ on $n$ vertices with each
edge taken with probability $p$ independently of the other edges, the expected
degree of each vertex is $p(n-1)$, and the following simple lemma shows that
with high probability no vertex will have degree which deviates much from this
quantity. We will assume throughout this subsection that $n \geq m_1 \geq
2^{200}$.

\begin{lemma}\label{commonneighbors}
The probability that in the random graph $G(n,p)$ there is a
vertex $v$ whose degree satisfies $|\textrm{deg}(v)-pn| > n^{3/4}$ is at most
$e^{-n^{1/2}}$.
\end{lemma}
\begin{proof}
For a fixed vertex $v$, its degree $\textrm{deg}(v)$ follows a binomial
distribution with parameters $n-1$ and $p$. Note that if $|\textrm{deg}(v)-pn|
> n^{3/4}$ then also $|\textrm{deg}(v)-p(n-1)| > n^{3/4}-1$.
From the Chernoff-type estimate (\ref{chernoffest}), we get that the
probability $|\textrm{deg}(v)-pn| > n^{3/4}$ is at most
$2e^{-2(n^{3/4}-1)^2/(n-1)}
\leq \frac{1}{n}e^{-n^{1/2}}$. As there are $n$ vertices, from the union bound,
we get the
probability that there is a vertex $v$ with $|\textrm{deg}(v)-pn| > n^{3/4}$ is at
most $e^{-n^{1/2}}$.
\end{proof}

For $X \in P_i$, we will use $N(X)$ to denote the neighborhood of $X$ in graph
$G_i$, that is, the set of $Y \in P_i$ such that $(X,Y)$ is an edge of $G_i$.
We have the following corollary of Lemma \ref{commonneighbors}.

\begin{corollary}\label{commonneighborscor}
Let $E_1$ be the event that there is $i$, $1 \leq i \leq s-1$, such that $G_i$
has a vertex $X$ with degree $|N(X)|$ satisfying $||N(X)|-p_im_i|>m_i^{3/4}$.
The probability of event $E_1$ is at most $\pi_1:=\sum_{i=1}^{s-1}
e^{-m_i^{1/2}}$.
\end{corollary}

\begin{lemma}\label{newedges}
Suppose $\nu= 3\sum_{i=1}^{s-1} p_i \leq 1/2$. For $2 \leq i \leq s-1$, let
$E_{2i}$ be the event that  $G_i$ has less than $\frac{1}{4}p_im_i^2$ edges
which are not edges of $G^i$. Let $E_2$ be the event that none of the events
$E_{2i}$, $2 \leq i \leq s-1$, occurs. The probability $\pi_2$ of event $E_2$
is at most $\pi_1+\sum_{i=2}^{s-1}e^{-p_i^2m_i^2/24}$, where $\pi_1$ is defined
in Corollary \ref{commonneighborscor}.
\end{lemma}
\begin{proof}
If event $E_1$ does not occur, given  $\nu \leq 1/2$, then the number of edges
of $G^i$ is at most $$\left(\sum_{j=1}^{i-1} p_j+m_j^{-1/4}\right)m_i^2/2 \leq
\frac{\nu}{4}m_i^2 \leq m_i^2/8.$$
Each of the remaining at least ${m_i \choose 2}-\frac{1}{8}m_i^2 \geq m_i^2/3$
unordered pairs of parts of  $P_i$ has probabilty $p_i$ of being an edge of
$G_i$, independently of each other. The expected number of edges of $G_i$ which
are not edges of $G^i$ is therefore at least
$\frac{p_im_i^2}{3}=\frac{p_im_i^2}{4}+\frac{p_im_i^2}{12}$. By
(\ref{chernoffest}),  the probability of event $E_{2i}$ given the number of
edges of $G^i$ is at most $m_i^2/8$ is at most
$e^{-2(p_im_i^2/12)^2/(m_i^2/3)}=e^{-p_i^2m_i^2/24}$. Summing over all $i$, the
probability of event $E_2$ given $E_1$ does not occur is at most
$\sum_{i=2}^{s-1}e^{-p_i^2m_i^2/24}$. We thus have that the probability of
$E_2$ is at most $\pi_2$.
\end{proof}

In a graph $G$ with vertex subsets $U,W$, we let $d_G(U,W)$ denote the fraction
of pairs in $U \times W$ which are edges of $G$. If $U=\{u\}$ consists of a
single vertex $u$, we let $d_G(u,W)=d_G(U,W)$. If the underlying graph $G$ is
clear, we will sometimes write $d(U,W)$ for $d_G(U,W)$. The following lemma
shows that there is a low probability that the density between a vertex and
certain vertex subsets is large.

\begin{lemma}\label{E2}
Let $E_3$ be the event that there is $i$, $1 \leq i \leq s-2$, and distinct
$X,Y \in P_i$, $d \in \{1,2\}$, and $v \in X^{3-d}_Y$, such that $(X,Y)$ is an
edge of $G_i$ but not an edge of $G^i$, and  $d_{G}(v,Y^d_X) > \nu$. The
probability of event $E_3$ is at most
$\pi_3:=\sum_{i=1}^{s-2}\sum_{j=i+1}^{s-1}m_im_je^{-4p_j^2m_j/m_i}$.
\end{lemma}
\begin{proof}
If $(X,Y)$ is an edge of $G_i$ but not an edge of $G^i$, then none of the edges
of $G$ between $X^{3-d}_Y$ and $Y^d_X$ come from the edges of any $G_j$ with $j
\leq i$. So for event $E_3$ to occur, there must be $1 \leq i < j \leq s-1$,
$X, Y \in P_i$ with $(X,Y)$ an edge of $G_i$, and $X' \in P_j$ with $X' \subset
X_Y^{3-d}$, such that $d_{G_j}(X',Y^*)>3p_j$, where $Y^*$ denotes the set of
$Y' \in P_j$ with $Y' \subset Y_X^{d}$.

Fix for now $i$, $1 \leq i \leq s-2$, and $j$ with $i+1 \leq j \leq s-1$. Fix
also an edge $(X,Y)$  of $G_i$ which is not an edge of $G^i$ and $d \in
\{1,2\}$. Fix a set $X' \in P_j$ with $X' \subset X^{3-d}_Y$ and as before let
$Y^*$
denote the set of all $Y' \in P_j$ with $Y' \subset Y^d_X$. The probability
that $d_{G_j}(X',Y^*) > 3p_j$ is by (\ref{chernoffest}) at most
$e^{-2(2p_j|Y^*|)^2/|Y^*|}=e^{-4p_j^2 m_j/m_i}$, since $|Y^*| =
\frac{m_j}{2m_i}$. Summing over all possible
choices of $i$, $j$, $(X,Y)$, $d$, and $X' \in P_j$ with $X' \subset
X^{3-d}_Y$, by the union bound we have the probability of event $E_3$ is at
most
$$\sum_{i=1}^{s-2}\sum_{j=i+1}^{s-1}2m_i^2 \cdot \frac{m_j}{2m_i}e^{-4p_j^2
m_j/m_i} = \pi_3.$$
\end{proof}

Note that the condition that $(X,Y)$ is an edge of $G_i$ but not of $G^i$ is
necessary, since it guarantees that none of the edges in $G_j$ with $j \leq i$
contributes to the edges between $X^{3-d}_Y$ and $Y^d_X$ in $G$. If $(X,Y)$ was
an edge of $G^i$, then we would have a complete bipartite graph between $X$ and
$Y$ and hence $d_G(v, Y_X^d) = 1$.

The codegree $\textrm{codeg}(u,v)$ of two vertices $u$ and $v$ is the number of
vertices $w$ which are connected to both $u$ and $v$. A second useful fact
about $G(n,p)$ is that with high probability the codegree of any two vertices
$u$ and $v$ is roughly $p^2 n$.

\begin{lemma} \label{commonneighbors2}
The probability that in the random graph $G(n,p)$ there are
distinct vertices $u$ and $v$ with $|\textrm{codeg}(u,v) - p^2 n| > n^{3/4}$ is
at most $e^{-n^{1/2}}$.
\end{lemma}
\begin{proof}
For fixed distinct vertices $u$ and $v$, the codegree codeg$(u,v)$ is
binomially distributed with parameters $n-2$ and $p^2$. Note that if
$|\textrm{codeg}(u,v) - p^2 n| > n^{3/4}$ then $|\textrm{codeg}(u,v) - p^2
(n-2)| > n^{3/4}-2$. By Chernoff's inequality (\ref{chernoffest}), the
probability that $|\textrm{codeg}(u,v) - p^2 n| > n^{3/4}$ is at most
$2e^{-2(n^{3/4}-2)^2/(n-2)} \leq n^{-2}e^{-n^{1/2}}$. Using the union bound
over all ${n \choose 2}$ choices of $u$ and $v$ yields the result.
\end{proof}

We have the following corollary of Lemma \ref{commonneighbors2}.

\begin{corollary}\label{commonneighbors2cor}
Let $E_4$ be the event that there is $i$, $1 \leq i \leq s-1$, such that $G_i$
has vertices $X,Y \in P_i$ with codegree satisfying
$|\textrm{codeg}(X,Y)-p_i^2m_i|>m_i^{3/4}$.
The probability of event $E_4$ is at most $\pi_4=\sum_{i=1}^{s-1}
e^{-m_i^{1/2}}$.
\end{corollary}

For $X \in P_i$, let $U(X)=\bigcup_{Y \in N(X)} Y$. The following three lemmas
will be used to prove Lemma \ref{puttog}, which bounds the probability that
there is $i$, $1 \leq i \leq s-1$, $X \in P_i$, and a vertex $v \not \in X$
such
that $d_G(v,U(X)) > \nu$. The proof, which puts together the next three lemmas,
makes
sure that it is unlikely that any $G_{j}$  contributes too much to the density
between $v$ and $U(X)$.

\begin{lemma}\label{puttog3}
Fix $1 \leq i < s$. The probability that there is a pair of distinct sets $X,Y
\in P_i$ which satisfy $d_{G_i}(Y,N(X)) > 2p_i$ is at most
$\pi_{5i}:=2e^{-m_i^{1/2}}$.
\end{lemma}
\begin{proof}
From Lemma \ref{commonneighbors}, we know that $|N(X)| \geq p_im_i-m_i^{3/4}
\geq 3 p_i m_i/4$ for all $X \in P_i$ with probability at least $1
-e^{-m_i^{1/2}}$. Also, by Lemma \ref{commonneighbors2}, we have
$\textrm{codeg}(X,Y) \leq p_i^2m_i+m_i^{3/4} \leq 3 p_i^2 m_i/2$ for all
distinct $X, Y \in P_i$ in graph $G_i$ with probability at least $1
-e^{-m_i^{1/2}}$. The number of edges between $Y$ and $N(X)$ in $G_i$ is just
the codegree $\textrm{codeg}(X,Y)$ of $X$ and $Y$ in $G_i$. Therefore,
$d_{G_i}(Y,N(X)) = \textrm{codeg}(X,Y)/|N(X)|$. Hence, with probability at
least $1 - 2e^{-m_i^{1/2}}$, we have
\[d_{G_i}(Y,N(X)) = \frac{\textrm{codeg}(X,Y)}{|N(X)|} \leq \frac{3p_i^2
m_i/2}{3p_i m_i/4} = 2 p_i.\]
\end{proof}

\begin{lemma}\label{puttog1}
Fix $1 < i < s$. Suppose every vertex of $G^i$ has degree at most $\nu_im_i/2$,
where $\nu_i=3\sum_{j<i}p_j$. The probability that there is a pair of distinct
sets $X,Y \in P_i$ which satisfy $d_{G^i}(Y,N(X)) > \nu_i$ is at most
$\pi'_{5i}:=e^{-m_i^{1/2}} + m_i^2 e^{- \nu_i p_i^2 m_i/4}$.
\end{lemma}
\begin{proof}
Note that $G^i$ is determined by $G_1,\ldots,G_{i-1}$. We show that,
conditioning on $G^i$ has maximum degree at most $\nu_i m_i/2$, the random
graph $G_i$ is such that the probability that there are distinct sets $X,Y \in
P_i$ which satisfy $d_{G^i}(Y,N(X))>\nu_i$ is at most $e^{-m_i^{1/2}} + m_i^2 e^{- \nu_i p_i^2 m_i/4}$.

Fix for now $X,Y \in P_i$. Let $U$ be the neighborhood of $Y$ in $G^i$. By
assumption $|U| \leq
\nu_im_i/2$. The expectation of $|N(X) \cap U|$ is at most $p_i \nu_i m_i/2$.
The probability that $|N(X) \cap U| > 3 p_i \nu_i m_i/4$ is by
(\ref{chernoffest}) at most $$e^{-2(p_i \nu_i  m_i/4)^2/|U|} \leq e^{-2(\nu_i
p_i m_i/4)^2/(\nu_i m_i/2)} = e^{-\nu_i p_i^2 m_i/4}.$$ By Lemma
\ref{commonneighbors}, we know that $|N(X)| \geq 3p_i m_i/4$ for all $X \in
P_i$ with probability at least $1 - e^{-m_i^{1/2}}$. Therefore, using
the union bound, with probability at least $$1 - e^{-m_i^{1/2}} -
m_i^2 e^{-\nu_i p_i^2 m_i/4},$$ we
have
\[d_{G^i} (Y, N(X)) = \frac{|N(X) \cap U|}{|N(X)|} \leq \frac{(3 \nu_i p_i
m_i/4)}{(3 p_i m_i/4)} = \nu_i,\]
for all $X, Y \in P_i$.
\end{proof}

\begin{lemma}\label{puttog2}
Fix $1 \leq i < j < s$. Suppose every vertex of $G_i$ has degree at least $p_i
m_i/2$. Let $E$ be the event that there is a set $X \in P_i$ and a set $Y
\in P_j$ with $Y \not \subset X$ with more than $2p_j|N(X)|\frac{m_j}{m_i}$
neighbors $Y'$ in $G_j$ with $Y' \subset U(X)$. The probability of event $E$ is
at most $\pi_{5ij}:=m_im_je^{-p_j^2p_im_j}$.
\end{lemma}
\begin{proof}
The number of $Y' \in P_j$ with $Y' \subset U(X)$ is $|N(X)|\frac{m_j}{m_i}$.
The probability that a given $Y$ has at least $2p_j|N(X)|\frac{m_j}{m_i}$
neighbors $Y'$ in $G_j$ with $Y' \subset U(X)$ is, by (\ref{chernoffest}), at
most
$$e^{-2\left(p_j|N(X)|\frac{m_j}{m_i}\right)^2/\left(|N(X)|\frac{m_j}{m_i}\right)}
=
e^{-2p_j^2|N(X)|\frac{m_j}{m_i}}.$$ As there are at most $m_jm_i$ such pairs
$X,Y$, we have by the union bound, the probability of event $E$ is at most
$$m_im_je^{-2p_j^2|N(X)|\frac{m_j}{m_i}} \leq m_im_je^{-p_j^2p_im_j}.$$
\end{proof}

From the previous three lemmas, we get the next lemma.

\begin{lemma}\label{puttog}
Consider the event $E_5$ that there
is $i$, $1 \leq i \leq s - 1$, $X \in P_i$, and vertex $v \not \in X$ with
$d_G(v,U(X)) > \nu$. The probability of event $E_5$ is at most
$\pi_5:=\pi_1+\sum_{i=1}^{s-1}\pi_{5i}+\sum_{i=2}^{s-1}\pi'_{5i}+\sum_{1 \leq i
< j < s}\pi_{5ij}$.
\end{lemma}
\begin{proof}
We look at edges in $G$ between $U(X)$ and $Y$ for $X \in P_i$ and $Y \in P_j$,
distinguishing three different cases, namely $j = i$, $j<i$ and $j>i$. For
event $E_5$ to occur at least one of the following events occurs:
\begin{itemize}
\item There is $1 \leq i < s$ and distinct sets $X,Y \in P_i$ with
$d_{G_i}(Y,N(X))>3p_i$.
\item There is $1 < i < s$ and distinct sets $X,Y \in P_i$ which satisfy
$d_{G^i}(Y,N(X))>\nu_i=\sum_{j<i}3p_j$.
\item There is $1 \leq i < j < s$ and sets $X \in P_i$ and $Y \in P_j$ with $Y
\not \subset X$ with
$d_{G_j}(Y,U(X))>3p_j$.
\end{itemize}

The first case is covered by Lemma \ref{puttog3}, the second by Lemma
\ref{puttog1} and the third by Lemma \ref{puttog2}. For Lemmas \ref{puttog1}
and \ref{puttog2} to be applicable, it is enough to know also that for any $i$
and any $X \in P_i$, $|N(X) - p_i m_i| \leq m_i^{3/4} \leq p_i m_i/4$. But this
is just the event that $E_1$ does not occur. From Corollary
\ref{commonneighborscor}, we know this holds with probability at least $1 -
\pi_1$.

Therefore, putting everything together, the probability of event $E_5$ is at
most
$$\pi_1+\sum_{i=1}^{s-1}\pi_{5i}+\sum_{i=2}^{s-1}\pi'_{5i}+\sum_{1 \leq i < j <
s}\pi_{5ij}.$$
\end{proof}

\begin{lemma}\label{E4}
Fix $1 < i \leq s-1$. Let $E_{6,i}$ be the event that there
is $X \in P_i$ such that $X$ has more than  $\nu|N(X)|$ neighbors $Y$ in $G^i$
with $Y \in N(X)$. Let $E_6$ be the event that at least
one of the events $E_{6,i}$ occurs for $1 < i \leq s-1$. Then, the probability
of event $E_6$ is at most $\pi_6 := \pi_1 + \sum_{i=2}^{s-1} m_ie^{-3\nu^2 p_i m_i/8}$.
\end{lemma}
\begin{proof} Let us assume that event $E_1$ does not occur. Then every vertex in $G_i$ has degree at least $\frac{3}{4}p_i m_i$. Moreover, for every $X \in P_i$, its number of neighbors in $G^i$ is at most
\[\left(\sum_{j=1}^{i-1} p_j + m_j^{-1/4}\right)m_i \leq \frac{5}{12} \nu m_i \leq \frac{1}{2} \nu (m_i - 1).\]
Furthermore, the graphs $G_i$ and $G^i$ are still independently chosen given the degree conditions imposed on them by $E_1$ not occuring. Fix for now $X \in P_i$. Then, the expected fraction of elements of $N(X)$
which are neighbors of $X$ in $G^i$ is at most $\nu/2$. Therefore, given
$|N(X)|$, the probability that the number of neighbors of $X$ in $G^i$ is more
than $\nu |N(X)|$ is, by (\ref{chernoffest}) and the fact that a hypergeometric
distribution is at least as concentrated as the corresponding binomial
distribution, at most
$e^{-2\left(\nu|N(X)|/2\right)^2/|N(X)|} = e^{-\nu^2|N(X)|/2}$. 
By the union bound and the assumption that the degree of $X$ in $G_i$ is at least $\frac{3}{4} p_i m_i$, the probability of event $E_{6,i}$ given that $E_1$ does not occur is at most $m_ie^{-3\nu^2 p_i m_i/8}$. Therefore, adding over all $i$, we get that the probability of event $E_6$ is at most $\pi_1 + \sum_{i=2}^{s-1} m_ie^{-3\nu^2 p_i m_i/8} = \pi_6$, as required.
\end{proof}

\begin{lemma}\label{E5}
Let $E_7$ be the event that there is $i$, $1 \leq i \leq s-1$, and vertex
subsets $U_1,U_2 \subset P_i$ of $G_i$ with $|U_1|=u_1$, $|U_2|=u_2$, and $u_1
\leq u_2$ such that
\begin{equation}\label{discre1} |e(U_1,U_2)-p_iu_1u_2| >5\sqrt{h},\end{equation}
where  $h=h(u_1,u_2)=u_1u_2^2\ln \frac{m_ie}{u_2}$. The probability of event
$E_7$ is at most
$\pi_7:=\sum_{i=1}^{s-1} 2m_i^{-2}$.
\end{lemma}
\begin{proof}
By Lemma \ref{gnpedge}, for each $i$, the probability that there are subsets
$U_1,U_2 \subset P_i$ such that (\ref{discre1}) fails is at most $2m_i^{-2}$.
By
the union bound, the probability of event $E_7$ is at most $\sum_{i=1}^{s-1}
2m_i^{-2}$.
\end{proof}

We gather the previous lemmas into one result, which shows that with positive
probability the edge distribution of $G$ has certain desirable properties.

\begin{lemma} \label{importantstep}
Suppose $\nu=3\sum_{i=1}^{s-1}p_i \leq 1/2$. With probability at least $1/2$,
the graph $G$ has the following properties for all $i$, $1 \leq i \leq s-1$.
\begin{itemize}
\item The degree of every vertex in graph $G_i$ differs from $p_im_i$ by at
most $m_i^{3/4}$ and the codegree of
every pair of distinct vertices differs from $p_i^2m_i$ by at most $m_i^{3/4}$.
\item The number of edges of $G_i$ not in $G^i$ is at
least $p_im_i^2/4$.
\item For all $X \in P_i$ and vertex $v \not \in X$, we have $d_G(v,U(X)) \leq
\nu$.
\item For all distinct $X,Y \in P_i$, $d \in \{1,2\}$, and $v \in X^{3-d}_Y$,
such that $(X,Y)$ is an edge of $G_i$ but $X$ and $Y$ are not adjacent in
$G^i$, we have $d_{G}(v,Y^d_X) \leq \nu$.
\item For all $X \in P_i$, $X$ has at most $\nu |N(X)|$ neighbors $Y$ in $G^i$ with $Y \in N(X)$.
\item For all vertex subsets $U_1,U_2 \subset P_i$ of graph $G_i$ with
$|U_1|=u_1$, $|U_2|=u_2$, and $u_1 \leq u_2$,
\begin{equation*}\label{discre} |e(U_1,U_2)-p_iu_1u_2| \leq
5\sqrt{h},\end{equation*}
where  $h=h(u_1,u_2)=u_1u_2^2\ln \frac{m_ie}{u_2}$.
\end{itemize}
\end{lemma}
\begin{proof}
By Corollaries \ref {commonneighborscor}, \ref{commonneighbors2cor} and Lemmas
\ref{newedges}, \ref{E2}, \ref{puttog}, \ref{E4}, \ref{E5} and the union bound,
the probability that at least one $E_h$, $1 \leq h \leq 7$, occurs is at most
$\sum_{h=1}^7 \textrm{Pr}[E_h] \leq \sum_{h=1}^7 \pi_h.$ Using the estimates
$\rho=2^{-20}$, $m_1 \geq 2^{200}$, $m_r=m_{r-1}a_{r-1}$ for $2 \leq r \leq s$,
where $a_{r-1}=2^{\lfloor \rho m_{r-1}^{9/10}\rfloor}$, and $p_i \geq
m_i^{-1/10}$ for $1 \leq r \leq s-1$, it is easy to verify that each $\pi_h
\leq 1/14$ and hence the probability that none of these events occur, i.e., $G$
has the desired properties, is at least $1/2$.
\end{proof}

For the rest of the proof of Theorem \ref{maingen}, we suppose that $G$ has the
properties described in
Lemma \ref{importantstep}.

\subsection{ Regular partitions are close to being refinements}

Let $\theta=m_1^{-1/2}$, $\zeta=\omega=20\theta$, $\beta = \frac{\zeta}{m_1}$,
and $\gamma = 1-\omega$. Suppose for
contradiction that there is an equitable partition $\mathcal{Z}:V=Z_1 \cup
\ldots
\cup Z_k$ of the vertex set of $G$ such that all but at most $\eta k^2$ ordered
pairs $(Z_j,Z_{\ell})$ of parts are $(\epsilon,\delta)$-regular, but
$\mathcal{Z}$ is not a $(\beta, \upsilon)$-refinement of $P_a$.

The two main lemmas for the proof are Lemmas \ref{mainclaim2} and
\ref{mainclaim1}, which show that if $Z_j$ satisfies certain conditions, then
there are at least $\theta^{-1}\eta k$ pairs $(Z_j,Z_{\ell})$ that are not
$(\epsilon,\delta)$-regular. The rest of the proof, Theorem \ref{mainagain},
shows that there are at least $\theta k$ $Z_j$ which satisfy the conditions of
Lemmas \ref{mainclaim2} or \ref{mainclaim1}. Together, we get at least
$\theta^{-1} \eta k \cdot \theta k = \eta k^2$ ordered pairs $(Z_j,Z_{\ell})$
which are not
$(\epsilon,\delta)$-regular, which completes the proof.

Since $P_1$ is a partition into $m_1$ parts, then, by the pigeonhole principle,
each $Z_j$ is
$\frac{1}{m_1}$-contained in $P_1$. We call $Z_j$ {\it ripe} with respect to
$r$
if $Z_j$ is $\beta$-contained in $P_r$ but not $(1-\beta)$-contained in
$P_{r}$. That is, $Z_j$ is ripe if there is $X \in P_r$ containing a
$\beta$-fraction of it but no $X \in P_r$ containing a $(1 - \beta)$-fraction
of it. Let $\psi=2^{20}\beta$. We call $Z_j$ {\it shattered} with respect to
$r$ if $Z_j$ is $(1-\beta)$-contained in $P_r$, but at least a $\psi$-fraction
of $Z_j$ is contained in subsets $X \cap Z_j$ with $X \in P_{r+1}$ and $|X \cap
Z_j| <\beta |Z_j|$. The sense here is that $Z_j$ is shattered by the partition
$P_{r+1}$ if $Z_j$ is almost completely contained in some $X \in P_r$ but it is
not well-covered by $P_{r+1}$.

We say that a subset $X \subset V$ {\it $(\beta,\gamma)$-supports} the
partition $\mathcal{Z}$ if at least a $\gamma$-fraction of the elements of $X$
are in sets $Z_j$ which $\beta$-overlap $X$. That is, a $\gamma$-fraction of
the elements of $X$ are in sets $Z_j$ for which $|X \cap Z_j| \geq \beta
|Z_j|$.

\begin{lemma}\label{claimfirstpart}
Each of the $m_1$ sets in the partition $P_1$ $(\beta,1-\beta m_1)$-supports
$\mathcal{Z}$.
\end{lemma}
\begin{proof}
Let $X \in P_1$. At most a $\beta$-fraction of $V$ is in sets of the form $X
\cap Z_j$ with $|X \cap Z_j| < \beta |Z_j|$. Hence, as $|X| = |V|/m_1$, at most
a $\beta m_1$-fraction of $X$ belongs to $Z_i$ which do not $\beta$-overlap
$X$.
\end{proof}

Let $S_i$ denote the set of $X \in P_i$ which $(\beta,\gamma)$-support
$\mathcal{Z}$. We will let $\kappa_i=\frac{|S_i|}{|P_i|}$. Let $W_i$ denote the
set of $X \in P_i$ for which  $|N(X) \cap S_i| \leq \kappa_i|N(X)|/4$.

\begin{lemma}\label{Wibound}
For $1 \leq i \leq s-1$ with $\kappa_i > 100p_i^{-2}m_i^{-1}\ln (m_ie)$, we
have $|W_i| \leq 100p_i^{-2}\ln (\kappa_i^{-1}e)$.
\end{lemma}
\begin{proof}
In graph $G_i$, the number $e(W_i,S_i)$ of pairs in $W_i \times S_i$ which are
edges is at most $\kappa_i/4$ times the sum of the degrees of the vertices in
$W_i$. Since, by Lemma \ref{importantstep}, every vertex has degree at most
$2p_i|P_i|$ in $G_i$, we have
$e(W_i,S_i) \leq |W_i|\cdot (\kappa_i/4) \cdot 2p_i|P_i|=p_i|S_i||W_i|/2$.
Hence, by Lemma \ref{importantstep},
$$p_i|S_i||W_i|/2 \leq |e(W_i,S_i)-p_i|W_i||S_i|| \leq 5\sqrt{h},$$
where $h=u_1u_2^2\ln \frac{m_i e}{u_2}$, and $u_1=\min (|W_i|,|S_i|)$ and
$u_2=\max(|W_i|,|S_i|)$. By squaring both sides, substituting
$u_1u_2=|W_i||S_i|$ and simplifying, we have $u_1
\leq 100p_i^{-2}\ln \frac{m_ie}{u_2}$. If $u_1=|S_i|=\kappa_im_i$, then
$$\kappa_im_i=u_1 \leq 100p_i^{-2}\ln \frac{m_ie}{u_2} \leq 100p_i^{-2} \ln
(m_ie),$$ contradicting our assumption. Hence,
$u_1=|W_i|$, and $|W_i| \leq 100p_i^{-2}\ln (\kappa_i^{-1}e)$.
 \end{proof}

The following simple proposition demonstrates the hereditary nature of
supporting sets.

\begin{proposition}\label{firstpropherabc}
Suppose $Y \in P_i$ is such that $Y$ $(\beta,\gamma)$-supports the partition
$\mathcal{Z}$. Then, for each $X \in P_i$ distinct from $Y$ and $d \in
\{1,2\}$, $Y^d_X$ $(\beta/4,1/4)$-supports the partition
$\mathcal{Z}$.
\end{proposition}
\begin{proof}
We will use the fact $\gamma \geq 7/8$. The sum of $|Z_t \cap Y_X^d|$ over all
$Z_t$ which $\beta$-overlap $Y$ but do not $\beta/4$-overlap $Y_X^d$ is at most
$|Y|/4$. Since $Y$ $(\beta,\gamma)$-supports the partition, the sum of $|Z_t
\cap Y^d_X|$ over all $Z_t$ which $\beta/4$-overlaps $Y^d_X$ is at least
$$|Y_X^d|-(1-\gamma)|Y|-|Y|/4 \geq |Y|/8=|Y_X^d|/4.$$
Hence $Y^d_X$ $(\beta/4,1/4)$-supports the partition $\mathcal{Z}$.
\end{proof}

The following lemma shows that if $Z_j$ satisfies certain conditions, then
there are many (at least $\theta^{-1}\eta k$) $Z_{\ell}$ such that
$(Z_j,Z_{\ell})$ is
not $(\epsilon,\delta)$-regular.

\begin{lemma}\label{mainclaim2} Suppose $X \in P_i \setminus W_i$,
$\kappa_i=|S_i|/|P_i| \geq 1/2$, $Z_j$ is shattered with respect to $i$, and
$Z_j$ $(1-\beta)$-overlaps $X$. There are at least $\theta^{-1}\eta k$ sets
$Z_{\ell} \in \mathcal{Z}$ for which $(Z_j,Z_{\ell})$ is not $(\epsilon,\delta)$-regular.
\end{lemma}
\begin{proof}
Since $Z_j$ is shattered with respect to $i$ and $Z_j$ $(1-\beta)$-overlaps
$X$, then $|X \cap Z_j| \geq (1-\beta)|Z_j|$, but the sum of $|X' \cap Z_j|$
over all $X' \in P_{i+1}$ with
$|X' \cap Z_j| <\beta|Z_j|$ is at least $\psi|Z_j|$.

Let $Z_j'= X \cap Z_j$, so $|Z_j'| \geq (1-\beta)|Z_j|$. For each $X' \in
P_{i+1}$ with $X' \subset X$ and $|X' \cap Z_j| <\beta|Z_j|$, let
$\lambda_{X'}=|X' \cap Z_j'|/|Z_j|$, so each
$\lambda_{X'}<\beta$, i.e.,
$\beta>||\lambda||_{\infty}$. Also, $||\lambda||_1 \geq \psi-\beta$ follows from the facts that $|X \cap Z_j| \geq (1-\beta)|Z_j|$ and the sum of $|X' \cap Z_j|$ over all $X' \in P_{i+1}$ with  $|X' \cap Z_j| <\beta|Z_j|$ is at least $\psi|Z_j|$. Therefore,
$$\sigma^2=\left(\frac{||\lambda||_2}{||\lambda||_1}\right)^2 \leq
\frac{||\lambda||_{\infty}}{||\lambda||_1} < \frac{\beta}{\psi-\beta}=\frac{1}{2^{20}-1} < 2^{-19}.$$

By Lemma \ref{quiteuseful} with $\alpha=1/8$, $\mu=2\rho^{1/2}=2^{-9}$,
$\sigma<2^{-9}$, and
$\tau=1-2^{-5}$, we have that the number of $Y \in N(X)$
for which \begin{equation}\label{12abcd} |Z_j \cap X^1_Y|,|Z_j \cap X^2_Y| \geq
\alpha||\lambda||_1|Z_j| \geq
\alpha(\psi-\beta)|Z_j|\end{equation} is at least $(1-2^{-5})|N(X)|$, where
$N(X)$ is
the neighborhood of $X$ in graph $G_i$. 

By Lemma \ref{importantstep}, the
number of $Y \in N(X)$ which are also adjacent to $X$ in $G^i$ is at most $\nu|N(X)|$.
Also, since $X \not \in W_i$, we have $|N(X) \cap S_i| \geq \kappa_i|N(X)|/4$.
Therefore, the number of $Y \in S_i$ with $(X,Y)$ an edge of $G_i$ but not an
edge
of $G^i$, and (\ref{12abcd}) is satisfied is at least $$(1-2^{-5}) |N(X)|-|N(X)
\setminus
S_i|-\nu|N(X)| > (\kappa_i/4-2^{-5}-\nu)|N(X)| \geq |N(X)|/16.$$

Fix such a $Y$, and let $U_d=Z_j \cap X^d_Y$ for $d \in \{1,2\}$, so
$|U_1|,|U_2| \geq \alpha(\psi-\beta)|Z_j|$. Since $Y \in S_i$, we have
$Y$ $(\beta,\gamma)$-supports $\mathcal{Z}$. By Proposition
\ref{firstpropherabc}, $Y^d_X$ $(\beta/4,1/4)$-supports
$\mathcal{Z}$. By Lemma \ref{importantstep}, for each vertex $v \in X^{3-d}_Y$,
we have $d(v,Y^{d}_X) \leq \nu$. In particular, $d(U_{3-d},Y^d_X) \leq \nu$.
Let $R^d$ be the union of all $Z_{\ell} \cap Y^d_X$ such that
$Z_{\ell}$
$\beta/4$-overlaps $Y^d_X$, so $R^d$ is a subset of $Y^d_X$ of cardinality at
least $|Y^d_X|/4$. Hence, $d(U_{3-d},R^d) \leq 4\nu$.

For $Z_{\ell}$ which $\beta/4$-overlaps $Y^d_X$, let $Z_{\ell}'=R^d \cap
Z_{\ell}$, so $|Z_{\ell}'|
\geq \beta|Z_{\ell}|/4$. We next show that there are many $Z_{\ell}'$ which
satisfy
\begin{equation}\label{thirdest2}
d(U_{3-d},Z_{\ell}') \leq 8 \nu.
\end{equation}

Indeed, the union of the $Z_{\ell}'$ which do not satisfy (\ref{thirdest2}) has
cardinality at most $\frac{1}{2}|R^d|$, so at least
$1/2$ of $R^d$ consists of the union of $Z_{\ell}'$ which satisfy
(\ref{thirdest2}).
The number of ${\ell}$ which satisfy
(\ref{thirdest2}) is at least
\begin{eqnarray*} \frac{1}{2}|R^d|/|Z_{\ell}| & \geq &
\frac{1}{2}(|Y^d_X|/4)/|Z_{\ell}| \\
& = & \frac{1}{16}|Y|/|Z_{\ell}| = \frac{1}{16}k/m_i,
\end{eqnarray*}
where in the last equality we used $|Y| = |V|/m_i$ and $|Z_{\ell}|=|V|/k$.

For each $Z_{\ell}'$ which satisfies (\ref{thirdest2}), we have $d(U_d,Z_{\ell}')=1$
since $(X,Y)$ is an edge of $G_i$ and, therefore, the density of edges between
$X_Y^d$ and $Y_X^d$ is $1$.
Hence
$$d(U_d,Z_{\ell}')-d(U_{3-d},Z_{\ell}') \geq 1-8\nu \geq \epsilon.$$
Since also $|U_d|,|U_{3-d}| \geq \alpha(\psi-\beta)|Z_j| \geq \delta |Z_j|$,
and $|Z_{\ell}'| \geq \frac{\beta}{4} |Z_{\ell}| \geq \delta |Z_{\ell}|$,
we have in this case $(Z_j,Z_{\ell})$ is not $(\epsilon,\delta)$-regular.

Since the number of such $Y$ is at least $|N(X)|/16$, we have that the number
of pairs $(Z_{\ell},Y^d_X)$ such that $Z_{\ell}$ $\beta/4$-overlaps $Y^d_X$ and
$(Z_j,Z_{\ell})$ is not $(\epsilon,\delta)$-regular is at least
$\left(\frac{1}{16}k/m_i\right)\left(|N(X)|/16\right) \geq
2^{-9}p_ik,$ where we used $|N(X)| \geq \frac{1}{2}p_im_i$ from Lemma
\ref{importantstep}. As $Z_{\ell}$ $\beta/4$-overlaps $Y^d_X$ in each such
pair, a
given $Z_{\ell}$ is in at most $4\beta^{-1}$ such pairs. Hence, the number of
$Z_{\ell}$
for which $(Z_j,Z_{\ell})$ is not $(\epsilon,\delta)$-regular is at least
$2^{-11}\beta p_ik \geq \theta^{-1}\eta k.$
\end{proof}

Like Lemma \ref{mainclaim2}, the next lemma shows that if $Z_j$ satisfies
certain conditions, then there are at least $\theta^{-1}\eta k$ $Z_{\ell}$ such
that
$(Z_j,Z_{\ell})$ is not $(\epsilon,\delta)$-regular.

\begin{lemma}\label{mainclaim1} Suppose $X \in P_i \setminus W_i$, $\kappa_i \geq 1/2$, $Z_j$ is ripe with respect to $i$, and $Z_j$
$\beta$-overlaps $X$. Then there are at least $\theta^{-1}\eta k$ sets
$Z_{\ell} \in \mathcal{Z}$ for
which $(Z_j,Z_{\ell})$ is not $(\epsilon,\delta)$-regular.
\end{lemma}
\begin{proof}
Since $Z_j$ is ripe with respect to $i$, $|X \cap Z_j| < (1 - \beta) |Z_j|$.
Therefore, letting $U' = Z_j \setminus X$, we have $|U'| \geq \beta|Z_j|$.

By Lemma \ref{importantstep}, for each vertex $v$ of $G$ which is not in $X$,
we have $d(v,U(X)) \leq \nu$. Since $X \not \in W_i$, we have 

\begin{equation}\label{a1b2}|N(X) \cap S_i|
\geq \kappa_i |N(X)|/4 \geq |N(X)|/8.
\end{equation}
So \begin{equation}\label{firstest}
d(v,\bigcup_{Y \in N(X) \cap S_i}Y) \leq 8\nu.
\end{equation}

Fix for this paragraph $Y \in N(X) \cap S_i$.
Since $Z_j$ $\beta$-overlaps $X$, there is $d=d(j,Y) \in \{1,2\}$ such that
$Z_j$ $\beta/2$-overlaps $X^d_Y$. Let $U_Y=Z_j \cap X^d_Y$, so $|U_Y| \geq
\frac{\beta}{2}|Z_j|$ and $d(U_Y,Y^d_X)=1$. As $Y \in S_i$, we have $Y$
$(\beta,\gamma)$-supports $\mathcal{Z}$. By Proposition \ref{firstpropherabc},
 $Y^d_X$ $(\beta/4,1/4)$-supports $\mathcal{Z}$.

For $Y \in N(X) \cap S_i$, let $R_Y$ denote the set of
vertices $y$ which are in $Y^d_X$ with $d=d(j,Y)$, and $y$ is also in a
$Z_{\ell}$ which
$\beta/4$-overlaps $Y^d_X$, so \begin{equation}\label{d4z}|R_Y| \geq \frac{1}{4}|Y_X^d|=\frac{1}{8}|Y|=\frac{|V|}{8m_i}.\end{equation} Let
$R=\bigcup_{Y \in N(X) \cap S_i} R_Y$. We have 
\begin{equation}\label{y25}|R| \geq |N(X) \cap S_i|\frac{|V|}{8m_i} \geq 2^{-6}|N(X)|\frac{|V|}{m_i} \geq 2^{-6}\frac{p_im_i}{2}\cdot \frac{|V|}{m_i}=2^{-7}p_i|V|,\end{equation}
where we used (\ref{d4z}), (\ref{a1b2}), and $|N(X)| \geq p_im_i/2$. 

By (\ref{firstest}) and (\ref{d4z}), we have for $v \not \in X$,
\begin{equation}\label{secondest}
d(v,R)\leq  2^6\nu. \end{equation}

By (\ref{secondest}), we have $d(U',R) \leq 2^6\nu$.
For $Z_{\ell}$ which $\beta/4$-overlaps $Y^d_X$ for some $Y \in N(X) \cap S_i$
and $d=d(j,Y)$, let $Z_{\ell}^Y=Z_{\ell} \cap Y^d_X$, so $|Z_{\ell}^Y|
\geq (\beta/4)|Z_{\ell}|$. By definition, for each $Y \in N(X) \cap S_i$, $R_Y$ is the
union of the sets $Z_{\ell}^Y$. We next show that there are many $Z_{\ell}^Y$ which satisfy
\begin{equation}\label{thirdest}
d(U',Z_{\ell}^Y) \leq 2^7\nu.
\end{equation}

Indeed, the union of the $Z_{\ell}^Y$ which do not satisfy (\ref{thirdest}) has
cardinality at most $\frac{1}{2}|R|$, so at least
$1/2$ of $R$ consists of the union of $Z_{\ell}^Y$ which satisfy
(\ref{thirdest}).
The number of pairs $(\ell,Y)$  which satisfy (\ref{thirdest}) is at least
\begin{eqnarray*} \frac{1}{2}|R|/|Z_\ell| & \geq & \frac{1}{2}2^{-7}p_i|V|/|Z_{\ell}|=2^{-8}p_ik. 
\end{eqnarray*}
where we used (\ref{y25}) and $|Z_{\ell}|=|V|/k$. Since for each such $\ell$, we have $Z_{\ell}$
$\beta/4$-overlaps
$Y_X^d$, each such $\ell$ is in at most $4\beta^{-1}$ of the pairs $(\ell,Y)$
we just
counted. Hence, the number of $\ell$ for which there is $Y$ such that
(\ref{thirdest}) holds is at least $2^{-10}\beta p_i k$.

By (\ref{thirdest}) and $d(U_Y,Z_{\ell}^Y)=1$,
we have $$d(U_Y,Z_{\ell}^Y)-d(U',Z_{\ell}^Y) \geq 1-2^7\nu>\epsilon,$$
and as $|U_Y|,|U'| \geq \frac{\beta}{2}|Z_j| \geq \delta |Z_j|$ and
$|Z_{\ell}^Y| \geq \frac{\beta}{4} |Z_{\ell}| \geq \delta |Z_{\ell}|$, we have
that $(Z_j,Z_{\ell})$ is not $(\epsilon,\delta)$-regular for at least
$2^{-10}\beta p_i k \geq \theta^{-1}\eta k$ values of $\ell$.
\end{proof}

The following theorem completes the proof.

\begin{theorem}\label{mainagain}
The number of ordered pairs $(Z_j,Z_{\ell})$ which are not
$(\epsilon,\delta)$-regular is at least  $\eta k^2$.
\end{theorem}
\begin{proof}
By assumption, $\mathcal{Z}$ is not a $(\beta,\upsilon)$-refinement of $P_a$.
Hence, the number of parts $Z_j$ of partition $\mathcal{Z}$ which are not
$(1-\beta)$-contained in $P_a$ is at least
$\upsilon k$. Let $i_0$ be the minimum positive integer for which $P_{i_0}$ is
not a $(\beta,\upsilon)$-refinement of $\mathcal{Z}$. As, by assumption, $P_a$
is not a $(\beta,\upsilon)$-refinement of $\mathcal{Z}$, we have $1 \leq i_0
\leq a$.

\begin{claim}\label{claimforthm1}
We have $\kappa_1=1$ and $\kappa_{i} \geq 1/2$ for $i<i_0$.
\end{claim}
As $\beta=\zeta/m_1$, by Lemma \ref{claimfirstpart}, each of the $m_1$ parts of
partition $P_1$ $(\beta,1-\zeta)$-supports $\mathcal{Z}_j$. As $\zeta=\omega$
and $\gamma=1-\omega$, it follows that   $S_1=P_1$ and
$\kappa_1=|S_1|/|P_1|=1$. From the definition of $i_0$, for each $i<i_0$,
$P_{i_0}$ is a $(\beta,\upsilon)$-refinement of $\mathcal{Z}$. Fix for this
paragraph such an $i<i_0$. Hence at most a $(\beta+\upsilon)$-fraction of the
vertices are in parts $Z_j \cap X$ with $X \in P_i$ and $Z_j \in \mathcal{Z}$
and $|Z_j \cap X|<(1-\beta)|Z_j|$. In particular, as $1-\beta>\beta$ and
$\gamma=1-\omega$, the fraction of $X \in P_i$ which do not
$(\beta,\gamma)$-support $\mathcal{Z}$ is at most
$\frac{\beta+\upsilon}{\omega}$. Hence $\kappa_i \geq
1-\frac{\beta+\upsilon}{\omega} \geq 1/2$, which completes the proof of Claim
\ref{claimforthm1}.

Consider the partition $\mathcal{Z}=\mathcal{Z}^1 \cup \mathcal{Z}^2 \cup
\mathcal{Z}^3 \cup \mathcal{Z}^4 \cup \mathcal{Z}^5 \cup \mathcal{Z}^6$, where
$Z_j \in \mathcal{Z}^h$ if $h$ is minimum such that $Z_j$ satisfies property
$h$ below.
\begin{enumerate}
\item There is $i<i_0$ and $X \in P_i \setminus W_i$ such that $Z_j$
is shattered with respect to $i$ and $(1-\beta)$-overlaps $X$ or if $Z_j$ is
ripe with respect to $i$ and $\beta$-overlaps $X$,
\item For every $X \in P_1$ such that $Z_j$ $\beta$-overlaps $X$, $X \in W_1$.
\item There is $i$, $1 < i \leq i_0$, and $X \in W_i$ such that $Z_j$
$\beta$-overlaps $X$,
\item $i_0>1$ and $Z_j$ is ripe with respect to $i_0$, and there is $X \in
W_{i_0}$ such that $Z_j$ $\beta$-overlaps $X$.
\item $Z_j$ is ripe with respect to $i_0$, and there is $X \in P_{i_0}
\setminus W_{i_0}$ such that $Z_j$ $\beta$-overlaps $X$.
\item $Z_j$ is $(1-\beta)$-contained in $P_{i_0}$.
\end{enumerate}

It is not immediately obvious that the above six subfamilies of $\mathcal{Z}$
form a partition of $\mathcal{Z}$, so we first show that this is indeed the
case.

\begin{claim}\label{claimforthm2}
The above six subfamilies form a partition of $\mathcal{Z}$.
\end{claim}
As $Z_j  \in \mathcal{Z}^h$ if and only $h$ is the minimum such that $Z_j$
satisfies property $h$, the subfamilies $\mathcal{Z}^h$, $1 \leq h \leq 6$, are
pairwise disjoint. We thus need to show that each $Z_j$ is in at least one
$\mathcal{Z}^h$. Suppose for contradiction that $Z_j$ is in none of the
$\mathcal{Z}^h$. By property 6, $Z_j$ is not $(1-\beta)$-contained in
$P_{i_0}$. If $Z_j$ is $\beta$-contained in $P_{i_0}$, then $Z_j$ is ripe with
respect to $i_0$, and there is $X \in P_{i_0}$  such that $Z_j$
$\beta$-overlaps $X$. Either every such $X \in W_{i_0}$ or there is such an $X
\not \in W_{i_0}$, and by properties 2, 4 and 5, we must have in this case
$Z_j$ is in a $\mathcal{Z}^h$ for some $h \leq 5$. So $Z_j$ is not
$\beta$-contained in $P_{i_0}$, and noting that every $Z_j$ is
$\beta$-contained in $P_1$, we must have $Z_j$ is ripe or shattered with
respect to at least one $i$ with $1 \leq i<i_0$. In particular, there is
$i<i_0$ and $X \in P_i$ such that $Z_j$ is shattered with respect to $i$ and
$(1-\beta)$-overlaps $X$ or $Z_j$ is ripe with respect to $i$ and
$\beta$-overlaps $X$. Since $Z_j \not \in \mathcal{Z}^1$, for every such
$i<i_0$ and $X \in P_i$, we must have $X \in W_i$. But then $Z_j \in
\mathcal{Z}^2$ or $\mathcal{Z}^3$. Thus $Z_j$ is in at least one of the six
subfamilies, completing the claim that that these subfamilies indeed form a
partition of $\mathcal{Z}$.

As the number of parts $Z_j$ of partition $\mathcal{Z}$ which are not
$(1-\beta)$-contained in $P_{i_0}$ is at least $\upsilon k$, we have
\begin{equation}\label{up6}
| \mathcal{Z} \setminus \mathcal{Z}^6| \geq \upsilon k.
\end{equation}

Let $w_i=|W_i|/|P_i|$. By Claim \ref{claimforthm1}, $\kappa_1=1$ and $\kappa_i
\geq 1/2$ for $i<i_0$. Hence, from
Lemma \ref{Wibound}, we have $w_1 \leq 100p_1^{-2}m_1^{-1}\ln(2e) \leq
m_1^{-1/2}$ and similarly
$w:=\sum_{1<i<i_0} w_i \leq m_2^{-1/2}$. Here we used $p_i \geq m_i^{-1/10}$,
$m_1 \geq 2^{200}$, and $m_i \geq 2^{m_{i-1}^{1/2}}$.

We next bound the size of $\mathcal{Z}^2$. If $Z_j \in \mathcal{Z}^2$, since
$Z_j$ does not $\beta$-overlap any $X \in P_1 \setminus W_1$, and
$|P_1|=m_1$, then at least a $(1-\beta m_1)$-fraction of $Z_j$ is contained in
sets $X \in W_i$. Hence, the fraction of $Z_j \in \mathcal{Z}$ which
satisfy $Z_j \in \mathcal{Z}^2$ is at most $(1-\beta m_1)^{-1}w_1 \leq
2m_1^{-1/2}$,
i.e., $|\mathcal{Z}^2| \leq (2m_1^{-1/2})k$.

Similarly, the fraction of $Z_j \in \mathcal{Z}$ such that there is $1<i < i_0$
and $X \in
W_i$ for which $Z_j$ $\beta$-overlaps $X$ is at most  $\beta^{-1}m_2^{-1/2}$.
Hence, $|\mathcal{Z}^3| \leq \beta^{-1}m_2^{-1/2}k$.

By Lemmas \ref{mainclaim2} and \ref{mainclaim1}, each $Z_j \in \mathcal{Z}^1$
is in at least $\theta^{-1}\eta k$ pairs $(Z_j,Z_{\ell})$ which are not
$(\epsilon,\delta)$-regular. We are thus done if $|\mathcal{Z}^1| \geq \theta
k$. So we may suppose $|\mathcal{Z}^1| < \theta k$.

We next give a lower bound on $\kappa_{i_0}$.

\begin{claim}\label{claimforthm3}
We have $\kappa_{i_0} \geq 1/2$.
\end{claim}
Note that if $i_0=1$, by Claim \ref{claimforthm1}, $\kappa_{1}=1$. So we may
suppose that $i_0>1$. In order to give a
lower bound on $\kappa_{i_0}$, we next give an upper bound on the union of all
sets $Z_j \cap X$ with $|Z_j \cap X| <\beta |Z_j|$ and $X \in P_{i_0}$. If
$Z_j$ is not $(1-\beta)$-contained in $P_{i_0}$, then it must be shattered or
ripe with respect to some $i$ with $i<i_0$, or must have at most $\psi|Z_j|$
vertices in parts $X \cap Z_j$ with $X \in P_{i_0}$ and $|X \cap Z_j|<\beta |Z_j|$.
Each $Z_j$ which is shattered or ripe with respect to some $i$ with $i<i_0$ is
in $\mathcal{Z}^1$, $\mathcal{Z}^2$, or $\mathcal{Z}^3$, and hence the number of such
$Z_j$ is at most
\begin{equation}\label{up123}
| \mathcal{Z}^1 \cup \mathcal{Z}^2 \cup \mathcal{Z}^3| \leq \theta k+
 (2m_1^{-1/2})k + \beta^{-1}m_2^{-1/2}k.
\end{equation}
Every set $Z_j$ which $(1-\beta)$-overlaps $P_{i_0}$ has at most a
$\beta$-fraction of it contained in sets $X \cap Z_j$ with $|X \cap
Z_j|<\beta|Z_j|$ and $X \in P_{i_0}$. In total, we get that the fraction of vertices
which belong to one of the sets $X \cap Z_j$ with $|X \cap Z_j|<\beta |Z_j|$
and $X \in P_{i_0}$ is at most $$\theta  + (2m_1^{-1/2}) +
\beta^{-1}m_2^{-1/2}+\beta+\psi.$$
The fraction of sets in $P_{i_0}$ which do not $(\beta,\gamma)$-support
$\mathcal{Z}$ is therefore $$1-\kappa_{i_0} \leq \omega^{-1}\left(\theta  +
(2m_1^{-1/2}) + \beta^{-1}m_2^{-1/2}+\beta+\psi\right) \leq 1/2.$$ Hence,
$\kappa_{i_0} \geq 1/2$, which completes Claim \ref{claimforthm3}.

Noting that $\kappa_{i_0} \geq 1/2$, the same argument that bounded
$|\mathcal{Z}^3|$ also gives that
\begin{equation}\label{up4}|\mathcal{Z}^4| \leq \beta^{-1} m_{2}^{-1/2}k.
\end{equation} From the bounds (\ref{up6}), (\ref{up123}), (\ref{up4}), we have
$$|\mathcal{Z}^5| \geq |\mathcal{Z}\setminus \mathcal{Z}^6|-|\mathcal{Z}^1 \cup
\mathcal{Z}^2 \cup \mathcal{Z}^3|-|\mathcal{Z}^4|\geq \upsilon k - \left(\theta k+
2m_1^{-1/2}k + \beta^{-1}m_2^{-1/2}k\right)-\beta^{-1} m_{2}^{-1/2}k \geq
\theta k.$$

As $\kappa_{i_0} \geq 1/2$, by Lemma \ref{mainclaim1}, each $Z_j \in
\mathcal{Z}^5$ is in at least $\theta^{-1}\eta k$ pairs $(Z_j,Z_{\ell})$ which
are not $(\epsilon,\delta)$-regular. Hence, the number of irregular pairs is at
least $$|\mathcal{Z}^5| \theta^{-1}\eta k \geq \eta k^2,$$ which completes the
proof.
\end{proof}

\subsection{Proof of Theorem \ref{exceptionalpairs}}\label{sectexcept}

To prove Theorem \ref{exceptionalpairs}, it suffices to prove the following
theorem.

\begin{corollary}
Let $\epsilon=1/2$, $\delta=2^{-400}$, $\eta<2^{-700}$, $s=\lfloor
2^{-600}\eta^{-1} \rfloor$, and $k$ be
at most a tower of twos of height $s$. There is a graph $G=(V,E)$ for which any
equitable partition $\mathcal{Z}$ of $V$ into at most $k$ parts has at least
$\eta k^2$ ordered pairs of parts which are not $(\epsilon,\delta)$-regular.
\end{corollary}
\begin{proof}
Let $m_1=2^{200}$ and $p_i=\max(m_i^{-1/10},2^{500}\eta)$ for $1 \leq i \leq s-1$
and consider the graph $G$ given
with positive probability by Theorem \ref{maingen}. As
$\nu=3\sum_{i=1}^{s-1}p_i$, we have
\begin{eqnarray*}
\nu \leq 3\sum_{i=1}^{s-1}(m_i^{-1/10}+2^{500}\eta)=3\cdot
2^{500}\eta(s-1)+3\sum_{i=1}^{s-1}m_i^{-1/10} \leq 3 \cdot 2^{-100}
+3\frac{3}{2}p_1 < 6p_1 < 2^{-10},
\end{eqnarray*}
so $1-2^7\nu>\epsilon$.
The first inequality uses that the maximum of two nonnegative real
numbers is at most their sum. The second inequality uses $s=\lfloor
2^{-600}\eta^{-1} \rfloor$ and the fact that the sum of $m_i^{-1/10}$ rapidly
converges, and $p_1 = m_1^{-1/10}=2^{-20}$.

Note that as $m_1=2^{200}>2^{2^{2^{2}}}$ and $m_i \geq 2^{m_{i-1}^{1/2}}$ for
$i>1$, we have $|P_i|=m_i$ is greater than a tower of twos of height $i+2$ for $1
\leq i \leq s$. By Theorem \ref{maingen} with $a=s-1$, any
$(\epsilon,\delta,\eta)$-regular equitable partition of $G$ is a
$(\beta,\upsilon)$-refinement of $P_{s-1}$. In particular, at least one part of
$P_{s-1}$ contains at least a $(1-\beta)$-fraction of a part from $\mathcal{Z}$.
As $1-\beta>1/2$, this implies $|\mathcal{Z}| \geq \frac{1}{2}|P_{s-1}| > k$, which
completes the proof.
\end{proof}

\section{Lower bound for the strong regularity lemma}\label{strongsect}

In this section we prove Theorem \ref{stronglow}, which gives a lower bound on
the strong regularity lemma and states the following. Let $0<\epsilon<2^{-100}$ and $f:\mathbb{N} \rightarrow (0,1)$ be a decreasing function with
$f(1) \leq 2^{-100}\epsilon^{6}$. Define $W_{\ell}$ recursively by $W_1=1$,
$W_{\ell+1} = T\left(2^{-70}\epsilon^5/f(W_{\ell})\right)$, where $T$ is the
tower function defined in the introduction. Let $W=W_{t-1}$ with
$t=2^{-20}\epsilon^{-1}$. Then there is a graph $G$ such that if equitable
partitions $\mathcal{A},\mathcal{B}$ of the vertex set of $G$ satisfy
$q(\mathcal{B}) \leq q(\mathcal{A})+\epsilon$ and $\mathcal{B}$ is
$f(|\mathcal{A}|)$-regular, then $|\mathcal{A}|,|\mathcal{B}| \geq W$.

We next describe the proof of Theorem \ref{stronglow}.  In the first subsection, we construct the graph $G$ as a specialization of the construction in Theorem \ref{maingen}. The graph $G$ we
use to prove Theorem \ref{stronglow} has vertex partitions
$P_{i,j}$ with $1 \leq i \leq t$, and $1 \leq j \leq h_i$ satisfying
$P_{i',j'}$ is a refinement of $P_{i,j}$ if $i'=i$ and $j' > j$ or if $i'>i$.
Furthermore, as the number of parts in each successive refinement is roughly
exponential in the number of parts in the previous partition, we show in the first subsection that $|P_{t-2,h_{t-2}-2}| \geq W$. The edges of $G$ are defined based on certain
graphs $G_{i,j}$ on $P_{i,j}$. In Subsection \ref{msddp}, we prove a lemma
which implies that the construction has the property that \begin{equation}
\label{eq1msd} q(P_{i,h_i})>q(P_{i,h_i-2})+2\epsilon\end{equation} for each
$i<t$.

Let $\mathcal{A}$ and $\mathcal{B}$ be equitable partitions of the vertex set of $G$ such that 
$q(\mathcal{B}) \leq q(\mathcal{A})+\epsilon$ and $\mathcal{B}$ is
$f(|\mathcal{A}|)$-regular. Let $M_1=1$ and $M_{\ell}=|P_{\ell-1,h_{\ell-1}-2}|$ for $1<\ell \leq t-1$. Let $r$ with $1 \leq r \leq t-1$ be maximum such that $|\mathcal{A}| \geq M_r$. Let
$P'=P_{r,h_r-2}$ and $P=P_{r,h_r}$. In Subsection
\ref{sub41}, after defining $G$, we show that it satisfies the hypothesis of
Theorem \ref{maingen}, and conclude that, as $\mathcal{B}$ is an
$f(|\mathcal{A}|)$-regular partition of $G$ and $f$ is a decreasing function, it must be close to being a
refinement of $P$. It follows that if $|\mathcal{A}| \geq M_{t-1}=|P_{t-2,h_{t-2}-2}| > W$, then $|\mathcal{B}| > W$ as well, and we are done in this case. Thus we may assume $|\mathcal{A}|<M_{t-1}$ and hence $r \leq t-2$. 
In Subsection \ref{qmsd} we prove
\begin{equation}\label{eq2msd}
q(\mathcal{A})<q(P')+\frac{\epsilon}{2}.\end{equation} This follows from a
lemma that states that $q(P')$ is close to the maximum mean square density
density over all partitions of the same number of parts as $P'$. In Subsection \ref{sub42}, we use the result that $\mathcal{B}$ is close to being a
refinement of $P$ to conclude \begin{equation}\label{eq3msd} q(P) \leq
q(\mathcal{B})+\frac{\epsilon}{2}.\end{equation}  Putting the three estimates
(\ref{eq1msd}) (with $i=r$ and noting in this case $P_{i,h_i}=P$,
$P_{i,h_i-2}=P'$), (\ref{eq2msd}), (\ref{eq3msd}) together, we get that
$$q(\mathcal{B}) \geq
q(P)-\frac{\epsilon}{2}>q(P')+2\epsilon-\frac{\epsilon}{2}>q(\mathcal{A})+\epsilon,$$
contradicting the hypothesis of Theorem \ref{stronglow}, and completing the
proof of Theorem \ref{stronglow}. \qed

\subsection{Construction of $G$ and proof that $\mathcal{B}$ is an approximate
refinement}\label{sub41}

We will construct the graph $G$ as a special case of the construction in
Theorem \ref{maingen}.

Let $t=2^{-20}\epsilon^{-1}$. We have
partitions $P_{\ell,j}$ of the vertex set $V$ for $1 \leq \ell \leq t$ and $1
\leq j \leq h_{\ell}$, where $h_{\ell}$ is defined later in the paragraph and
$P_{\ell,j}=P_i$ are the partitions used to construct $G$ in Theorem
\ref{maingen} with $i=j+\sum_{d<\ell}h_d$. We set
$m_{\ell,j}=|P_{\ell,j}|=|P_i|=m_i$, and $p_{\ell,j}=p_i$.
As above, let $M_1=1$ and $M_{\ell}=m_{\ell-1,h_{\ell-1}-2}$ for $1<\ell \leq t$. Let
$\epsilon_{\ell}=f(M_{\ell})$,
$h_{\ell}=\frac{\epsilon^{5}}{2^{70}\epsilon_{\ell}}$, and
$p_{\ell,j}=\max(m_{\ell,j}^{-1/10},2^{30}\epsilon^{-4}\epsilon_{\ell})$ for $1
\leq j \leq h_{\ell}$ with
$j \neq h_{\ell}-1$, and
$p_{\ell,h_{\ell}-1}=\max(m_{\ell,j}^{-1/10},2^{30}\epsilon^{-4}\epsilon_{\ell},2^{10}\epsilon)$.

Let $m_1=2^{10}\epsilon^{-2}$, so $m_1 \geq 2^{200}$.  Note that, as each
$m_{\ell,j}$ is exponential in a power of $m_{\ell,j-1}$, we get that
$M_{\ell}$ is at least a tower of $2$s of height $h_{\ell}$. That is, $M_{\ell}
\geq T\left(2^{-70}\epsilon^{5}/f(M_{\ell-1})\right)$. In particular, by
induction, $M_{\ell} \geq W_{\ell}$, where $W_{\ell}$ is defined earlier in
this section.

We will apply Theorem \ref{maingen} to conclude the following corollary which
states that any sufficiently regular partition of $G$ is roughly a refinement
of a particular $P_{\ell,j}$. To accomplish this we need to show that the
conditions of the theorem hold, which we postpone until after stating the
following corollary. We fix $G$ to be a graph satisfying the properties of
Lemma \ref{importantstep} so that if $G$ also satisfies the conditions stated
in Theorem \ref{maingen}, then it satisfies the conclusion of Theorem
\ref{maingen}.

\begin{corollary}\label{corforstr}
Let $r \leq t-1$ be the maximum positive integer for which $|\mathcal{A}| \geq M_{r}$, so $f(|\mathcal{A}|) \leq f(M_r) = \epsilon_r$, and $P=P_{r,h_{r}}$. The partition $\mathcal{B}$, which is $\epsilon_{r}$-regular, is a
$(\beta,\upsilon)$-refinement of $P$ with $\beta=20m_1^{-3/2}$ and
$\upsilon=5m_1^{-1/2}$.
\end{corollary}

 Note that
\begin{eqnarray*}
\nu & = & 3\sum_{i=1}^{s-1} p_i=\sum_{\ell=1}^{t}\sum_{j=1}^{h_{\ell}}
p_{\ell,j}\leq
2^{10}\epsilon t+\sum_{\ell=1}^{t} \sum_{j=1}^{h_{\ell}}
\left(m_{\ell,j}^{-1/10}+2^{30}\epsilon^{-4}\epsilon_{\ell}\right)
\\ & \leq &  2^{10}\epsilon t+\sum_{i=1}^{s-1} m_i^{-1/10}+\sum_{\ell=1}^{t}
\sum_{j=1}^{h_{\ell}} 2^{30}\epsilon^{-4}\epsilon_{\ell} \leq  2^{10}\epsilon
t+\sum_{i=1}^{s-1} m_i^{-1/10}+\sum_{\ell=1}^{t}2^{-40}\epsilon \leq
2^{-9},
\end{eqnarray*}
where we used that the maximum of a set of nonnegative numbers is at most their sum, and substituted in $h_{\ell}=\frac{\epsilon^{5}}{2^{70}\epsilon_{\ell}}$,
$m_1 = 2^{10}\epsilon^{-2} \geq 2^{200}$, $m_{i+1} =m_ia_i \geq
m_i2^{\lfloor 2^{-20} m_i^{9/10} \rfloor}$, and $t=2^{-20}\epsilon^{-1}$.  We thus have $1-2^7\nu
\geq 1/2 \geq \epsilon_r$. Notice if $\eta=\epsilon_{r}=f(M_{r})$, then, for $1 \leq i
\leq r$ and $1 \leq j \leq h_i$, we have $$p_{i,j} \geq
2^{30}\epsilon^{-4}\epsilon_{i}
\geq 2^{30}\epsilon^{-4}\epsilon_{r} = 2^{10}\eta m_1^2,$$
where we used $m_1=2^{10}\epsilon^{-2}$. Since $\beta=20m_1^{-3/2} = 20 \cdot 2^{-15}\epsilon^3$ and
 $f(1) \leq 2^{-100}\epsilon^6$, we have $\delta = \epsilon_{r}=f(M_{r}) \leq
f(M_1)=f(1)<\beta/4$.

By the above estimates, the conditions of Theorem \ref{maingen} are satisfied,
and Corollary \ref{corforstr} stated above indeed holds. \qed 

Note that if $r=t-1$ in Corollary \ref{corforstr}, then $|\mathcal{A}| \geq M_{t-1}=|P_{t-2,h_{t-2}-2}|>W$, and 
$\mathcal{B}$ is a $(\beta,\upsilon)$-refinement of $P=P_{r,h_r}$. 
As $1-\beta>1/2$, this implies $|\mathcal{B}| \geq \frac{1}{2}|P_{r,h_r}| > W$, which
completes the proof of Theorem \ref{stronglow} in this case. We can therefore assume $r<t-1$.  

\subsection{Approximate refinements and mean square density}\label{sub42}

From Corollary \ref{corforstr} and the following lemma, we deduce at the end of
this subsection that
if $P$ is the partition in Corollary \ref{corforstr}, then $q(P) \leq
q(\mathcal{B})+\frac{\epsilon}{2}$.

\begin{lemma}\label{almostref}
Suppose $G$ is a graph, $P$ is a vertex partition, and $Q$ is an equitable
partition which is a $(\beta,\upsilon)$-refinement of $P$. Then $q(P) \leq
q(Q)+2\beta+\frac{1}{2}\upsilon$.
\end{lemma}
\begin{proof}
Let $Q'$ be the common refinement of $P$ and $Q$, so $q(Q') \geq q(P)$. Let
$X,Y \in Q$ be such that $X,Y$ are each $(1-\beta)$-contained in $P$. Let
$X=X_1 \cup \ldots \cup X_r$ be the partition of $X$ consisting of parts from
$Q'$ with $|X_1| \geq (1-\beta)|X|$, and $Y=Y_1 \cup \ldots \cup Y_s$ be the
partition of $Y$ consisting of parts from $Q'$ with $|Y_1| \geq (1-\beta)|Y|$.
Let $p=d(X_1,Y_1)$ and $p'=\frac{1}{1-p_1q_1}\sum d(X_i,Y_j)p_iq_j$, where
$p_i=\frac{|X_i|}{|X|}$, $q_j=\frac{|Y_j|}{|Y|}$ and the sum is over all pairs
$(i,j) \in [r] \times [s]$ except $(i,j)=(1,1)$. That is, $p'$ is the weighted
average edge density between the pairs of parts except $(X_1,Y_1)$. We have
$$\sum_{i=1}^r\sum_{j=1}^s d^2(X_i,Y_j)p_iq_j \leq p^2p_1q_1+\sum_{(i,j) \not =
(1,1)} d(X_i,Y_j)p_iq_j = p_1q_1p^2+p'(1-p_1q_1)$$
and $$d(X,Y)=pp_1q_1+p'(1-p_1q_1).$$
Let $\epsilon=1-p_1q_1$, so \begin{eqnarray*}
\sum_{i=1}^r\sum_{j=1}^s d^2(X_i,Y_j)p_iq_j-d^2(X,Y) & \leq &
(1-\epsilon)p^2+p'\epsilon-\left(p(1-\epsilon)+p'\epsilon\right)^2 \\ & = &
\epsilon \left((1-\epsilon)p^2+p'-2pp'(1-\epsilon)-p'^2\epsilon\right) \\ &
\leq & \epsilon \left((1-\epsilon)p^2+p'-2pp'(1-\epsilon)\right) \\ &  \leq &
\epsilon \\ & \leq & 2\beta.\end{eqnarray*}
The second to last inequality is by noting the right hand side of the third to
last line is linear in $p'$ and must therefore be maximized when $p'=0$ or $1$,
and the last inequality follows from $\epsilon=1-p_1q_1$ and $p_1,q_1 \geq
1-\beta$.

Now for parts $X,Y \in Q$ that are not both $(1-\beta)$-contained in $P$, again
letting $X=X_1 \cup \ldots \cup X_r$ and $Y=Y_1 \cup \ldots \cup Y_s$ be the
partitions of $X$ and $Y$ consisting of parts from $Q'$, and letting $q$ denote
the edge density between $X$ and $Y$, and  $p_i=\frac{|X_i|}{|X|}$,
$q_j=\frac{|Y_j|}{|Y|}$,  we have
$$\sum_{i=1}^r\sum_{j=1}^s d^2(X_i,Y_j)p_iq_j - d^2(X,Y)  \leq q-q^2 \leq
1/4.$$

Since $Q$ is a $(\beta,\upsilon)$-refinement of $P$, at most $a$ 
$2\upsilon$-fraction of the pairs of parts from $Q$ are such that not both
parts are $(1-\beta)$-contained in $P$. Putting together the estimates from the
last two paragraphs, we therefore get $$q(P) \leq q(Q') \leq
q(Q)+2\beta+\frac{1}{4} \cdot 2\upsilon.$$
\end{proof}

Noting that $m_1=2^{10}\epsilon^{-2}$, $\beta=20m_1^{-3/2}  < \epsilon/8$ and
$\upsilon=5m_1^{-1/2} \leq \epsilon/4$ in Corollary \ref{corforstr}, we have
the following corollary of Corollary \ref{corforstr} and Lemma \ref{almostref}.

\begin{corollary}
If $P$ is the partition in Corollary \ref{corforstr}, then $q(P) \leq
q(\mathcal{B})+\frac{\epsilon}{2}$.
\end{corollary}

\subsection{Mean square densities of the defining partitions}\label{msddp}

The next lemma shows that the mean square density of each successive partition
increases by a constant factor of the edge density of each $G_i$.

\begin{lemma}
For each $i$, we have $q(P_{i+1}) \geq q(P_i)+2^{-5}p_i$.
\end{lemma}
\begin{proof}
The fraction of pairs $(X,Y) \in P_i \times P_i$ which are edges of $G_i$ and not edges of $G^i$ is at
least $p_i/4$ by the second property in Lemma \ref{importantstep}. For each
such pair, the equitable partitions $X=X_Y^1 \cup X_Y^2$, $Y=Y_X^1 \cup Y_X^2$
satisfy $d(X_Y^d,Y_X^d)=1$ and $d(X_Y^d,Y_X^{3-d}) \leq \nu \leq 1/4$ for
$d=1,2$. Let $d_1=d(X_Y^1,Y_X^2)$ and $d_2=d(X_Y^2,Y_X^1)$, so
$$\sum_{i=1}^2\sum_{j=1}^2\frac{1}{4}d^2(X_Y^i,Y_X^j)-d^2(X,Y)=\frac{1}{2}+\frac{d_1^2+d_2^2}{4}-\left(\frac{1}{2}+\frac{d_1+d_2}{4}\right)^2 \geq \frac{1}{4}-\frac{(d_1+d_2)}{4}
\geq \frac{1}{8}.$$
As we get this density increment for at least a $p_i/4$-fraction of the pairs
$(X,Y) \in P_i \times P_i$, we get a total density increment of at least
$\frac{1}{8}\frac{p_i}{4}=2^{-5}p_i$.
\end{proof}

We have the following corollary, noting that $p_{r,h_r-1} \geq 2^{10}\epsilon$.

\begin{corollary}
For $P=P_{r,h_r}$ and $P'=P_{r,h_r-2}$, we have
$$q(P)=q(P_{r,h_r}) \geq q(P_{r,h_r-1})+2^{-5}p_{r,h_r-1}
\geq q(P_{r,h_r-1})+2\epsilon \geq q(P')+2\epsilon.$$
\end{corollary}
\subsection{Quasirandomness and mean square density} \label{qmsd}

The goal of this subsection is to show that if $\mathcal{A}$ is a vertex
partition of $G$ with $|\mathcal{A}| \leq |P_i|$, then $q(\mathcal{A})$ is at most $q(P_{i})+p_i$ plus a small error term. To accomplish this, we show that the graphs used to define $G$ are quasirandom with small error.

The study of quasirandom graphs began with the papers by Thomason \cite{Th} and
Chung, Graham, and Wilson \cite{CGW}. They showed that a large number of
interesting graph properties satisfied by random graphs are all equivalent.
These properties are known as {\it quasirandom properties}, and any graph that
has one of these properties (and hence all of these properties) is known as a
{\it quasirandom graph}.

This development was heavily influenced by and closely related to
Szemer\'edi's regularity lemma. Furthermore, all
known proofs of Szemer\'edi's theorem on long arithmetic progressions in dense
subsets of the integers use some notion of quasirandomness. For graphs on $n$
vertices with edge density $p$ bounded away from zero, the following two
properties are quasirandom properties. The first property states that the
number of $4$-cycles (or, equivalently, the number of closed walks of length
$4$) in the graph is $p^4n^4+o(n^4)$. The second property states that all pairs
of vertex subsets $S,T$ have edge density roughly $p$ between them, apart from
$o(n^2)$ edges. This fact, that the number of $4$-cycles in a graph can control
the edge distribution, is quite notable. For our purposes, we will need to show
that the first property implies the second property, with reasonable error
estimates. The now standard proof bounds the second largest (in absolute value)
eigenvalue of the adjacency matrix of the graph, and then applies the expander
mixing lemma, which bounds the edge discrepancy between subsets in terms of the
subset sizes and the second largest eigenvalue.

\begin{lemma}\label{p1p2}
Suppose $G=(V,E)$ is a graph with $n$ vertices and average degree $d$, and the
number of closed walks of length $4$ in $G$ is at most $d^4+\alpha n^4.$ For all vertex subsets $S$ and $T$,
$$|e(S,T)-\frac{d|S||T|}{n}|<\lambda\sqrt{|S||T|},$$ where $\lambda \leq
\alpha^{1/4}n$.
\end{lemma}
\begin{proof}
Let $A$ be the adjacency matrix of $G$, and
$\lambda_1,\lambda_2,\ldots,\lambda_n$ be the eigenvalues of $A$, with
$|\lambda_1| \geq |\lambda_2| \geq \ldots \geq |\lambda_n|$. Let
$\lambda=|\lambda_2|$.
It is easy to check that $\lambda_1 \geq d$. Let $\lambda=|\lambda_2|$. The
number of closed walks of length $4$ in $G$ is equal to the trace
$$Tr(A^4)=\sum_{i=1}^n \lambda_i^4 \geq \lambda_1^4+\lambda^4.$$
As $\lambda_1 \geq d$, and the number of closed walks of length $4$ is at most
$d^4+\alpha n^4$, we conclude $\lambda \leq (\alpha n^4)^{1/4}=\alpha^{1/4}n.$
The expander mixing lemma (see Section 2.4 of \cite{KrSu}) states that for all
vertex subsets $S,T$, we have $|e(S,T)-\frac{d|S||T|}{n}|<\lambda
\sqrt{|S||T|}$. This completes the proof.
\end{proof}

A spanning subgraph of graph $G$ is a subgraph of $G$ on the same vertex set
$V$ as $G$.
We let $H_i$ be the spanning subgraph of $G$ where vertices $u,v \in V$ are
adjacent in $H_i$ if and only if there is an edge $(X,Y)$ of $G_i$, and $j \in
\{1,2\}$ with $u \in X^j_Y, v\in Y^j_X$. Note that the edge set of $G$ is
precisely the union of the edge sets of the $H_i$, although this is likely not
an edge partition. We next use Lemma \ref{p1p2} to show that the edges of $H_i$
are uniformly distributed. That is, the edge density in $H_i$ is roughly the
same between large vertex subsets of $V$.

\begin{lemma}\label{Hiuniform}
Let $|V|=n$. For each $i$, the graph $H_i$ on vertex set $V$ has the property
that for all vertex subsets $S$ and $T$,
$$|e_{H_i}(S,T)-\frac{p_i}{2}|S||T||<2m_i^{-1/80}p_in\sqrt{|S||T|}.$$
\end{lemma}
\begin{proof}
Note that each edge $(X,Y)$ of $G_i$ gives rise to two complete bipartite
graphs, between $X_Y^j$ and $Y_X^j$ with $j \in \{1,2\}$, in $H_i$. In
particular, each such edge of $G_i$ contributes $\frac{n}{2m_i}$ degree in
graph $H_i$ to each vertex in $X$ and in $Y$.

We first give a lower bound on the average degree $d$ in $H_i$. From the first
property in Lemma \ref{importantstep}, every vertex in $G_i$ has degree
differing from $p_im_i$ by at most $m_i^{3/4}$. Hence, every vertex in $H_i$
has degree differing from $p_im_i \frac{n}{2m_i}=p_in/2$ by at most
$m_i^{3/4}\cdot \frac{n}{2m_i}=\frac{1}{2}m_i^{-1/4}n$. Thus, the average
degree $d$ of $H_i$ satisfies $|d-\frac{1}{2}p_i n| \leq
\frac{1}{2}m_i^{-1/4}n$.

We next give an upper bound on the number $W_4$ of labeled closed walks of
length four in $H_i$. By counting over the first and third vertex of the closed
walk, we have $W_4=\sum_{u,v} |N_{H_i}(u,v)|^2$, that is, $W_4$ is the sum of
the squares of the codegrees over all labeled pairs of vertices of $H_i$.
By the first part of Lemma \ref{importantstep}, if $X,Y$ are distinct parts of
partition $P_i$, then the codegree of $X$ and $Y$ in $G_i$ is at most
$p_i^2m_i+m_i^{3/4}$. Hence, from Corollary \ref{firstcor},
if $u$ and $v$  are in different parts in the partition $P_i$, then
$$|N_{H_i}(u,v)| \leq
(\frac{1}{4}+a_{i}^{-1/4})(p_i^2m_i+m_i^{3/4})\frac{n}{m_i}=(\frac{1}{4}+a_{i}^{-1/4})(p_i^2+m_i^{-1/4})n.$$
For each pair of vertices $u,v$ in the same part of $P_i$, we have $u$ and $v$
have the same neighborhood in $H_i$ and in this case we use the trivial
estimate $|N_{H_i}(u,v)| \leq n$.  In total, we get

\begin{eqnarray*} W_4 & = & \sum_{u,v} |N_{H_i}(u,v)|^2 \leq m_i(m_i-1)
\left(\frac{n}{m_i}\right)^2 \cdot
\left((\frac{1}{4}+a_{i}^{-1/4})(p_i^2+m_i^{-1/4})n\right)^2+m_i
\left(\frac{n}{m_i}\right)^2\cdot n^2 \\ & \leq &
\left(1+5m_i^{-1/20}\right)p_i^4n^4/16,\end{eqnarray*}
where we used $p_i \geq m_i^{-1/10}$, $a_i=2^{\lfloor \rho m_i^{9/10} \rfloor}$
with $\rho = 2^{-20}$ and $m_i \geq m_1 \geq 2^{200}$.

Let \begin{eqnarray*} \alpha & = & n^{-4}\left(W_4-d^4\right) \leq
n^{-4}\left(W_4-(1-4m_i^{-1/4}p_i^{-1})p_i^4n^4/16\right) \leq
\left(5m_i^{-1/20}+4m_i^{-1/4}p_i^{-1}\right)p_i^4/16 \\ & \leq &
m_i^{-1/20}p_i^4. \end{eqnarray*}

By the choice of $\alpha$, we have $W_4 = d^4+\alpha n^4$. From Lemma
\ref{p1p2}, we have
$$|e(S,T)-\frac{d|S||T|}{n}|<\alpha^{1/4}n\sqrt{|S||T|}.$$
Substituting in that the average degree $d$ differs from $p_in/2$ by at most
$m_i^{-1/4}n/2$, the bounds
$\alpha^{1/4} \leq m_i^{-1/80}p_i$, $m_i^{-1/4}/2 \leq m_i^{-1/80}p_i$, and
$|S||T| \leq n\sqrt{|S||T|}$, and using the triangle inequality, we have the
desired estimate holds on the number $e_{H_i}(S,T)$ of edges in $H_i$ between
$S$ and $T$.
\end{proof}

We next prove the following lemma which estimates the edge density of $G$
between certain vertex subsets.

\begin{lemma} \label{xyab}
Let $X,Y \in P_i$ be distinct with $(X,Y)$ not an edge of $G^i$. If also
$(X,Y)$ is not an edge of $G_i$ and $A \subset X$, $B \subset Y$, or if $(X,Y)$
is an edge of $G_i$ and there is $j \in \{1,2\}$ such that $A \subset X_Y^j$
and $B \subset Y_X^{3-j}$, then $$\left|d_G(A,B)-\left(1-\prod_{h >
i}\left(1-\frac{p_h}{2}\right)\right)\right|<6m_{i+1}^{-1/80}p_{i+1}\frac{n}{\sqrt{|A||B|}},$$
where $n=|V|$ is the number of vertices of $G$.
\end{lemma}
\begin{proof}
For $i' \geq i$, let $d_{i'}$ denote the density between $A$ and $B$ of the
pairs which are edges of at least one $H_{\ell}$ with $\ell \leq i'$. In
particular, by the choice of $A$ and $B$, no edges of $H_h$ for $h \leq i$ go
between $A$ and $B$, and hence $d_i=0$. Furthermore, we have
$d_{i+1}=d_{H_{i+1}}(A,B)$. By Lemma \ref{Hiuniform}, the number of edges
between $A$ and $B$ in $H_{i+1}$ satisfies
\begin{equation}\label{dtifirst}\left|e_{H_{i+1}}(A,B)-\frac{p_{i+1}}{2}|A||B|\right|\leq
2m_{i+1}^{-1/80}p_{i+1}n\sqrt{|A||B|}.\end{equation}

 Let $t_{i}=1$ and for $i'>i$, let
$t_{i'}=\prod_{i + 1 \leq h \leq i'}\left(1-\frac{p_h}{2}\right)$. We prove by
induction on $i'$ that for each $i' \geq i+1$, we have
\begin{equation}\label{dti}
| d_{i'}-(1-t_{i'})|<q\prod_{h=i+1}^ {i'}(1+p_h),
\end{equation}
where $$q=2m_{i+1}^{-1/80}p_{i+1}\frac{n}{\sqrt{|A||B|}}.$$
In the base case $i'=i+1$, we have the desired estimate (\ref{dti}) from dividing
(\ref{dtifirst}) out by $|A||B|$. So suppose we have established (\ref{dti})
for $i'$, and we next prove it for $i'+1$, completing the proof of (\ref{dti})
by induction.

Let $X',Y' \in P_{i'}$ with $X' \subset X$, $Y' \subset Y$, and $(X',Y')$ not
an edge of $G^{i'}$. If $(X',Y')$ is not an edge of $G_{i'}$, letting $A'=X'
\cap A$ and $B'=Y' \cap B$, or if $(X',Y')$ is an edge of $G_{i'}$, and letting
$j \in \{1,2\}$ and $A'=X_Y'^j \cap A$ and $B'=Y_X'^{3-j} \cap B$, we have
$$\left|e_{H_{i'+1}}(A',B')-\frac{p_{i'+1}}{2}|A'||B'|\right| \leq
2m_{i'+1}^{-1/80}p_{i'+1}n\sqrt{|A'||B'|}.$$

Each such pair $X',Y'$ with $(X',Y')$ not an edge of $G_{i'}$ gives rise to a
pair $(A',B')$, and each such pair with $(X',Y')$ an edge of $G_{i'}$ gives
rise to two pairs $(A',B')$ of this form, one for each $j \in \{1,2\}$. The
number of pairs $(X',Y')$ is
$\left(1-d_{G^{i'}}(X,Y)\right)\left(m_{i'}/m_i\right)^2$. The total number
$\Delta$ of such pairs $(A',B')$ is therefore
$$\Delta=\left(1-d_{G^{i'}}(X,Y)+d_{G_{i'}}(X,Y)\right)(m_{i'}/m_i)^2.$$ On the
other hand, the sum of $|A'||B'|$ over all such pairs is $(1-d_{i'})|A||B|$.
Hence, the average value of $|A'||B'|$ over all such pairs $(A',B')$ is
$(1-d_{i'})|A||B|/\Delta.$

By the triangle inequality, summing over all such pairs $A',B'$, we have the
number of edges $E$ of $H_{i'+1}$ between $A$ and $B$ which are not edges of
any $H_{\ell}$ with $\ell \leq i'$ satisfies
\begin{eqnarray*}\left|E-\frac{p_{i'+1}}{2}(1-d_{i'})|A||B|\right| & \leq &
\sum_{A',B'}2m_{i'+1}^{-1/80}p_{i'+1}n\sqrt{|A'||B'|} \\ & \leq &
2m_{i'+1}^{-1/80}p_{i'+1}n((1-d_{i'})|A||B|)^{1/2}\Delta^{1/2} \\ & \leq &
4m_{i'+1}^{-1/80}p_{i'+1}n(|A||B|)^{1/2}m_{i'}/m_i,\end{eqnarray*}
where we used Jensen's inequality for the concave function $f(y)=y^{1/2}$.

Hence,
\begin{eqnarray*}|d_{i'+1}-(1-t_{i'+1})|& = &
\left|d_{i'}(A,B)+\frac{E}{|A||B|} -(1-t_{i'+1})\right| \leq
 |d_{i'}-(1-t_{i'})|+\left|\frac{E}{|A||B|} -(t_{i'}-t_{i'+1})\right| \\ & = &
 |d_{i'}-(1-t_{i'})|+\left|\frac{E}{|A||B|} -\frac{p_{i'+1}}{2}t_{i'}\right|\\
& \leq &   (1+\frac{p_{i'+1}}{2})|d_{i'}-(1-t_{i'})|+\left|\frac{E}{|A||B|}
-\frac{p_{i'+1}}{2}(1-d_{i'})\right| \\ & \leq &
(1+\frac{p_{i'+1}}{2})|d_{i'}-(1-t_{i'})|+4m_{i'+1}^{-1/80}p_{i'+1}n(|A||B|)^{-1/2}m_{i'}/m_i
\\ & \leq & q\left(1+\frac{p_{i'+1}}{2}\right)\prod_{h=i+1}^{i'}(1+p_h) +
4m_{i'+1}^{-1/80}p_{i'+1}n(|A||B|)^{-1/2}m_{i'}/m_i \\ & \leq &
q\prod_{h=i+1}^{i'+1}(1+p_h),
\end{eqnarray*}
which completes the induction proof of (\ref{dti}).

As $\sum p_h \leq 1$, we have $\prod (1+p_h) \leq e$. From (\ref{dti}) with
$i'=s-1$, we get
$$\left|d_G(A,B)-\left(1-\prod_{h>i}(1-p_h)\right)\right|=|d_{s-1}-(1-t_{s-1})|<q\prod_{h=i+1}^{s-1}(1+p_h)<6m_{i+1}^{-1/80}p_{i+1}\frac{n}{\sqrt{|A||B|}},$$
which completes the proof.
\end{proof}

The following lemma is the main result in this subsection, showing that
$q(\mathcal{A})-q(P')$ is small, where the mean square densities are with
respect to the graph $G$.

\begin{lemma}
For $P'=P_{r,h_r-2}$, we have $q(\mathcal{A}) \leq q(P')+\frac{\epsilon}{2}$.
\end{lemma}
\begin{proof}
Consider the partition $\mathcal{A}'$ which is the common refinement of $P'$
and $\mathcal{A}$. The number of parts of $\mathcal{A}'$ is at most
$|P'||\mathcal{A}| \leq |P'|^2$, and each part of $P'$ is refined into at most
$|\mathcal{A}| \leq |P'|$ parts of $\mathcal{A}'$. Let $i$ be such that $P_i=P_{r,h_r-2}=P'$.
As $\mathcal{A}'$ is a refinement of $P'$, in $H_j$ for each $j<i$ between each pair of parts of
$\mathcal{A}'$ the edge density is $0$ or $1$.
Noting that $\mathcal{A}'$ is a
refinement of $\mathcal{A}$, we have
\begin{equation}\label{qapi}q(\mathcal{A})-q(P_i)\leq q(\mathcal{A}')-q(P_i) =
\sum_{X,Y \in P_i} m_i^{-2}\sum_{A,B \subset \mathcal{A}', A \subset X, B
\subset Y} \frac{|A||B|}{|X||Y|}\left(d^2(A,B)-d^2(X,Y)\right).\end{equation}

Note that the summand in the above sum if $(X,Y)$ is an edge of $G^i$ is $0$ as
in this case  $d(A,B)=d(X,Y)=1$. We have $d^2(A,B)-d^2(X,Y) \leq 1$ for $(X,Y)$
an edge of $G_i$, and the fraction of pairs $(X,Y)$ which are edges of $G_i$ is
at most $p_i+m_i^{-1/4}$.

For a pair $X,Y \in P_i$ with $(X,Y)$ not an edge of $G^i$ or $G_i$, $A,B
\subset \mathcal{A}'$ with $A \subset X$ and $B \subset Y$, we have by Lemma
\ref{xyab} and the triangle inequality that
\begin{equation}\label{dabdxy}|d(A,B)-d(X,Y)| \leq 2 \cdot
6m_{i+1}^{-1/80}p_{i+1}\frac{n}{\sqrt{|A||B|}}.\end{equation}
Summing over all parts $A,B$ of $\mathcal{A}'$  with $A \subset X$ and $B
\subset Y$, we have
\begin{eqnarray*}\sum_{A,B \subset \mathcal{A}', A \subset X, B \subset Y}
|A||B|\left(d^2(A,B)-d^2(X,Y)\right) & \leq & \sum_{A,B \subset
\mathcal{A}', A \subset X, B \subset Y} |A||B|2\left|d(A,B)-d(X,Y)\right| \\ &
\leq & \sum_{A,B \subset \mathcal{A}', A \subset X, B \subset
Y}24m_{i+1}^{-1/80}p_{i+1}n\sqrt{|A||B|} \\ & \leq &
24m_{i+1}^{-1/80}p_{i+1}n^2,
\end{eqnarray*}
where the first inequality follows from $a^2-b^2=(a+b)(a-b) \leq 2(a-b)$ for $0
\leq a,b \leq 1$, the second inequality is by (\ref{dabdxy}), and the last
inequality is by using the Cauchy-Schwarz inequality, noting that $$\sum_{A,B
\subset \mathcal{A}', A \subset X, B \subset Y} |A||B|=|X||Y|=(n/m_i)^2,$$ and
the number of pairs $A, B \subset \mathcal{A}'$ satisfying  $A \subset X, B
\subset Y$ is at most $m_i^2$.

Dividing out by $|X||Y|=(n/m_i)^2$, we have,
\begin{equation}\label{frest}\sum_{A,B \subset \mathcal{A}', A \subset X, B
\subset Y} \frac{|A||B|}{|X||Y|}\left(d^2(A,B)-d^2(X,Y)\right) \leq
24m_{i+1}^{-1/80}p_{i+1}m_i^2.
 \end{equation}

From the estimate (\ref{frest}), we have from (\ref{qapi}) that
\begin{equation} q(\mathcal{A})-q(P')\leq
p_i+m_i^{-1/4}+24m_{i+1}^{-1/80}p_{i+1}m_i^2 \leq 3p_i \leq
\frac{\epsilon}{2},
\end{equation}
where we used $$p_i=\max(m_i^{-1/10},2^{30}\epsilon^{-4}\epsilon_r),$$
$$\epsilon_r \leq \epsilon_1= f(1)=2^{-100}\epsilon^6,$$ $$i \geq h_1-2 =
\frac{\epsilon^5}{2^{70}\epsilon_1}-2 \geq 2^{29}\epsilon^{-1}$$ and hence $m_i
\geq (6/\epsilon)^{10}$. This completes the proof.

\end{proof}

\section{Induced graph removal lemma} \label{indremovalsection}

The induced graph removal lemma states that for any fixed graph $H$ on $h$
vertices and $\epsilon>0$, there is $\delta=\delta(\epsilon,H)>0$ such that if
a graph $G$ on $n$ vertices has at most $\delta n^h$ induced copies of $H$,
then we can add or delete $\epsilon n^2$ edges of $G$ to obtain an induced
$H$-free graph. The main goal of this section is to prove Theorem
\ref{inducedtower}, which gives a bound on $\delta^{-1}$ which is a tower in $h$ of height polynomial in $\epsilon^{-1}$. We in fact prove the key corollary of the strong regularity lemma, Lemma \ref{strongeasycor}, with a tower-type bound.
This is sufficient to prove the desired tower-type bound for the induced graph removal lemma.

We first use the weak regularity lemma of Duke, Lefmann, and R\"odl to find a
large subset of a graph which is $\epsilon$-regular with itself.  By
iteratively pulling out such subsets and redistributing the set of leftover
vertices, we obtain a partition of any vertex subset into large vertex subsets
each of which is $\epsilon$-regular with itself. Then, in Subsection
\ref{inducedsub3}, we establish Lemma \ref{scrl}, the strong cylinder
regularity lemma, with a tower-type bound. We show in Subsection
\ref{inducedsub4} that the strong cylinder regularity lemma implies the key
corollary of the strong regularity lemma, Lemma \ref{strongeasycor}, with a
tower-type bound. This in turn implies Theorem \ref{inducedtower}.

 In this section and the next, we call a pair $(A,B)$ of vertex subsets of a graph {\it
$\epsilon$-regular} if for all $A' \subset A$ and $B' \subset B$ with $|A'|
\geq \epsilon|A|$ and $|B'| \geq \epsilon |B|$, we have
$\left|d(A',B')-d(A,B)\right| \leq \epsilon$.

\subsection{The Duke-Lefmann-R\"odl regularity lemma} \label{inducedsub2}

Given a $k$-partite graph $G=(V,E)$ with $k$-partition $V=V_1 \cup \ldots \cup
V_k$, we will consider a partition $\mathcal{K}$ of the cylinder $V_1 \times
\cdots \times V_k$ into cylinders $K=W_1 \times \cdots \times W_k$, $W_i
\subset V_i$ for $i=1,\ldots,k$, and we let $V_i(K)=W_i$.  Recall from the
introduction that a cylinder is $\epsilon$-regular if all the ${k \choose 2}$
pairs of subsets $(W_i,W_j)$, $1 \leq i < j \leq k$, are $\epsilon$-regular.
The partition $\mathcal{K}$ is $\epsilon$-regular if all but an
$\epsilon$-fraction of the $k$-tuples $(v_1,\ldots,v_k) \in V_1 \times \cdots
\times V_k$ are in $\epsilon$-regular cylinders in the partition $\mathcal{K}$.

The weak regularity lemma of Duke, Lefmann, and R\"odl \cite{DLR} is now as
follows. Note that, like the Frieze-Kannan weak regularity lemma, it has only a
single-exponential bound on the number of parts, which is much better than the
tower-type bound on the number of parts in Szemer\'edi's regularity lemma.
Duke, Lefmann, and R\"odl \cite{DLR} used their regularity lemma to derive a
fast approximation algorithm for the number of copies of a fixed graph in a
large graph.

\begin{lemma}\label{dukelefrod}
Let $0<\epsilon<1/2$ and $\beta=\beta(\epsilon)=\epsilon^{k^2\epsilon^{-5}}$.
Suppose $G=(V, E)$ is a $k$-partite graph with $k$-partition $V=V_1 \cup \ldots
\cup V_k$. Then there exists an $\epsilon$-regular partition $\mathcal{K}$ of
$V_1 \times \cdots \times V_k$ into at most $\beta^{-1}$ parts such that, for
each $K \in \mathcal{K}$ and $1 \leq i \leq k$, we have $|V_i(K)| \geq
\beta|V_i|$.
\end{lemma}

\subsection{Finding an $\epsilon$-regular subset}

For a graph $G=(V,E)$, a vertex subset $U \subset V$ is {\it
$\epsilon$-regular} if the pair $(U,U)$ is $\epsilon$-regular. The following
lemma demonstrates that any graph contains a large vertex subset which is
$\epsilon$-regular.

\begin{lemma}\label{oneepsilonregularsubset}
For each $0<\epsilon<1/2$, let
$\delta=\delta(\epsilon)=2^{-\epsilon^{-(10/\epsilon)^{4}}}$. Every graph
$G=(V,E)$ contains an $\epsilon$-regular vertex subset $U$ with $|U| \geq
\delta |V|$.
\end{lemma}

Lemma \ref{dukelefrod} implies that each $k$-partite graph $G=(V,E)$ with
$k$-partition $V=V_1 \cup \ldots \cup V_k$ has a cylinder $K$ which is
$\epsilon$-regular in which each part has size $|V_i(K)| \geq
\epsilon^{k^2\epsilon^{-5}}|V_i|$. The proof can be easily modified to show
that if each part of $G$ has the same size, then each
part of the $\epsilon$-regular cylinder $K$ has equal size, which is at least
$\epsilon^{k^2\epsilon^{-5}}|V_i|$. This implies that for any graph $G=(V,E)$,
if $G$ has at least $k$ vertices, by considering any $k$ vertex disjoint
subsets of equal size $\lfloor |G|/k \rfloor \geq |G|/(2k)$, and then applying
this result, we get the following lemma.

\begin{lemma}\label{onecyl} For each $0<\epsilon<1/2$, any graph $G=(V,E)$ on at least $k$ vertices has an
$\epsilon$-regular $k$-cylinder with parts of equal size, which is at least
$\frac{1}{2k}\epsilon^{k^2\epsilon^{-5}}|V|$.
\end{lemma}

The {\it $t$-color Ramsey number} $r_t(s)$ is the minimum $k$ such that every
$t$-coloring of the edges of the complete graph $K_k$ on $k$ vertices contains
a monochromatic clique of order $s$. A simple pigeonhole argument (see
\cite{GRS}) gives $r_t(s) \leq t^{ts}$ for $t \geq 2$.

\begin{lemma}\label{steps} For integers $s,t \geq 2$, let $k=t^{ts}$. Let
$G=(V,E)$ be a graph on at least $k$ vertices, and $0<\alpha<1/2$.
The graph $G$ contains an $\alpha$-regular $s$-cylinder with parts of equal
size at least $N=\frac{1}{2k}\alpha^{k^2\alpha^{-5}}|V|$ such that the density
between each pair of parts differs by at most $1/t$.
\end{lemma}
\begin{proof}
Note that $k=t^{ts} \geq r_t(s)$. By Lemma \ref{onecyl}, $G$ contains an
$\alpha$-regular $k$-cylinder $U_1 \times \cdots \times U_k$ with parts of equal size at least $N$. Partition
the unit interval  $[0,1]=I_1 \cup \ldots \cup I_t$ into $t$ intervals of
length $1/t$. Consider the edge-coloring of the complete graph on $k$ vertices $1,\ldots,k$ for
which the color of edge $(i,j)$ is the number $a$ for which the density
$d(U_i,U_j) \in I_a$. Since $k\geq r_t(s)$, there is a monochromatic clique of
order $s$ in this $t$-coloring, and the corresponding parts form the desired
$s$-cylinder.
\end{proof}

\begin{lemma}\label{sunioncyl}
Suppose $\alpha \leq 1/9$ and $U_1 \times \cdots \times U_s$ is an
$\alpha$-regular cylinder in a graph $G=(V,E)$ with $s \geq 2\alpha^{-1}$
parts $U_i$ of equal size and the densities between the pairs of distinct parts
lie in an interval of length at most $\alpha$. Then the set $U=U_1 \cup \ldots
\cup U_s$ is $\epsilon$-regular with itself, where $\epsilon=3\alpha^{1/2}$.
\end{lemma}
\begin{proof}
Let $A,B \subset U$ with $|A|,|B| \geq \epsilon|U|$ and, for $1 \leq i \leq s$,
let $A_i=A \cap U_i$ and $B_i=B \cap U_i$.
Suppose $d(U_i,U_j) \in [\gamma,\gamma+\alpha]$ for $1 \leq i < j \leq s$. Let
$A^1$ be the union of all $A_i$ for which $|A_i| \geq \alpha|U_i|$ and $A^2=A \setminus A^1$. Similarly, let
$B^1$ be the union of all $B_i$ for which $|B_i| \geq \alpha|U_i|$ and $B^2=B \setminus B^1$. We have
$|A^2| < \alpha|U| \leq \alpha \epsilon^{-1}|A|$ and $|B^2| < \alpha|U| \leq
\alpha \epsilon^{-1}|B|$.

Let $I_1$ denote the set of all pairs $(i,i)$ with $i \in [s]$, $I_2$ the set
of all pairs $(i,j) \in [s] \times [s]$ with $i \not = j$, $A_i \subset A^1$, and
$B_j \subset B^1$, and $I_3=[s] \times [s] \setminus (I_1 \cup I_2)$. Let
$D(A_i,B_j)=\left|d(A_i,B_j)-\gamma\right|\frac{|A_i||B_j|}{|A||B|}$. We have
$$\sum_{(i,i) \in I_1} D(A_i,B_i) \leq \sum_{(i,i) \in I_1}
\frac{|A_i||B_i|}{|A||B|} \leq \max_i \frac{|B_i|}{|B|} \leq \max_i
\frac{|U_i|}{|B|} \leq \frac{1}{s\epsilon}.$$
If $(i,j) \in I_2$, using the triangle inequality and $\alpha$-regularity,
$$|d(A_i,B_j)-\gamma| \leq |d(A_i,B_j)-d(U_i,U_j)|+|d(U_i,U_j)-\gamma| \leq
\alpha+\alpha=2\alpha.$$
Hence, $$\sum_{(i,j) \in I_2} D(A_i,B_j) \leq \sum_{(i,j) \in I_2}
2\alpha\frac{|A_i||B_j|}{|A||B|} \leq 2\alpha.$$
Finally, $$\sum_{(i,j) \in I_3} D(A_i,B_j) \leq \sum_{(i,j) \in I_3}
\frac{|A_i||B_j|}{|A||B|} \leq
1-\left(1-\frac{|A^2|}{|A|}\right)\left(1-\frac{|B^2|}{|B|}\right)<1-(1-\alpha\epsilon^{-1})^2
\leq 2\alpha\epsilon^{-1}.$$

We have by the triangle inequality \begin{eqnarray*} \left|d(A,B)-\gamma\right|
& \leq & \sum_{1 \leq i,j \leq s}D(A_i,B_j) \\ &
= &   \sum_{(i,j) \in I_1}D(A_i,B_j)+\sum_{(i,j) \in I_2}D(A_i,B_j)+\sum_{(i,j)
\in I_3}D(A_i,B_j) \\ & \leq & \frac{1}{s\epsilon}+2\alpha+2\alpha\epsilon^{-1}
\leq \frac{\epsilon}{2}.
\end{eqnarray*}
By the triangle inequality, for any $A,B,X,Y \subset U$ each of cardinality at
least $\epsilon|U|$, we have $$\left|d(A,B)-d(X,Y)\right|\leq
\left|d(A,B)-\gamma\right|+\left|\gamma-d(X,Y)\right| \leq
\frac{\epsilon}{2}+\frac{\epsilon}{2}=\epsilon.$$ In particular, this holds for
$X=Y=U$, and hence $U$ is $\epsilon$-regular.
\end{proof}

By applying Lemma \ref{steps} with $\alpha=(\epsilon/3)^2$ and $s=t=\lceil
2\alpha^{-1} \rceil$, and then applying Lemma \ref{sunioncyl},  we obtain Lemma
\ref{oneepsilonregularsubset}. Note that the proof assumes that the number of vertices of the graph is sufficiently large, at least $k=t^{ts}$, but we can make this assumption as otherwise we can trivially pick $U$ to consist of a single vertex, which is $\epsilon$-regular.

The next lemma shows that if we have an $\epsilon$-regular pair, and add a
small fraction of vertices to one part, then the pair is still regular, but
with a slightly worse regularity.

\begin{lemma}\label{partfoura}
Suppose $A$ and $B$ are vertex subsets of a graph $G$ which form an
$\epsilon$-regular pair, and $C$ is a vertex subset disjoint from $B$
of size $|C| \leq \beta|B|$. Then the pair $(A,B \cup C)$ is $\alpha$-regular
with $\alpha=\epsilon+\sqrt{\beta}+\beta$.
\end{lemma}
\begin{proof}
Let $A' \subset A$ and $B' \cup C' \subset B \cup C$ with $B' \subset B$, $C'
\subset C$, $|A'| \geq \alpha|A|$ and $|B'\cup C'| \geq \alpha|B \cup C|$. Note
that $|A'| \geq \alpha |A| \geq \epsilon|A|$ and $|B'|=|B' \cup C'|-|C'| \geq
\alpha|B \cup C|-|C| \geq (\alpha-\beta)|B| \geq \epsilon|B|$. Since the pair
$(A,B)$ is $\epsilon$-regular, we have $$\left | d(A',B')-d(A,B) \right | \leq
\epsilon.$$

Also, $|C'| \leq |C| \leq \beta |B| \leq \beta|B \cup C| \leq
\beta\alpha^{-1}|B' \cup C'|$. Therefore,
 \begin{eqnarray*}|d(A',B' \cup C')-d(A',B')| & = &
\left|d(A',B')\frac{|B'|}{|B' \cup C'|}+d(A',C')\frac{|C'|}{|B' \cup
C'|}-d(A',B')\right|\\ & = &  \left|d(A',C')-d(A',B')\right| \frac{|C'|}{|B'
\cup C'|}\leq \frac{|C'|}{|B' \cup C'|} \leq \beta\alpha^{-1}.
\end{eqnarray*}
We similarly have $$|d(A,B \cup
C)-d(A,B)|=\left|d(A,C)-d(A,B)\right|\frac{|C|}{|B \cup C|}  \leq \beta.$$
Hence, by the triangle inequality, we have $\left|d(A',B' \cup C')-d(A,B \cup
C)\right|$ is at most \begin{eqnarray*}  \left|d(A',B' \cup
C')-d(A',B')\right|+\left|d(A',B')-d(A,B)\right|+\left|d(A,B)-d(A,B \cup
C)\right|  & \leq &  \beta\alpha^{-1}+\epsilon+\beta \\ & \leq &
\alpha.\end{eqnarray*}
Hence, the pair $(A,B \cup C)$ is $\alpha$-regular.
\end{proof}

By repeatedly pulling out $3\epsilon/4$-regular sets using Lemma
\ref{oneepsilonregularsubset} until there are at most
$\frac{\epsilon^2}{100}|V|$ remaining vertices, distributing the remaining
vertices among the parts, and using Lemma \ref{partfoura} twice in each part,  we
arrive at the following lemma. It shows how to partition a graph into large
parts, each part being $\epsilon$-regular with itself.

\begin{lemma}\label{epsdeltaone}
For each $0<\epsilon<1/2$, let
$\delta=\delta(\epsilon)=2^{-\epsilon^{-(20/\epsilon)^{4}}}$. Every graph
$G=(V,E)$ has a vertex partition $V=V_1 \cup \ldots \cup V_k$ such that for
each $i$, $1 \leq i \leq k$, we have $|V_i| \geq \delta|V|$ and $V_i$ is an
$\epsilon$-regular set.
\end{lemma}

\subsection{Tools}

In this subsection, we prove two simple lemmas concerning mean square density
which will be useful in establishing and using the strong cylinder regularity
lemma.

The first lemma, which is rather standard, shows that for any vertex partition
$P$, there is a vertex equipartition $P'$ with a similar number of parts to $P$
and mean square density not much smaller than the mean square density of $P$.
It is useful in density increment arguments where at each stage one would like
to work with an equipartition.

\begin{lemma}\label{makeequip} Let $G=(V,E)$ be a graph, and $P:V=V_1 \cup
\ldots \cup V_k$ be a vertex partition into $k$ parts.  There is an equitable
partition $P'$ of $V$ into $t$ parts such that $q(P') \geq q(P)-2\frac{k}{t}$.

\end{lemma}
\begin{proof}
For an equipartition of $V$ into $t$ parts, we have a certain number of parts
of order $\lfloor |V|/t \rfloor$ and the remaining parts are of order $\lceil
|V|/t \rceil$.
For each part $V_i \in P$, partition it into parts of order $\lfloor |V|/t
\rfloor$ or $\lceil |V|/t \rceil$ so that there are not too many parts of either
order to allow an equipartition of the whole set, with possibly one remaining set of cardinality less than $|V|/t$. Let $Q$ be this refinement of $P$. From the Cauchy-Schwarz inequality, it follows that $q(Q) \geq q(P)$. Let $U$ be the vertices in the remaining parts of
$Q$, so $|U| < k|V|/t$.

Arbitrarily chop the vertices of $U$ into parts of the desired orders so as to
obtain an equipartition $P'$.
We have $$q(P') \geq \sum_{X,Y \in Q, X,Y \subset V \setminus U}
d^2(X,Y)\frac{|X||Y|}{|V|^2}\geq
q(Q)-\left(1-\left(1-\frac{|U|}{|V|}\right)^2\right) \geq q(Q)-2\frac{k}{t}
\geq q(P)-2\frac{k}{t}.$$
\end{proof}

The next lemma is helpful in deducing the induced graph removal lemma from the strong cylinder regularity lemma. 
Let $G=(V,E)$ and $P:V=V_1 \cup \ldots \cup V_k$ be an equipartition, and
$\mathcal{K}$ be a partition of the cylinder $V_1 \times \cdots \times V_k$
into cylinders. For $K=W_1 \times \cdots \times W_k \in \mathcal{K}$, define
the density $d(K)=\frac{|W_1| \times \cdots \times |W_k|}{|V_1| \times \cdots
\times |V_k|}$. The cylinder $K$ is {\it $\epsilon$-close} to $P$ if
$\left|d(W_i,W_j)-d(V_i,V_j)\right| \leq \epsilon$ for all but at most $\epsilon k^2$ pairs $1 \leq i \not = j \leq k$. if cylinder $K$ is not $\epsilon$-close to $P$, then 
$$\sum_{1 \leq i \not =  j \leq k} \left|d(W_i,W_j)-d(V_i,V_j)\right| > \epsilon^2
k^2.$$
The cylinder partition $\mathcal{K}$ is {\it $\epsilon$-close} to $P$ if  $\sum
d(K) \leq \epsilon$, where the sum is over all $K \in \mathcal{K}$ that are not
$\epsilon$-close to $P$.
Note that if $\mathcal{K}$ is not $\epsilon$-close, then
$$\sum_{K \in \mathcal{K}}\sum_{1 \leq i \not = j \leq k}
\left|d(W_i,W_j)-d(V_i,V_j)\right|d(K) > \epsilon^3 k^2.$$
Recall that $Q(\mathcal{K})$ is the common refinement of all the parts $V_i(K)$ with $i \in [k]$ and $K \in \mathcal{K}$.

\begin{lemma}\label{cylinderclose}
Let $G=(V,E)$ and $P:V=V_1 \cup \ldots \cup V_k$ be an equipartition with no
empty parts, i.e., $|V| \geq k$. Let $\mathcal{K}$ be a partition of the
cylinder $V_1 \times \cdots \times V_k$ into cylinders. If $Q=Q(\mathcal{K})$
satisfies $q(Q) \leq q(P)+\epsilon$, then $\mathcal{K}$ is $2^{1/3}
\epsilon^{1/6}$-close to $P$.
\end{lemma}
\begin{proof}
 Let $Q_i$ denote the partition of $V_i$ which is the restriction of partition
$Q$ to $V_i$.

Since $P$ is an equipartition and $|V| \geq k$, then all parts have order at
least $\left \lfloor \frac{|V|}{k} \right \rfloor \geq \frac{|V|}{2k}$.
Therefore, \begin{equation}\label{QPlow}\epsilon \geq q(Q)-q(P)=\sum_{1 \leq
i,j \leq k}\left(q(Q_i,Q_j)-q(V_i,V_j)\right)\frac{|V_i||V_j|}{|V|^2} \geq
\frac{1}{4k^2}\sum_{1 \leq i \not = j \leq
k}\left(q(Q_i,Q_j)-q(V_i,V_j)\right),\end{equation}
where $q(Q_i,Q_j)=\sum_{A \in Q_i,B \in Q_j}d^2(A,B)p_A p_B$ with $p_A =
\frac{|A|}{|V_i|}$ and $p_B = \frac{|B|}{|V_j|}$, and
$q(V_i,V_j)=d^2(V_i,V_j)$.

Fix for now $1 \leq i \not = j \leq k$. For $K=W_1 \times \cdots \times W_k \in
\mathcal{K}$, we have $$d(W_i,W_j)=\sum
d(A,B)\frac{|A|}{|W_i|}\frac{|B|}{|W_j|},$$
and hence, by the triangle inequality,
$$|d(W_i,W_j)-d(V_i,V_j)| \leq  \sum
|d(A,B)-d(V_i,V_j)|\frac{|A|}{|W_i|}\frac{|B|}{|W_j|},$$
where the sums are over all $A \in Q_i$ with $A \subset W_i$ and $B \in Q_j$
with $B \subset W_j$.
Summing over all $K \in \mathcal{K}$, we have,
\begin{eqnarray*} \sum_{K=W_1 \times \cdots W_k \in \mathcal{K}}
\left|d(W_i,W_j)-d(V_i,V_j)\right|d(K) &  \leq & \sum_{K=W_1 \times \cdots W_k
\in \mathcal{K}} \sum
\left|d(A,B)-d(V_i,V_j)\right|\frac{|A|}{|W_i|}\frac{|B|}{|W_j|}d(K)
\\ & = & \sum_{A \in Q_i,B \in Q_j} \left|d(A,B)-d(V_i,V_j)\right|p_{A}p_{B} \\
& \leq & \left(\sum_{A \in Q_i,B \in Q_j} \left(d(A,B)-d(V_i,V_j)\right)^2
p_{A}p_{B}\right)^{1/2} \\ & = &
\left(q(Q_i,Q_j)-q(V_i,V_j)\right)^{1/2}.\end{eqnarray*}
where the first equality follows by swapping the order of summation and the
last inequality is the Cauchy-Schwarz inequality.

Summing over all $1 \leq i \not = j \leq k$ and changing the order of summation,
\begin{eqnarray*}\sum_{K=W_1 \times \cdots W_k \in \mathcal{K}} \, \sum_{1 \leq
i \not = j \leq k}  \left|d(W_i,W_j)-d(V_i,V_j)\right|d(K) & \leq &  \sum_{1 \leq i \not = j
\leq k} \left(q(Q_i,Q_j)-q(V_i,V_j)\right)^{1/2} \\ & \leq & \left(k^2\sum_{1
\leq i \not = j \leq k} q(Q_i,Q_j)-q(V_i,V_j)\right)^{1/2} \\ & \leq & \sqrt{k^2 \cdot
4k^2\epsilon}=2\epsilon^{1/2} k^2,\end{eqnarray*}
where the second inequality is the Cauchy-Schwarz inequality and the last
inequality uses the estimate (\ref{QPlow}).
By the remark before the lemma, we get that $\mathcal{K}$ is
$\left(2\epsilon^{1/2}\right)^{1/3}=2^{1/3}\epsilon^{1/6}$-close to $P$.
\end{proof}

\subsection{The strong cylinder regularity lemma} \label{inducedsub3}

Using the lemmas established in the previous subsections, in this subsection we
prove Lemma \ref{scrl}, the strong cylinder regularity lemma, with a tower-type
bound.

Recall that a $k$-cylinder $W_1 \times \cdots \times W_k$ is strongly
$\epsilon$-regular if all pairs $(W_i,W_j)$ with $1 \leq i,j \leq k$ are
$\epsilon$-regular. A partition $\mathcal{K}$ of $V_1 \times \cdots \times V_k$
into cylinders is strongly $\epsilon$-regular if all but $\epsilon|V_1| \times
\cdots \times |V_k|$ vertices $(v_1,\ldots,v_k) \in V_1 \times \cdots \times
V_k$ are contained in strongly $\epsilon$-regular cylinders $K \in
\mathcal{K}$.

We recall the statement of the strong cylinder regularity lemma.
\begin{lemma} \label{scr2}
For $0<\epsilon<1/3$, positive integer $s$, and decreasing function $f:\mathbb{N} \rightarrow
(0,\epsilon]$, there is $S=S(\epsilon,s,f)$ such that the following holds. For every
graph $G$, there is an integer $s \leq k \leq S$, an equitable  partition $P:V=V_1 \cup
\ldots \cup V_k$ and a strongly $f(k)$-regular partition $\mathcal{K}$ of the
cylinder $V_1 \times \cdots \times V_k$ into cylinders satisfying that the
partition $Q=Q(\mathcal{K})$ of $V$ has at most $S$ parts and $q(Q) \leq
q(P)+\epsilon$. Furthermore, there is an absolute constant $c$ such that
letting $s_1=s$ and $s_{i+1}=t_4\left(\left(s_i/f(s_i)\right)^c\right)$, we may
take $S=s_{\ell}$ with $\ell=2\epsilon^{-1}+1$.
\end{lemma}
\begin{proof}
We may assume $|V| \geq S$, as otherwise we can let $P$ and $Q$ be the trivial
partitions into singletons, and it is easy to see the lemma holds.
We will define a sequence of partitions $P_1,P_2,\ldots$ of equitable partitions,
with $P_{j+1}$ a refinement of $P_j$ and $q(P_{j+1}) > q(P_j)+\epsilon/2$. Let
$P_1$ be an arbitrary equitable partition of $V$ consisting of $s_1=s$ parts. Suppose we
have already found an equitable partition $P_j:V=V_1 \cup \ldots \cup V_k$ with
$k \leq s_{j}$.

Let $\beta(x,\ell)=x^{\ell^2x^{-5}}$ as in Lemma \ref{dukelefrod} and
$\delta(x)=2^{-x^{-(20/x)^4}}$ as in Lemma \ref{epsdeltaone}. We apply Lemma
\ref{epsdeltaone} to each part $V_i$ of the partition $P_j$ to get a partition
of each part $V_i=V_{i1} \cup \ldots \cup V_{ih_i}$ of $P_i$ into parts each of
cardinality at least $\delta|V_i|$, where $\delta=\delta(\gamma)$ and
$\gamma=f(k) \cdot \beta$ with $\beta=\beta(f(k),k)$, such that each part $V_{ih}$ is
$\gamma$-regular. Note that $\delta^{-1}$ is at most triple-exponential in a
polynomial in $k/f(k)$. For each $k$-tuple $\ell=(\ell_1,\ldots,\ell_k) \in
[h_1] \times \cdots \times [h_k]$,  by Lemma \ref{dukelefrod} there is an
$f(k)$-regular partition $\mathcal{K}_{\ell}$ of the cylinder $V_{1\ell_1}
\times \cdots \times V_{k\ell_k}$ into at most $\beta^{-1}$ cylinders such
that, for each $K \in \mathcal{K}_{\ell}$, $|V_{i\ell_i}(K)| \geq \beta
|V_{i\ell_i}|$.   The union of the $\mathcal{K}_{\ell}$ forms a partition
$\mathcal{K}$  of $V_1 \times \cdots \times V_k$ which is strongly
$f(k)$-regular.

Recall that $Q=Q(\mathcal{K})$ is the partition of $V$ which is the common refinement of all
parts $V_i(K)$ with $i \in [k]$ and $K \in \mathcal{K}$. The number of parts of
$\mathcal{K}$ is at most $\delta^{-k}\beta^{-1}$, and hence the number of parts
of $Q$ is at most $k2^{1/(\delta^k \beta)}$. Thus, the number of parts of $Q$
is at most quadruple-exponential in a polynomial in $k/f(k)$. Let $P_{j+1}$ be
an equitable partition into $4\epsilon^{-1}|Q|$ parts with
$q(P_{j+1}) \geq q(Q)-\frac{\epsilon}{2}$, which exists by Lemma
\ref{makeequip}. Hence, there is an absolute constant $c$ such that
$$|P_{j+1}| \leq t_4\left((k/f(k))^c\right) \leq s_{j+1}.$$

If $q(Q) \leq q(P_j)+\epsilon$, then we may take $P=P_j$ and
$Q=Q(\mathcal{K})$, and these partitions satisfy the desired properties.
Otherwise, $q(P_{j+1}) \geq q(Q)-\frac{\epsilon}{2} >
q(P_j)+\frac{\epsilon}{2}$, and we continue the sequence of partitions. Since
$q(P_1) \geq 0$, and the mean square density goes up by more than $\epsilon/2$
at each step and is always at most $1$, this process must stop within
$2/\epsilon$ steps, and we obtain the desired partitions.
\end{proof}

\subsection{A tower-type bound for the key corollary} \label{inducedsub4}

In the previous subsection, we established the strong cylinder regularity lemma
with a tower-type bound. We next use this result to deduce a tower-type bound
for Lemma \ref{strongeasycor}, which is the key corollary of the strong
regularity lemma, and easily implies the induced graph removal lemma as shown
below. We recall the statement of Lemma \ref{strongeasycor} below.

\begin{lemma} \label{easycor2}
For each $0 < \epsilon < 1/3$ and decreasing function $f:\mathbb{N}\rightarrow
(0,\epsilon]$ there is $\delta'=\delta'(\epsilon,f)$ such that every graph $G=(V,E)$ with $|V| \geq \delta'^{-1}$ 
has an equitable partition $V=V_1 \cup \ldots \cup V_k$ and vertex subsets $W_i
\subset V_i$ such that $|W_i| \geq \delta' |V|$, each pair $(W_i,W_j)$ with
$1 \leq i \leq j \leq k$ is $f(k)$-regular, and all but at most $\epsilon k^2$
pairs $1 \leq i \leq j \leq k$ satisfy $|d(V_i,V_j)-d(W_i,W_j)| \leq \epsilon$.  Furthermore, we may take $\delta'=\frac{1}{8S^2}$, where $S=(\frac{\epsilon^6}{4},s,f)$ is defined as in Lemma \ref{scr2} and $s=2\epsilon^{-1}$.
\end{lemma}
\begin{proof} Let $\alpha=\frac{\epsilon^6}{4}$, $s=2\epsilon^{-1}$, and $\delta'=\frac{1}{8S^2}$,
where $S=S(\alpha,s,f)$ is as in Lemma \ref{scr2}. We apply Lemma \ref{scr2} with
$\alpha$ in place of $\epsilon$. We get an equipartition $P:V=V_1 \cup \ldots
\cup V_k$ with $s \leq k \leq S$ and a strongly $f(k)$-regular partition $\mathcal{K}$ of $V_1 \times
\cdots \times V_k$ into cylinders such that the refinement $Q=Q(\mathcal{K})$
of $P$ has at most $S=S(\alpha, s, f)$ parts and satisfies $q(Q) \leq q(P)+\alpha$. 
Since $|V| \geq \delta'^{-1}=8S^2$, and $P$ is an equipartition into $k \leq S$ parts, the cardinality of each part $V_i \in P$ satisfies $|V_i| \geq \frac{|V|}{2S}$. By Lemma \ref{cylinderclose},  as $2^{1/3}\alpha^{1/6}=\epsilon$, the cylinder
partition $\mathcal{K}$ is $\epsilon$-close to $P$. Hence, at most an
$\epsilon$-fraction of the $k$-tuples $(v_1,\ldots,v_k) \in V_1 \times \cdots
\times V_k$ belong to parts $K=W_1 \times \cdots \times W_k$ of $\mathcal{K}$
that are not $\epsilon$-close to $P$. Since $Q(\mathcal{K})$ has at most $S$
parts, the fraction of $k$-tuples $(v_1,\ldots,v_k) \in V_1 \times \cdots
\times V_k$ that belong to parts $K=W_1 \times \cdots \times W_k$ of
$\mathcal{K}$ with $|W_i|<\frac{1}{4S} |V_i|$ for at least one $i \in [k]$ is at
most $\frac{1}{4S}  \cdot S = \frac{1}{4}$. Therefore, at least a fraction $1-f(k)-\epsilon-\frac{1}{4}>0$
of the $k$-tuples  $(v_1,\ldots,v_k) \in V_1 \times \cdots \times V_k$ belong
to parts $K=W_1 \times \cdots \times W_k$ of $\mathcal{K}$ satisfying $K$ is
strongly $f(k)$-regular,  $|W_i| \geq \frac{1}{4S} |V_i| \geq \delta'|V|$ for $i \in [k]$, and $K$
is $\epsilon$-close to $P$.  Since a positive fraction of the $k$-tuples belong
to such $K$, there is at least one such $K$. This $K$ has the desired
properties. Indeed the number of pairs $1 \leq i \not = j \leq k$ for which $|d(W_i,W_j)-d(V_i,V_j)| >
\epsilon$ is at most $\epsilon k^2$ and hence the number of pairs $1 \leq i \leq j \leq k$ for 
which $|d(W_i,W_j)-d(V_i,V_j)| > \epsilon$ is at most $\epsilon k^2/2+k \leq \epsilon k^2$. This completes the proof. 
\end{proof}

We next discuss how to obtain the induced graph removal lemma from Lemma
\ref{strongeasycor}. This is a bit easier to obtain than in \cite{AFKS} because
Lemma \ref{strongeasycor} has the additional property that the subsets $W_i$ in
the cylinder $K$ are $\epsilon$-regular. We finish this section by giving this
proof and discussing the bound it gives for the induced graph removal lemma,
which is a tower in $h$ of height polynomial in $\epsilon^{-1}$. We first need the
following counting lemma, which is rather standard (see, e.g., Lemma 3.2 in
Alon, Fischer, Krivelevich, and Szegedy \cite{AFKS} for a minor variant). We
omit the proof.

\begin{lemma}
If $H$ is a graph with vertices $1,\ldots,h$, and $G$ is a graph with not
necessarily disjoint vertex subsets $W_1,\ldots,W_h$ such that every pair
$(W_i,W_j)$ with $1 \leq i < j \leq h$ is $\gamma$-regular with $\gamma \leq
\frac{1}{4h}\eta^h$, $|W_i| \geq \gamma^{-1}$ for $1 \leq i \leq h$ and, for $1
\leq i < j \leq k$, $d(W_i,W_j)>\eta$ if $(i,j)$ is an edge of $H$ and
$d(W_i,W_j)<1-\eta$ otherwise, then $G$ contains at least
$\left(\frac{\eta}{4}\right)^{{h \choose 2}}|W_1| \times \cdots \times |W_h|$
induced copies of $H$ with the copy of vertex $i$ in $W_i$.
\end{lemma}

We finish the section with a quantitative version of Theorem \ref{inducedtower}.

\begin{theorem}
There is an absolute constant $c$ such that for any graph $H$ on $h$ vertices
$1,\ldots,h$ and $0 < \epsilon< 1/2$ there is $\delta>0$ with $\delta^{-1} =
t_j(h)$ with $j=O(\epsilon^{-6})$ such that if a graph $G$ on $n$ vertices has
at most $\delta n^h$ induced copies of $H$, then we can add or delete $\epsilon
n^2$ edges of $G$ to obtain an induced $H$-free graph.
\end{theorem}
\begin{proof}
Let $\eta=\frac{\epsilon}{8}$. Let $\delta=\left(\frac{\eta}{4}\right)^{h^2}\delta'^{h}$, where $\delta'=\delta'(\eta,f)$ as in Lemma \ref{easycor2} and $f(k)=\frac{1}{4h}\eta^h$.  If the number of vertices satisfies $|V| < \delta^{-1/h}$, then $\delta|V|^h<1$ and there are no induced copies of $H$, in which case no edges of $G$ need to be modified. We can therefore assume that $|V| \geq \delta^{-1/h}=\left(\frac{\eta}{4}\right)^{-h}\delta'^{-1}$. 
As $|V| \geq \delta'^{-1}$, we can apply Lemma \ref{easycor2} to graph $G=(V,E)$
with $\eta$ in place of $\epsilon$ and $f$ as above. We obtain an
equitable partition $V=V_1 \cup \ldots \cup V_k$ and vertex subsets $W_i
\subset V_i$ such that $|W_i| \geq \delta' |V| \geq \left(\frac{\eta}{4}\right)^{-h}$, the cylinder $W_1 \times
\cdots W_k$ is strongly $f(k)$-regular, and all but at most $\eta k^2$ pairs $1
\leq i \leq j \leq k$ satisfy $|d(W_i,W_j)-d(V_i,V_j)| \leq \eta$.

The above counting lemma shows that if there is any mapping $\phi:[h] \rightarrow [k]$
such that $d(W_{\phi(i)},W_{\phi(j)})>\eta$ for $(i,j)$ an edge of $H$ and
$d(W_{\phi(i)},W_{\phi(j)})<1-\eta$ for $i,j$ distinct and nonadjacent in $H$,
then $G$ contains at least $\left(\frac{\eta}{4}\right)^{{h \choose 2}}|W_1| \times \cdots \times |W_h| \geq \delta n^h$ induced copies of $H$. Hence, no such
mapping $\phi$ exists. That is, for every mapping $\phi:[h] \rightarrow [k]$,
there is an edge $(i,j)$ for which  $d(W_{\phi(i)},W_{\phi(j)}) \leq \eta$ or
distinct $i,j$ that are nonadjacent in $H$ with $d(W_{\phi(i)},W_{\phi(j)})
\geq 1-\eta$.

For each pair $(W_i,W_j)$ for which $d(W_i,W_j)\leq \eta$, delete all edges
between $V_i$ and $V_j$, and for each pair $(W_i,W_j)$ for which $d(W_i,W_j)
\geq 1-\eta$, add all edges between  $V_i$ and $V_j$. Let $G'$ be this
modification of $G$. By the previous paragraph, there are no induced copies of
$H$ in $G'$.

We have left to show that few edge modifications are made in obtaining $G'$
from $G$. If a pair $(W_i,W_j)$ for which edges were modified satisfies
$|d(W_i,W_j)-d(V_i,V_j)| \leq \eta$, then the density between the two sets was
only changed by at most $2\eta$. The number of $ 1\leq i \leq j \leq k$ for which $|d(W_i,W_j)-d(V_i,V_j)| > \eta$ is at most $\eta k^2$. Since $V_1,\ldots,V_k$ is an equipartition into nonempty parts, at most $\eta k^2 \cdot 4\left(\frac{n}{k}\right)^2=4 \eta n^2$ edges are changed between such pairs. In total at most $2\eta{n \choose 2}+ 4\eta n^2 \leq 5\eta n^2 < \epsilon n^2$ edges were changed in order to obtain $G'$ from $G$.

From Lemma \ref{easycor2}, we have $\delta' = \frac{1}{8S^2}$, where $S = S(\alpha, s, f)$ is the function from Lemma \ref{scr2} with $\alpha = \frac{\eta^6}{4}$, $s = 2 \eta^{-1}$ and $f(k) = \frac{1}{4h} \eta^h$. From Lemma \ref{scr2}, $S(\alpha, s, f)$ will be at most a tower in $h$ whose height is proportional to $\eta^{-6}$. Therefore, by the choice of $\eta$ and $\delta$ in the above proof of the induced graph removal lemma, we indeed get the desired tower-type bound. This also completes the proof of Theorem \ref{inducedtower}. \end{proof}

\section{Regular approximation lemma} \label{RALsection}

In this section we show how to derive the regular approximation lemma from
Tao's regularity lemma, as discussed in Subsection \ref{LSsubsection}. The key
lemma is Lemma \ref{keyforral}, which shows how to turn a bipartite graph into
a regular pair by changing some edges according to a weak regular partition.
We use the notation $x = y \pm \epsilon$ to denote the fact that $y-\epsilon
\leq x \leq y+\epsilon$.

It will be helpful to use the Hoeffding-Azuma inequality for concentration of
measure. Say that a random variable $X(\omega)$ on an $n$-dimensional product
space $\Omega = \prod_{i=1}^n \Omega_i$ is
\emph{Lipschitz}\/ if changing $\omega$ in any single coordinate
affects the value of $X(\omega)$ by at most one.  The Hoeffding-Azuma
inequality (see, e.g., \cite{AlSp}) provides concentration for these
distributions.

\begin{theorem}[Hoeffding-Azuma Inequality]
  \label{thm:azuma}
  Let $X$ be a Lipschitz random variable on an $n$-dimensional
  product space.  Then for any $t \geq 0$,
  \begin{displaymath}
    \pr{ | X - \E{X} | > t }
    \leq
    2 \exp\left\{
      -\frac{t^2}{2n}
    \right\}.
  \end{displaymath}
\end{theorem}

For a bipartite graph across parts $A$ and $B$, and partitions
$\mathcal{A}:A=A_1 \cup \ldots \cup A_r$ and $\mathcal{B}:B=B_1 \cup \ldots
\cup B_s$, let $q(A,B)=d^2(A,B)$ and
$q(\mathcal{A}, \mathcal{B}) = \sum_{i,j}\frac{|A_i||B_j|}{|A||B|}d^2(A_i,B_j)$
be the mean square density across the partitions $\mathcal{A}$ and
$\mathcal{B}$.

\begin{lemma}\label{keyforral} Let $0<\delta<1$. Suppose
$A,B$ are disjoint vertex subsets of a graph with $|A| \geq |B| > 8\delta^{-2}$ and $d(A,B)=\eta$. Suppose further that $\mathcal{A}:A=A_1 \cup \ldots \cup A_r$ and $\mathcal{B}:B=B_1 \cup \ldots \cup B_s$ form a weak $\delta$-regular
partition of the pair $(A,B)$, i.e., for all $S \subset A$ and $T \subset B$,
we have
$$\left|\sum_{i=1}^r \sum_{j=1}^s |A_i \cap S||B_j \cap
T|d(A_i,B_j)-d(S,T)|S||T| \right|\leq \delta |A||B|.$$
Then, one can add or remove at most
$\left(\delta+\left(q(\mathcal{A},\mathcal{B})-q(A,B)\right)^{1/2}\right)|A||B|$
edges across $(A,B)$ and thus turn it into a $2\delta^{1/3}$-regular pair
satisfying $d(A, B) = \eta  \pm \delta$.
\end{lemma}
\begin{proof}
Let $\alpha_{i,j}=d(A_i,B_j)-d(A,B)$. If $\alpha_{i,j}\geq 0$, we delete each
of the edges connecting $A_i$ and $B_j$ independently with probability
$\frac{\alpha_{i,j}}{d(A_i,B_j)}$. If $\alpha_{i,j} <0$, we add each of the
nonedges between $A_i$ and $B_j$ with probability
$-\frac{\alpha_{i,j}}{1-d(A_i,B_j)}$. Clearly the expected value of $d(A,B)$
after these modifications is $\eta$. By the Hoeffding-Azuma inequality, the
probability that the new density deviates from $\eta$ by more than $\delta$ is
at most $$2 \exp\left\{
      -\frac{(\delta |A||B|)^2}{2|A|||B|}
    \right\}=2\exp\left\{-\delta^2|A||B|/2\right\}<1/4.$$ Also, the expected
number of edges changed is
\begin{eqnarray*}\sum_{i,j} |\alpha_{i,j}| |A_i||B_j| & =
&\sum_{i,j}|d(A_i,B_j)-d(A,B)||A_i||B_j| = |A||B| \sum_{i, j} |d(A_i, B_j) - d(A,
B)| p_i q_j\\ & \leq & |A||B|\left(\sum_{i, j} \left(d(A_i, B_j) - d(A,
B)\right)^2 p_i q_j\right)^{1/2} = |A||B| \left(q(\mathcal{A}, \mathcal{B}) - q(A,
B)\right)^{1/2},\end{eqnarray*}
where $p_i = |A_i|/|A|$ and $q_j = |B_j|/|B|$ and in the inequality we used the
Cauchy-Schwarz inequality. By the Hoeffding-Azuma inequality, the probability
that the number of edges changed deviates by more than $\delta |A||B|$ from its
expected value is at most
$$2 \exp\left\{
      -\frac{(\delta |A||B|)^2}{2|A||B|}
    \right\}=2\exp\left\{-\delta^2|A||B|/2\right\}<1/4.$$
Consider now two subsets $A' \subset A$ and $B' \subset B$. As $(A, B)$ was
initially weak $\delta$-regular,
the expected value of $e(A',B')$ differs from $\eta |A'||B'|$ by at most
$\delta |A||B|$.
By the Hoeffding-Azuma inequality, we get that the probability that $e(A',B')$
deviates from its expected value by more than $\delta |A||B|$ is at most
$$2 \exp\left\{
      -\frac{(\delta |A||B|)^2}{2|A'||B'|}
    \right\} \leq 2\exp\left\{-\delta^2|A||B|/2\right\}<2\exp\left\{-4|A|\right\}\leq 2^{-|A|-|B|-2},$$
where we use $|A| \geq |B| >8\delta^{-2}$.
As there are $2^{|A|+|B|}$ choices for $(A',B')$, we get that with probability at
least $3/4$, all pairs $(A', B')$ are within $2\delta |A||B|$ edges of having edge
density $\eta$. To recap, we get that with probability at least $1/4$ we made
at most
$\left(\delta+\left(q(\mathcal{A},\mathcal{B})-q(A,B)\right)^{1/2}\right)|A||B|$
edge modifications, $d(A, B) = \eta \pm \delta$ and all subsets $A' \subset
A,B' \subset B$ are within $2\delta |A||B|$ edges from having edge density $\eta$. Hence, 
there is such a choice for these edge modifications, and we claim that this implies that the pair $(A,B)$ is $2\delta^{1/3}$-regular. Indeed,
otherwise there would be $A' \subset A$, $B' \subset B$, with $|A'| \geq 2\delta^{1/3}|A|$, $|B'|\geq 2\delta^{1/3}|B|$, and $|d(A',B')-d(A,B)|>2\delta^{1/3}$, which implies that
$A',B'$ differs by at least $2\delta^{1/3}|A'||B'| \geq (2\delta^{1/3})^3|A||B|=8\delta |A||B|$  edges from
having edge density $d(A,B)$, a contradiction. This completes the proof. 
\end{proof}

We next use Lemma \ref{keyforral} to deduce the regular approximation lemma from Tao's regularity lemma.

\begin{proof}
Let $\delta:\mathbb{N} \rightarrow (0,1)$ be defined by
$\delta(t)=\min\left(\frac{g(t)^{3}}{32t^2},\epsilon/2\right)$. Let
$\epsilon_0=(\epsilon/2)^2$. Let $T_0=T_0(\delta,\epsilon_0,s)$ be the bound on
the number of parts in Tao's regularity lemma and $T=16T_0/\delta(T_0)^2$. If
the number $n$ of vertices of the graph $G$ satisfies $n \leq T$, then we can
partition $G$ into parts of size one, and the desired conclusion is satisfied
in this case. Hence, we may assume $n > T$. By Tao's regularity lemma, the
graph $G$ has an equitable vertex partition $P$ into $t \geq s$ parts and an equitable
vertex refinement $Q$ into at most $T_0$ parts which is weak
$\delta(t)$-regular such that $q(Q) \leq q(P)+\epsilon_0$.

For each pair of parts $(A,B)$ of partition $P$, let $\mathcal{A}$ and
$\mathcal{B}$ denote the partitions of $A$ and $B$ given by partition $Q$.
Since $Q$ is a weak $\delta(t)$-regular partition, and $A$ and $B$ have
cardinality at least $\lfloor \frac{n}{t} \rfloor \geq \frac{n}{2t}$, then the
partitions $\mathcal{A}$ and $\mathcal{B}$ form a weak $4t^2\delta(t)$-regular
partition. Note that $4t^2\delta(t) \leq \frac{g(t)^3}{8}$.

Since $|A|, |B| \geq \frac{n}{2t} > 8/\delta(t)^2$, we may apply Lemma
\ref{keyforral} to the graph between $A$ and $B$. That is, we may change at most
$\left(\delta(t)+\left(q(\mathcal{A},\mathcal{B})-q(A,B)
\right)^{1/2}\right)|A||B|$ edges across $A$ and $B$ and, in so doing, make
$(A,B)$ a $g(t)$-regular pair, where we used
that $g(t)=2\left(\frac{g(t)^3}{8}\right)^{1/3}$. In total, the number of edges
we change to obtain a graph $G'$ which is $g$-regular with respect to partition
$P$ is at most $$\sum_{A,B \in P}
\left(\delta(t)+\left(q(\mathcal{A},\mathcal{B})-q(A,B) \right)^{1/2}\right)|A||B|
\leq (\delta(t)+\epsilon_0^{1/2})n^2 \leq \epsilon n^2,$$
where we used Jensen's inequality for the concave function $h(x)=x^{1/2}$, the
inequality $q(Q) \leq q(P)+\epsilon_0$, and the bounds $\delta(t) \leq
\epsilon/2$, $\epsilon_0=(\epsilon/2)^2$.
To complete the proof, we recall that the number of parts in partition $P$ is at least $s$ and at most $T_0=T_0(\delta,\epsilon_0,s)$.
\end{proof}

\section{Frieze-Kannan weak regularity lemma}\label{weakregsection}

In this section we prove Theorem \ref{FKWRL} which provides a lower bound on
the weak regularity lemma.
For a vertex partition $P:V=V_1 \cup \ldots \cup V_k$ of a graph $G=(V,E)$, let
$$f_P(A,B)=f_P^G(A,B)=e(A,B)-\sum_{1 \leq i, j \leq k}
d(V_i,V_j)|A \cap V_i||B \cap V_j|,$$ which is the difference between the
number of edges between $A$ and $B$ and the expected number of edges based on
the densities across the pairs of parts of the partition and the intersection
sizes of $A$ and $B$ with the parts. We call a partition $P$ of the vertex set
of a graph $G=(V,E)$ {\it weak $\epsilon$-regular}  if it satisfies
$$|f_P(A,B)| \leq \epsilon |V|^2$$ for all $A,B \in V$.
Recall that the weak regularity lemma states that for each $\epsilon>0$ there
is a positive integer $k(\epsilon)$ such that every graph has an equitable weak
$\epsilon$-partition into at most $k(\epsilon)$ parts. Moreover, one may take $k(\epsilon) = 2^{O(\epsilon^{-2})}$. We will prove that the number of parts required in the weak regularity lemma satisfies $k(\epsilon)=2^{\Omega(\epsilon^{-2})}$, thus matching the upper
bound. 

The following simple lemma of Frieze and Kannan (see Lemma 7(a) of
\cite{FrKa1}) shows that the notion of weak regularity is robust. 

\begin{lemma}\label{weakrobust}
If a partition is weak $\epsilon$-regular, then any refinement of it is weak
$2\epsilon$-regular.
\end{lemma}

The robustness of weak regularity described by Lemma \ref{weakrobust} is not
shared by the usual notion of regular partition. For example,
for any fixed $\epsilon>0$ and positive integer $t$, almost surely any
partition into $t$ parts of a uniform random graph on sufficiently many
vertices is $\epsilon$-regular, while any partition of the vertex set into
parts of size $2$ is not $(\epsilon,\delta,\eta)$-regular with $\epsilon=1$,
$\delta=\eta=1/2$. This is because almost surely in any such partition, between
most pairs of parts of size $2$, there will be at least one edge and at least
one nonedge.

What we will actually prove is the stronger result that any weak $\epsilon$-regular partition must have $2^{\Omega(\epsilon^{-2})}$ parts, whether it is equitable or not. The corresponding regularity lemma, which is an immediate corollary of the usual weak regularity lemma, is the following.

\begin{lemma} \label{noneqweak}
For each $\epsilon>0$ there is a positive integer $k^*(\epsilon)$ such that
every graph $G = (V, E)$ has a vertex partition $P$ with at most $k^*(\epsilon)$ parts
which is weak $\epsilon$-regular.
\end{lemma}

In the other direction, the equitable version of the weak regularity lemma also follows from Lemma \ref{noneqweak}. This is because of the robustness property discussed in Lemma \ref{weakrobust} above, that is, any refinement of a weak $\epsilon$-regular partition is a $2\epsilon$-regular partition. By arbitrarily refining a not necessarily equitable partition into an equitable partition (except for a small fraction of vertices, which we distribute evenly amongst the other parts), we get an equitable weak $3 \epsilon$-partition whose number of parts is only a factor polynomial in $\epsilon^{-1}$ larger.

In order to prove the lower bound for weak regularity, we will need to perform some further reductions. We first state a bipartite variant which can easily be shown to be equivalent to Lemma \ref{noneqweak}. For a bipartite graph $G=(U,V,E)$ with $|U|=|V|=n$, partitions
$P_1:U=U_1 \cup \ldots \cup U_k$ and $P_2:V=V_1 \cup \ldots \cup V_{k'}$, and
vertex subsets $A \subset U$ and $B \subset V$, let
$$f_{P_1,P_2}(A,B)=f_{P_1,P_2}^G(A,B)=e(A,B)-\sum_{i=1}^k \sum_{j=1}^{k'}
d(U_i,V_j)|A \cap U_i| |B \cap V_j|.$$ We call the pair of partitions $P_1,P_2$
{\it weak $\epsilon$-regular} with respect to the bipartite graph $G$ if
$$|f_{P_1,P_2}(A,B)| \leq \epsilon n^2$$ for all $A \subset U$ and $B \subset
V$.

\begin{lemma}\label{brl}
For each $\epsilon>0$ there is a positive integer $k'(\epsilon)$ such that
every bipartite graph $G=(U,V,E)$ with parts of equal size has partitions $P_1$
of $U$ and $P_2$ of $V$ each with at most $k'(\epsilon)$ parts
which form a weak $\epsilon$-regular partition.
\end{lemma}

To prove Theorem \ref{FKWRL}, it suffices to show
$k'(\epsilon)=2^{\Omega(\epsilon^{-2})}$. Indeed, this follows from the bound
$k'(\epsilon) \leq k^*(\epsilon/2)$. This inequality follows from first
considering a single weak $\epsilon/2$-regular partition $P$ for the bipartite
graph $G$ into at most $k^*(\epsilon/2)$ parts, and then refining it into a
partition $P'$ with at most $2k^*(\epsilon/2)$ parts based on the intersections
of the parts of $P$ with $U$ and $V$.  By Lemma \ref{weakrobust}, $P'$ is a
weak $\epsilon$-regular partition. Letting $P_1$ be the parts of $P'$ in $U$
and $P_2$ be the parts of $P'$ in $V$, the pair $P_1,P_2$ form a weak
$\epsilon$-regular partition, each with at most $k^*(\epsilon/2)$ parts, so that
$k'(\epsilon) \leq k^*(\epsilon/2)$.

To get a lower bound for the weak regularity lemma, we do not need to show the
other direction of the equivalence between the weak regularity lemma and Lemma
\ref{brl}, that Lemma \ref{brl} implies the weak regularity lemma. However,
this is rather simple, so we sketch it here. From a graph $G$ we can consider
the bipartite double cover of $G$, which is the tensor product of $G$ with
$K_2$. Applying Lemma \ref{brl}, we get a pair $P_1,P_2$ of partitions of
$V(G)$ which form a weak $\epsilon/2$-regular partition with respect to the
bipartite double cover of $G$. Refining the two partitions $P_1,P_2$ of $V(G)$,
we get by Lemma \ref{weakrobust} a weak $\epsilon$-regular partition for $G$,
thus establishing $k^*(\epsilon) \leq k'(\epsilon/2)^2$.

The following technical lemma will allow us to construct a weighted graph
rather than a graph. A similar idea is present in the lower bound construction
of Gowers \cite{Go} for Szemer\'edi's regularity lemma. Let $W$ be a
$[0,1]$-valued $n \times n$ matrix. We view $W$ as a weighted graph with parts
$U$ and $V$, where $U$ and $V$ denote the set of columns and rows,
respectively, of $W$. Let $e_W(A,B)=\sum_{a \in A,b \in B} W(a,b)$ and
$d_W(A,B)=\frac{e_W(A,B)}{|A||B|}$.

\begin{lemma}\label{weighthigh}
Let $M$ be an $n \times n$ matrix with entries in the interval $[0,1]$. Let $G=(U,V,E)$ be a bipartite
random graph with $|U|=|V|=n$ and edges chosen independently given by $M$ and
let $\theta=4n^{-1/2}$. With probability at least $1-e^{-4n}$, we have
$|e_M(A,B)-e_G(A,B)| \leq \theta n^2$ for every pair of sets $A \subset U$, $B
\subset V$.
\end{lemma}
\begin{proof}
Given two vertices $u \in U$ and $v \in V$, let $a(u,v)$ be the random variable
$G(u,v) - M (u,v)$ (where $G$ has been identified with its adjacency matrix).
The mean of $a(u,v)$ is zero for all $u,v$ and the modulus of $a(u,v)$ is at
most $1$. Hence, by the Hoeffding-Azuma inequality (Theorem \ref{thm:azuma}),
given two sets $A \subset U$ and $B \subset V$, the probability that
$$\left | \sum_{(u,v) \in A \times B} a(u,v) \right | \geq \theta n^2$$
is at most $2\exp \left \{ -(\theta n^2)^2/(2|A||B|) \right \} \leq 2\exp \left
\{ -8n \right \}$. Summing over all $A \subset U$ and $B \subset V$, the
probability that there are subsets $A \subset U$ and $B \subset V$ with
$|e_G(A,B)-e_M(A,B)| \geq \theta n^2$ is at most $2^{2n}\cdot 2e^{-8n} \leq
e^{-4n}$.
\end{proof}

For partitions $P_1:U=U_1 \cup \ldots \cup U_k$ and $P_2:V=V_1 \cup \ldots \cup
V_{k'}$, let
$$f_{P_1,P_2}(A,B)=e_W(A,B)-\sum_{i=1}^k\sum_{j=1}^{k'}d_W(U_i,V_j)|U_i \cap
A||V_j \cap B|.$$  We say that partitions $P_1,P_2$ {\it form a weak
$\epsilon$-regular partition of $W$} if $|f_{P_1,P_2}(A,B)| \leq \epsilon n^2$
for all subsets $A \subset U$ and $B \subset V$.

\begin{corollary}\label{weightunweight}
Suppose $W=(U,V,E)$ is an edge-weighted graph with weights in $[0,1]$ and
$|U|=|V|=n$. Let $G=(U,V,E)$ be a bipartite random graph with $|U|=|V|=n$ and
edges chosen independently given by $W$ and let $\theta=4n^{-1/2}$. With
probability at least $1-e^{-4n}$, every pair of partitions $P_1:U=U_1 \cup
\ldots \cup U_k$ and $P_2:V=V_1 \cup \ldots \cup V_{k'}$ which form a weak
$\epsilon$-partition for $G$ also form a weak $(\epsilon+2\theta)$-regular
partition for $W$.
\end{corollary}
\begin{proof}
By Lemma \ref{weighthigh}, with probability at least $1-e^{-4n}$, we have
$|e_G(A,B)-e_W(A,B)| \leq \theta n^2$ for every pair of vertex subsets $A
\subset U$ and $B \subset V$. Suppose this indeed holds.
For graph $G$, we have
$$f_{P_1,P_2}^G(A,B)=e_{G}(A,B)-\sum_{i=1}^k\sum_{j=1}^{k'}d_G(U_i,V_j)|A \cap
U_i||B \cap V_j|.$$
The first term is within $\theta n^2$ of $e_W(A,B)$. The second term is the
average of $e_{G}(A',B')$ over all subsets $A' \subset U$ and $B' \subset V$
with $|A' \cap U_i|=|A \cap U_i|$ for all $i$ and $|B' \cap V_j|=|B \cap V_j|$
for all $j$, and hence is within $\theta n^2$ of the corresponding average of
$e_W(A',B')$ over all of the same pairs $(A',B')$. By the triangle inequality,
for $W$, we get that for all $A \subset U$ and $B \subset V$, we have
$|f_{P_1,P_2}^W(A,B)| \leq |f_{P_1,P_2}^G(A,B)|+2\theta n^2 \leq
(\epsilon+2\theta)n^2$. Hence, $P_1,P_2$ also form a weak
$(\epsilon+2\theta)$-regular partition for $W$.
\end{proof}

From Corollary \ref{weightunweight} and the previous remarks, to obtain the
desired lower bound in Theorem \ref{lbweakmain} on the number of parts in the
weak regularity lemma it suffices to prove a lower bound of the form
$2^{\Omega(\epsilon^{-2})}$ on the number of parts $k_0(\epsilon)$ in the
following weak regularity lemma for weighted bipartite graphs.

\begin{lemma}\label{brlweight}
For each $\epsilon>0$ there is a positive integer $k_0(\epsilon)$ such that
every edge-weighted bipartite graph $G=(U,V,E)$ with weights in $[0,1]$ and
parts of equal size has partitions $P_1$ of $U$ and $P_2$ of $V$ each with at
most $k_0(\epsilon)$ parts  which form a weak $\epsilon$-regular partition.
\end{lemma}

Lemma \ref{brlweight} is also known as the weak matrix regularity lemma. This
is because it provides, for any $n \times n$ matrix with entries in $[0,1]$,
partitions of the rows and columns into a bounded number of parts, such that
the sum of the entries in any submatrix (which is the product of a set of rows
and columns) is within $\epsilon n^2$ of what is expected based on the
intersections with the parts and the density between the parts.

Our goal is to construct a bipartite graph $G$ with edge weights in $[0,1]$
which provides a lower bound of the form $k_0(\epsilon) =
2^{\Omega(\epsilon^{-2})}$. Suppose $0<\epsilon \leq 2^{-50}$. Consider the
following weighted bipartite graph $G$. The graph has parts $U$ and $V$ each of
order $n=2^{2^{-45}\epsilon^{-2}}$. Let $r=2^{-40}\epsilon^{-2}$ and
$\alpha=2^{14}\epsilon$. Consider, for $1 \leq i \leq r$, equitable partitions
$U=U^i_0 \cup U^i_1$ and $V=V^i_0 \cup V^i_1$ picked uniformly and
independently at random. For vertices $u \in U$ and $v \in V$, let $s(u,v)$ be
the number of $i \in [r]$ for which there is $j \in \{0,1\}$ such that $u \in
U^i_j$ and $v \in V^i_j$, and $t(u,v)$ be the number of $i \in [r]$ for which
there is $j \in \{0,1\}$ such that $u \in U^i_j$ and $v \in V^i_{1-j}$, so that
$s(u,v)+t(u,v)=r$. Let $W(u,v)=\frac{1}{2}+(s(u,v)-t(u,v))\alpha$. We define
the weight $w(u,v)$ between $u$ and $v$ as follows. If $0 \leq W(u,v) \leq 1$,
then $w(u,v)=W(u,v)$, if $W(u,v)<0$, then $w(u,v)=0$, and if $W(u,v)>1$, then
$w(u,v)=1$.

Call a pair $(u,v) \in U \times V$ {\it extreme} if $|W(u,v)-1/2| >1/4$, and a
vertex $u \in U$ {\it nice} if it is in at most $n/8$ pairs $(u,v)$ with $v \in
V$ which are extreme.

\begin{lemma}\label{lem34a}
With probability at least $3/4$, all but at most $e^{-100}n$ vertices of $U$
are nice.
\end{lemma}
\begin{proof}
Fix a pair $(u,v) \in U \times V$. The event $(u,v)$ is extreme is the same as
$|s(u,v)-t(u,v)|\alpha>1/4$, or equivalently that
$|s(u,v)-r/2|>\frac{1}{8\alpha}$. The number $s(u,v)$ is a sum of $r$
independent variables, with values $0$ or $1$ each with probability $1/2$, and
hence follows a binomial distribution with mean $r/2$. By Chernoff's bound
(\ref{chernoffest}), the probability that $|s(u,v)-r/2| > \frac{1}{8\alpha}$ is
less than $2e^{-2(1/(8\alpha))^2/r}=2e^{-2^{7}}$. Hence, by linearity of
expectation, the expected number of extreme pairs $(u,v) \in U \times V$ is
less than $2e^{-2^{7}}n^2$. Therefore, by Markov's inequality, the probability
that there are at least $8e^{-2^{7}}n^2$ extreme pairs $(u,v)$ is less than
$1/4$. Since any nice vertex is contained in at most $n/8$ extreme pairs, we
see that with probability at least $3/4$, all but at most $64 e^{-2^7} n \leq
e^{-100} n$ vertices in $U$ are nice.
\end{proof}

For $h \in [r]$, we let $s_h(u,v)$ denote the number of $i \in [r] \setminus
\{h\}$ for which there is $j \in \{0,1\}$ such that $u \in U^i_j$ and $v \in
V^i_j$, and $t_h(u,v)$ be the number of $i \in [r] \setminus \{h\}$ for which
there is $j \in \{0,1\}$ such that $u \in U^i_j$ and $v \in V^i_{1-j}$, so that
$s_h(u,v)+t_h(u,v)=r-1$. Let $W_h(u,v)=\frac{1}{2}+(s_h(u,v)-t_h(u,v))\alpha$.
As above, we define the weight $w_h(u,v)$ by  $w_h(u,v)=W_h(u,v)$ if $0 \leq
W_h(u,v) \leq 1$, $w_h(u,v)=0$ if $W_h(u,v) < 0$, and $w_h(u,v)=1$ if $W_h(u,v)
> 1$.

\begin{lemma}\label{nicevery}
Suppose $u \in U^h_j$ and we are given $|w_h(u,v)-1/2| \leq 1/2-\alpha$ for at
least $\frac{7}{8}n$ vertices $v \in V$, and we do not yet know the partition
$V=V^h_0 \cup V^h_1$. Then the probability that
$d_w(u,V^h_j)-d_w(u,V^h_{1-j}) \geq \alpha/2$ is at least $1-\frac{1}{4rn}$.
\end{lemma}
\begin{proof}
Consider the event $E$ that
$$\sum_{v \in V^h_{1-j}} w_h(u,v)-\sum_{v \in V^h_j} w_h(u,v) \geq \alpha
n/4.$$
Note that the expected value of the left hand side is $0$. Recall that a
hypergeometric distribution is at least as concentrated as the sum of
independent random variables with the same values (for a proof, see Section 6
of \cite{Ho}). By the Hoeffding-Azuma inequality (Theorem \ref{thm:azuma}), the
probability of event $E$ is at most
$$2\exp\left \{-\frac{(\alpha n/8)^2}{2n}\right \} = 2\exp \left
\{-2^{-7}\alpha^2n\right \} \leq \frac{1}{4rn},$$
where in the last inequality we use $0<\epsilon \leq 2^{-50}$,
$r=2^{-40}\epsilon^{-2}$, $n=2^{2^{-45}\epsilon^{-2}}$, and
$\alpha=2^{14}\epsilon$.

For a fixed $u$, if $|w_h(u,v) - \frac{1}{2}| \leq \frac{1}{2} - \alpha$, then
$w(u,v) = w_h(u,v) + \alpha$ if $v$ is in $V_j^h$ and $w(u,v) = w_h(u,v) -
\alpha$ if $v$ is in $V_{1-j}^h$. For all $v$, $w(u,v)$ is within $\alpha$ of
$w_h(u,v)$. Therefore, letting $w(u, S) = \sum_{s \in S} w(u, s)$, we see, since
$|w_h(u,v) - \frac{1}{2}| \leq \frac{1}{2} - \alpha$ for at least
$\frac{7}{8}n$ vertices of $v$, that
\[w(u, V_j^h) - w(u,V_{1-j}^h) \geq \frac{7}{8} \alpha n - \frac{1}{8}\alpha n
+ w_h(u, V^h_j) - w_h(u, V_{1-j}^h) \geq \frac{3}{4} \alpha n - \frac{1}{4}
\alpha n = \frac{\alpha}{2} n.\]
The result follows.
\end{proof}

Call a nice vertex $u \in U$ {\it very nice} if for each $h \in [r]$ and $j \in
\{0,1\}$ with $u \in U^h_j$, $$d(u,V^h_j)-d(u,V^h_{1-j}) \geq \alpha/2.$$

\begin{corollary}\label{cor34b}
With probability at least $3/4$, every nice vertex $u$ is very nice.
\end{corollary}
\begin{proof}
Given $u$ is nice, then for each $h \in [r]$, we must have $|W_h(u,v)-1/2| \leq
1/2-\alpha$ for all but at most $\frac{7}{8}n$ vertices $v \in V$. The
probability that there is a vertex which is nice but not very nice is by Lemma
\ref{nicevery} at most $rn \cdot \frac{1}{4rn} \leq 1/4$, which completes the
proof.
\end{proof}

From Lemma \ref{lem34a} and Corollary \ref{cor34b}, we have the following
corollary.

\begin{corollary}\label{cor12c}
With probability at least $1/2$, the graph $G$ has the following properties.
\begin{itemize}
\item The number of vertices in $U$ which are not nice is at most $e^{-100}n$.
\item Every nice vertex is very nice.
\end{itemize}
\end{corollary}

Consider the random bipartite graph $B=B(n,r)$ with vertex parts $U$ and $[r]$
where $i \in [r]$ is adjacent to $u \in U$ if $u \in U_0^i$. By Lemma
\ref{firstlemma1} with $\mu=1/4$, as $r \geq 32\log n$, we have the following
proposition.

\begin{corollary} \label{prop1n2}
With probability at least $3/4$, for each pair $u,u' \in U$, the number of $i$
for which $u$ and $u'$ both belong to $U_j^i$ for some $j \in \{0,1\}$ is less
than $\frac{3}{4}r$.
\end{corollary}

Hence, with probability at least $1/4$, the graph $G$ satisfies the properties
in Corollaries \ref{cor12c} and \ref{prop1n2}. Fix such a graph $G$.

\begin{theorem}\label{lowbo}
No partitions $P_1:U_1 \cup \ldots \cup U_k$ of $U$ and $P_2:V_1 \cup \ldots
\cup V_{k'}$ of $V$ with $k \leq n/2$ form a weak $\epsilon$-regular partition.
As $n/2 \geq 2^{2^{-46}\epsilon^{-2}}$, we therefore have $k_0(\epsilon) >
2^{2^{-46}\epsilon^{-2}}$ for $0 < \epsilon \leq 2^{-50}$.
\end{theorem}

Theorem \ref{lowbo} gives a lower bound on the number $k_0(\epsilon)$ of parts
for the weak matrix regularity lemma (Lemma \ref{brlweight}) with approximation
$\epsilon$. Before we prove this theorem, we remark that it has no restriction
on the number of parts of partition $P_2$, and further shows that $P_1$ has to
be almost the finest partition (partition into singletons) to obtain a pair of
partitions which are weak $\epsilon$-regular.

\begin{proof} Suppose for contradiction that the partitions $P_1$ and $P_2$ are
weak $\epsilon$-regular. That is,
$|f_{P_1,P_2}(A,B)| \leq \epsilon n^2$ for all subsets $A \subset U$ and $B
\subset V$.

Fix for now $U_t$ with $|U_t| \geq 2$. Call the pair $(i,t) \in [r] \times [k]$
{\it useful} if $|U_t \cap U_j^i| \geq |U_t|/32$ for $j \in \{0,1\}$. Let
$M_{t}$ be the number of $i \in [r]$ for which the pair $(i,t)$ is useful.
The sum $$S_t=\sum_{i \in [r]}|U_t \cap U_0^i||U_t \cap U_{1}^i| \leq
r|U_t|^2/32+M_t|U_t|^2/4$$ is precisely the number of triples $u,u',i$ with
$u,u'$ distinct elements of $U_t$ and $i \in [r]$ for which $u$ and $u'$ are
not in the same set in the partition $U=U_0^i \cup U_1^i$. By Corollary
\ref{prop1n2}, the sum $S_t$ is at least $\frac{1}{4}r{|U_t| \choose 2} \geq
|U_t|^2 r/16$. Hence, $$r|U_t|^2/32+M_t|U_t|^2/4 \geq S_t \geq |U_t|^2 r /16.$$
We thus have $M_t \geq r/8$.

Since $M_t \geq r/8$ for each $t$ for which $|U_t| \geq 2$ and there are at
most $k$ parts in partition $P_1$ of order $1$, there is an $i$ for which
partition $i$ satisfies that at least $(n-k)/8 \geq n/16$ vertices $u \in U$
are in $U_t$ with the pair $(i,t)$ useful. Fix such an $i$. For each $t$ for
which $(i,t)$ is useful and all but at most $|U_t|/64$ vertices in $U_t$ are
nice, for $j \in \{0,1\}$, let $U_{t,j}$ be a subset of $U_t \cap U_j^i$ of
cardinality exactly $\lceil |U_t|/64 \rceil$, and $A_j$ denote the union of all
such $U_{t,j}$. Recall from Corollary \ref{cor12c} that there are at most
$e^{-100}n$ vertices in $U$ which are not nice. Hence, there are at most $64
\cdot e^{-100}n$ vertices in $U$ which belong to a $U_t$ for which the pair
$(i,t)$ is useful but there are at least $|U_t|/64$ vertices in $U_t$ which are
not nice. Thus, the number of vertices in $U$ which belong to a $U_t$ for which
$(i,t)$ is useful and there are at most $|U_t|/64$ vertices in $U_t$ which are
not nice is at least
$$\frac{n}{16}- 64e^{-100}n > \frac{n}{32}.$$
We thus have $|A_0|=|A_1| > \frac{n}{32}/64=2^{-11}n$.

Note that, by construction, we have for each $t \in [k]$, $\ell \in [k']$ and
$T \subset V$,
$$|A_0 \cap U_t||T \cap V_{\ell}|d(U_t,V_{\ell})=|A_1 \cap U_t||T \cap
V_{\ell}|d(U_t,V_{\ell}).$$
Thus, if the partitions $P_1,P_2$ form a weak $\epsilon$-regular partition, we
would have to have
\begin{equation}\label{comparej}|e(A_0,V_j^i)-e(A_1,V_j^i)| \leq 2\epsilon
n^2.\end{equation}
for $j \in \{0,1\}$. However, as each $u \in A_0$ is in $U_0^i$ and is very
nice, we have $$d(A_0,V_0^i)-d(A_0,V_1^i) \geq \alpha/2.$$ Since $|A_0| >
2^{-11}n$ and $|V_j^i|=n/2$ for $j \in \{0,1\}$, we have
$$e(A_0,V_0^i)-e(A_0,V_1^i) > 2^{-13}\alpha n^2.$$
Similarly, $$e(A_1,V_1^i)-e(A_1,V_0^i) > 2^{-13}\alpha n^2.$$
Adding the previous two inequalities, we have \begin{equation}\label{lhsneed}
e(A_0,V_0^i)-e(A_1,V_0^i)+e(A_1,V_1^i)-e(A_0,V_1^i) >2^{-12}\alpha
n^2.\end{equation} But, by (\ref{comparej}) for $j\in \{0,1\}$, the left hand
side of (\ref{lhsneed}) is at most $4\epsilon n^2$ in modulus, contradicting
the above inequality and $\alpha=2^{14}\epsilon$. This completes the proof.
\end{proof}

{\bf Remark:} While Theorem \ref{lowbo} provides for each $\epsilon$ only one
graph which requires at least $2^{\Omega(\epsilon^{-2})}$ parts in any weak
$\epsilon$-regular pair of partitions, it is a simple exercise to modify the
proof to show that all blow-ups of $G$ also satisfy this property, thus
obtaining an infinite family of such graphs. For a graph $G$ on $n$ vertices
and a positive integer $t$, the {\it blow-up} $G(t)$ of $G$ is the graph on
$nt$ vertices obtained by replacing each vertex $u$ by an independent set
$I_u$, and a vertex in $I_u$ is adjacent to a vertex in $I_v$  if and only if
$u$ and $v$ are adjacent.

\section{Concluding remarks}
\begin{itemize}
\item{\bf Weak regularity lemmas without irregular pairs}

While proving his famous theorem on arithmetic progressions in dense subsets of
the integers, Szemer\'edi \cite{Sz1} actually developed a regularity lemma
which is weaker than what is now commonly known as Szemer\'edi's regularity
lemma \cite{Sz}. The following version is a strengthening of the original
version, as it guarantees that all pairs, instead of all but an
$\epsilon$-fraction of pairs, under consideration are $\epsilon$-regular. The
key extra ingredient is an application of Lemma \ref{partfoura} to redistribute
the small fraction of vertices which are not in regular sets.

\begin{lemma}\label{strongoriginal}
For each $0<\epsilon<1/2$ there are integers $k=k(\epsilon)$ and
$K=K(\epsilon)$ such that the following holds. For every graph $G=(V,E)$, there
is an equitable partition $V=V_1 \cup \ldots \cup V_t$ into at most $k$ parts
such that for each $i$, $1 \leq i \leq t$, there is a partition $V=V_{i1} \cup
\ldots \cup V_{ij_i}$, with $j_i \leq K$,
such that for all $1 \leq i \leq t$ and $1 \leq j \leq j_i$ the pair
$(V_i,V_{ij})$ is $\epsilon$-regular. Furthermore,
$k(\epsilon)=2^{\epsilon^{-C}}$ and $K(\epsilon) = O(\epsilon^{-1})$, where $C$
is an absolute constant.
\end{lemma}

Szemer\'edi \cite{Sz1} originally gave a triple exponential upper bound on the
number of parts in the original regularity lemma, whereas it is now known (see
\cite{RoSc}) that the correct bound is single exponential. Through iterative
applications, the original regularity lemma was used by Ruzsa and Szemer\'edi
\cite{RuSz} to resolve the $(6,3)$-problem, and by Szemer\'edi \cite{Sz73} to
establish the upper bound on the Ramsey-Tur\'an problem for $K_4$. It is a
relatively simple exercise to show that Szemer\'edi's original regularity lemma
implies the Frieze-Kannan weak regularity lemma, but with a bound that is one
exponential worse than the tight bound established in the previous section.
This can be accomplished by showing that the common refinement of the
partitions in the original regularity lemma satisfies the Frieze-Kannan weak
regularity lemma.

There are a number of notable properties of Lemma \ref{strongoriginal}. First,
all pairs $(V_i,V_{ij})$ under consideration in Lemma \ref{strongoriginal} are
regular. In contrast, Theorem \ref{exceptionalpairs} shows that we must allow
for many irregular pairs in Szemer\'edi's regularity lemma.  Second, the bounds
are much better than in Szemer\'edi's regularity lemma. The bounds on the
number of parts is only single-exponential, instead of the tower-type bound
which appears in the standard regularity lemma. Furthermore, each of the
partitions of $V$ have at most $K(\epsilon)=O(\epsilon^{-1})$ parts. Indeed,
this can be established by first proving any bound on $K(\epsilon)$, and then
using the following additive property of regularity to combine parts. Namely,
applying Lemma \ref{strongoriginal} with any bound on $K(\epsilon)$ and with
$\epsilon^2 /4$ in place of $\epsilon$, and, for each $i$, partitioning $V$
into $O(\epsilon^{-1})$ parts, each part consisting of the union of parts
$V_{ij}$ for which $d(V_i,V_{ij})$ lies in an interval of length at most
$\epsilon/2$, the following lemma shows that $V_i$ together with each part of
the new partition forms an  $\epsilon$-regular pair.

\begin{lemma}\label{alepsi}
Let $0<\epsilon<1$ and $\alpha=\epsilon^2/4$. Suppose $A,B_1,\ldots,B_r$ are
disjoint sets satisfying $(A,B_i)$ is $\alpha$-regular for $1 \leq i \leq r$
and $|d(A,B_i)-d(A,B_j)| \leq \epsilon/2$ for $1 \leq i,j \leq r$. Then,
letting $B=B_1 \cup \ldots \cup B_r$, the pair $(A,B)$ is $\epsilon$-regular.
\end{lemma}

To save space, we omit the details of how to prove Lemmas \ref{strongoriginal}
and \ref{alepsi}.

Another interesting consequence of Lemma \ref{strongoriginal} is that it
implies that every graph on $n$ vertices has an $\epsilon$-regular pair in
which one part is of size $\Omega(\epsilon n)$ and the other part has size
$2^{-\epsilon^{-O(1)}}n$. It is well known (see \cite{KoSi}), that one can find
an $\epsilon$-regular pair in which both parts have size
$2^{-\epsilon^{-O(1)}}n$. In some applications, having one regular pair is
sufficient  (see, e.g., \cite{Ea}, \cite{Ha}, \cite{KoSi}), and one obtains
much better bounds than by applying Szemer\'edi's regularity lemma. In the
other direction, there are graphs (see Theorem 1.4 in \cite{PRR} for a tight
result) for which any $\epsilon$-regular pair has a part of size
$2^{-\epsilon^{-\Omega(1)}}n$. We can nevertheless guarantee that one part is
of size $\Omega(\epsilon n)$. It seems likely that having such a large part in
a regular pair could be useful in some applications.

By iterative application of Lemma \ref{strongoriginal}, one can also obtain a
version of the Duke-Lefmann-R\"odl lemma such that all cylinders in the
partition are $\epsilon$-regular. In fact, using Lemma \ref{epsdeltaone}, one
can further guarantee that all cylinders in the partition are strongly
$\epsilon$-regular, and the bound is still of constant-tower height.

\item{\bf A part irregular with no other parts}

In the proof of Theorem \ref{exceptionalpairs}, we found a graph $G$ such that for any partition into $k$ parts there are at least $\theta k$ parts $V_i$ for which there are at
least $\theta^{-1}\eta k$ pairs $(V_i,V_j)$ which are not
$(\epsilon,\delta)$-regular, where $0<\theta<1$ is a fixed constant. Is it
possible to improve this result so that all parts are in $\eta k$ irregular
pairs? The answer is no, as Lemma \ref{strongoriginal} implies that any graph
has an equitable partition into only $2^{\epsilon^{-O(1)}}$ parts in which one
part is $\epsilon$-regular with all the other parts.

\item {\bf A single regular subset}

It would be interesting to determine tight bounds for the size of the largest $\epsilon$-regular subset which can be found in every graph. In Lemma \ref{oneepsilonregularsubset}, we showed that every graph $G = (V, E)$ must contain an $\epsilon$-regular subset $U$ of size at least $2^{-\epsilon^{-(10/\epsilon)^4}} |V|$. On the other hand, a result of Peng, R\"odl and Ruci\'nski \cite{PRR} implies that there are graphs $G = (V, E)$ which contain no $\epsilon$-regular subset of size $\epsilon^{c\epsilon^{-1}} |V|$. We conjecture that the correct bound is single exponential, though the power may be polynomial in $\epsilon^{-1}$. 

\item {\bf Equitable partitions and not necessarily equitable partitions}

In the regularity lemmas considered in this paper, we often assume the
partitions we consider are equitable partitions. It is not difficult to see
that this assumption is non-essential and the bounds do not change much.
Indeed, consider for example the following variant of Szemer\'edi's regularity
lemma.

\begin{lemma}\label{szemvar}
For each $\epsilon,\delta,\eta>0$ there is a positive integer
$M=M(\epsilon,\delta,\eta)$ for which the following holds. Every graph
$G=(V,E)$ has a vertex partition $V=V_1 \cup \ldots \cup V_k$ with $k \leq M$
such that the sum of $|V_i||V_j|$ over all pairs $(V_i,V_j)$ which are not
$(\epsilon,\delta)$-regular is at most $\eta |V|^2$.
\end{lemma}

The key difference between this version of Szemer\'edi's regularity lemma and
the usual statement is that the parts of the partition are not necessarily of
equal order, and we measure the regularity of the partition by the sum of the
products of the sizes of the pairs of parts that are
$(\epsilon,\delta)$-regular. Szemer\'edi's regularity lemma clearly implies
Lemma \ref{szemvar}, as
it further specifies that the partition is an equitable partition, and the
condition that the sum of $|V_i||V_j|$ over all pairs $(V_i,V_j)$ which are not
$(\epsilon,\delta)$-regular is at most $\eta |V|^2$ is then essentially the
same as saying that the number of pairs $(V_i,V_j)$ which are not
$(\epsilon,\delta)$-regular is at most $\eta k^2$. To see that Lemma
\ref{szemvar}
implies Szemer\'edi's regularity lemma with similar bounds, first apply Lemma
\ref{szemvar}, and then randomly refine each part $V_i=V_{i0} \cup V_{i1} \cup
\ldots \cup V_{ij_i}$ into parts of size $m=\frac{1}{100}\alpha^2 |V|/M$, where
$\alpha=\min(\epsilon,\delta,\eta)$, except possibly $V_{i0}$, which can have
size less than $m$ as not necessarily each $V_i$ has cardinality a multiple of
$m$. The total number of remaining vertices, those in $V_0= \bigcup_{i=1}^k
V_{i0}$, is less than $km\leq \frac{1}{100}\alpha^2|V|$, as there are less than
$m$ remaining vertices from each of the $k$ parts $V_i$. Redistributing the
vertices of $V_0$ evenly among the parts of size $m$, we get a new equitable
partition where each part has size between $m$ and at most
$(1+\frac{1}{50}\alpha^2)m$. It is not hard to show using an additional
application of the Frieze-Kannan weak regularity lemma, that because we
randomly refined each part, almost surely for all pairs $(V_h,V_i)$ which are
$(\epsilon,\delta)$-regular, all pairs $(V_{ha},V_{ij})$  are
$(2\epsilon,2\delta)$-regular. That is, the regularity between pairs of parts
is almost surely inherited between large vertex subsets. It then easily follows
that the equitable partition is similar both in the number of parts and the
degree of regularity to the original partition from Lemma \ref{szemvar}.

Because of this equivalence, we get similar lower bounds for regularity lemmas
where the partitions are not necessarily equitable partitions. For example, for
Lemma \ref{szemvar},  for some absolute constants $\epsilon,\delta>0$, we get a
bound on $M(\epsilon,\delta,\eta)$ which is a tower of $2$s of height
$\Omega(\eta^{-1})$. Similarly, in Theorem \ref{stronglow} and Corollary
\ref{strongcor} giving lower bounds on the strong regularity lemma, the
assumption that $\mathcal{B}$ is an equitable partition is not needed.

Finally, as we have already noted in Section \ref{weakregsection}, it is much easier to show that for the Frieze-Kannan weak regularity lemma we do not need to assume that the partition is equitable. This is a simple consequence of the robustness of weak regularity under refinement.

\item {\bf Irregular pairs and half graphs}

A (generalized) half-graph has vertex set $A \cup B$ with $2n$ vertices
$A=\{a_1,\ldots,a_n\}$ and $B=\{b_1,\ldots,b_n\}$, in which $(a_i,b_j)$ is an
edge if and only if $i \leq j$ (but the edges within $A$ and $B$ could be
arbitrary). As mentioned in the introduction, any partition of a large
half-graph into a constant number of parts has some irregular pairs. Malliaris
and Shelah \cite{MSh} recently showed that for each fixed $k$, every graph on
$n$ vertices with no induced subgraph which is a half-graph on $2k$ vertices
has an $\epsilon$-regular partition with no irregular parts and the number of
parts is at most $\epsilon^{-c_k}$, where $c_k$ is single-exponential in $k$.
This shows that any construction forcing irregular pairs in the regularity
lemma, like that given in the proof of Theorem \ref{exceptionalpairs}, must
contain large half-graphs, of size double-logarithmic in the number of
vertices.

\item {\bf Distinct regular approximations}

The notion of $f$-regularity, which appears in the regular approximation lemma,
has since appeared elsewhere. Alon, Shapira, and Stav \cite{AlShSt} investigate
the question of determining if a graph $G=(V,E)$ can have distinct regular
partitions. Two equitable partitions $\mathcal{U}:U=U_1 \cup \ldots \cup U_k$
and $\mathcal{V}:V=V_1 \cup \ldots \cup V_k$ into $k$ parts are said to be {\it
$\epsilon$-isomorphic} if there is a permutation $\pi:[k] \rightarrow [k]$ such
that for all but at most $\epsilon{k \choose 2}$ pairs $1 \leq i<j \leq k$,
$|d(U_i,U_j)-d(V_{\pi(i)},V_{\pi(j)})| \leq \epsilon$. They prove that for some
$f(k)=\Theta\left(\frac{\log^{1/3} k}{k^{1/3}}\right)$ and infinitely many $k$,
and for every $n>n(k)$, there is a graph on $n$ vertices with two $f$-regular
partitions which are not $1/4$-isomorphic.
On the other hand,  they show that if $f(k) \leq \min(1/k^2,\epsilon/2)$, then
any two equitable partitions of $G$ into $k$ parts which  are each $f$-regular
must be $\epsilon$-isomorphic.

\item {\bf Multicolor and directed regularity and removal lemmas}

The proof of Szemer\'edi's regularity lemma has been extended to give
multicolor (see \cite{KoSi}) and directed (see \cite{AlSh04}) extensions. These
imply multicolor and directed generalizations of the graph removal lemma. As
discussed in \cite{Fo}, the new proof of the graph removal lemma with a
logarithmic tower height extends with a similar bound to these versions as
well. Axenovich and Martin \cite{AxMa} recently extended the strong regularity
lemma in a similar fashion to give multicolor and directed versions, in order
to establish extensions of the induced graph removal lemma.  Our proof of the
induced graph removal lemma with a tower-type bound similarly extends to give
multicolor and directed versions.

\item {\bf On proofs of regularity lemmas}

As noted by Gowers, the constructions in \cite{Go} not only show that the bound
in Szemer\'edi's regularity lemma is necessarily large, but that, in some
sense, the proof is necessary. Any proof must follow a long sequence of
successively finer partitions, each exponentially larger than the previous one.
While this notion is hard to make precise, it should be clear to anyone who has studied
the proof of the regularity lemma and the lower bound construction of Gowers.
Theorem \ref{exceptionalpairs} adds further weight to this conviction. Furthermore, the proof of Theorem \ref{stronglow} shows that any proof of the strong regularity lemma requires a long sequence of partitions, each of tower-type larger than the previous partition. That is, the
iterated use of Szemer\'edi's regularity lemma is required in any proof of the
strong regularity lemma.
\end{itemize}

\vspace{0.1cm} \noindent {\bf Acknowledgment.} \, We would like to thank Noga Alon for helpful comments.

\vspace{0.1cm} \noindent {\bf Note added.}\, After this paper was
written we learned that a variant of Theorem \ref{stronglow} was also proved,
independently and simultaneously, by Kalyanasundaram and Shapira. In the
situation of Corollary \ref{stroncor}, their theorem gives a lower bound of
wowzer-type in $\sqrt{\log \epsilon^{-1}}$ for the strong regularity lemma.

\end{document}